\documentclass[11pt]{article}

\usepackage[usenames,dvipsnames,svgnames,table]{xcolor}
\usepackage[titletoc,title]{appendix}
\usepackage{amsfonts}
\usepackage{graphicx}
\usepackage{hyperref}
\usepackage{amsmath}
\usepackage{amssymb}
\usepackage{framed}
\usepackage{amsthm}
\usepackage{bbm}
\usepackage{comment}
\usepackage{cancel}
\usepackage{float}
\usepackage{listings}

\newtheorem{theorem}{Theorem}[section]
\newtheorem{proposition}[theorem]{Proposition}
\newtheorem{lemma}[theorem]{Lemma}
\newtheorem{claim}[theorem]{Claim}

\newtheorem{question}[theorem]{Question}

\newtheorem{conjecture}[theorem]{Conjecture}

\theoremstyle{definition}
\newtheorem{remark}[theorem]{Remark}

\theoremstyle{definition}
\newtheorem*{remark*}{Remark}

\newtheorem*{proposition*}{Proposition}

\theoremstyle{definition}
\newtheorem{example}[theorem]{Example}

\newcommand{\mrm}[1]{\mathrm{#1}}
\newcommand{\skipline}{$\phantom{}$}

\newcommand{\integers}{\mathbb{Z}}
\newcommand{\pintegers}{\mathbb{N}}
\newcommand{\reals}{\mathbb{R}}
\newcommand{\spm}{\left\{-1,1\right\}}

\newcommand{\seq}{\subseteq}
\newcommand{\sm}{\setminus}

\newcommand{\defeq}{\mathrel{\mathop:}=}
\newcommand{\eqdef}{=\mathrel{\mathop:}}
\newcommand{\re}{\mrm{Re}}
\newcommand{\one}{\mathbbm{1}}
\newcommand{\eps}{\epsilon}
\newcommand{\sgn}{\mrm{sgn}}

\newcommand{\pr}{\Pr}
\newcommand{\given}[2]{#1\,\middle|\, #2}
\DeclareMathOperator*{\be}{\mathbb{E}}
\DeclareMathOperator*{\var}{\mrm{Var}}

\newcommand{\andd}{\wedge}
\newcommand{\norm}[1]{\left\Vert {#1}\right\Vert}

\newcommand{\li}{\left}
\newcommand{\ri}{\right}
\newcommand{\hleft}{\li\langle}
\newcommand{\hright}{\ri\rangle}
\newcommand{\hseg}[2]{\hleft {#1}, {#2} \hright}
\newcommand{\hpr}[3]{\Pr \li[ {#1} \in \hseg{#2}{#3} \ri]}

\newcommand{\cc}{\colon}
\newcommand{\func}[3]{{#1}\cc {#2} \to {#3}}

\newcommand{\isleq}{\stackrel{?}{\leq}}
\newcommand{\isgeq}{\stackrel{?}{\geq}}
\newcommand{\set}[2]{\left\{ \given{#1}{#2}\right\}}
\newcommand{\dd}{\mrm{d}}

\newcommand{\restrict}[2]{{
		\left.\kern-\nulldelimiterspace 
		#1 
		\vphantom{\big|} 
		\right|_{#2} 
}}

\usepackage{titlesec}

\titlespacing{\paragraph}{%
	0pt}{
	0.5\baselineskip}{
	1em}

\newcommand{\tops}[1]{\texorpdfstring{#1}{}}

\newcommand{\nathan}[1]{\begin{color}{purple}{#1}\end{color}}
\definecolor{ohadcolor}{RGB}{0, 127, 255}
\newcommand{\ohad}[1]{\begin{color}{ohadcolor}{#1}\end{color}}

\begin{document}
\title{
Proof of Tomaszewski's Conjecture on Randomly Signed Sums}

	
	\author{Nathan Keller\thanks{Department of Mathematics, Bar Ilan University, Ramat Gan, Israel.
			\texttt{nkeller@math.biu.ac.il}. Research supported by the Israel Science Foundation (grant no. 1612/17) and by the Binational US-Israel Science Foundation (grant no. 2014290).} \mbox{ }and  Ohad Klein\thanks{Department of Mathematics, Bar Ilan University, Ramat Gan, Israel. \texttt{ohadkel@gmail.com}. Research supported by the Clore Scholarship Programme.}}
	
	\maketitle
	
	\begin{abstract}
    We prove the following conjecture, due to Tomaszewski (1986): Let $X= \sum_{i=1}^{n} a_{i} x_{i}$, where $\sum_i a_i^2=1$ and each $x_i$ is a uniformly random sign. Then $\Pr[|X|\leq 1] \geq 1/2$.

    Our main novel tools are local concentration inequalities and an improved Berry-Esseen inequality for Rademacher sums.
	\end{abstract}
	

\section{Introduction}

\subsection{Background}

In the April 1986 issue of \emph{The American Mathematical Monthly}, Richard Guy~\cite{Guy86} presented an open question, attributed to Boguslav Tomaszewski:
\begin{question}
Consider $n$ real numbers $a_1,a_2,\ldots,a_n$ such that $\sum_i a_i^2=1$. Of the $2^n$ expressions $|\epsilon_1 a_1 + \ldots + \epsilon_n a_n|$, with $\epsilon_i=\pm 1$, $1 \leq i \leq n$, can there be more with value $>1$ than with value $\leq 1$?
\end{question}
In the first paper studying the problem, Holzman and Kleitman~\cite{HK92} presented several equivalent formulations:
\begin{itemize}
\item \emph{Sum partitions:} Let $a_1,\ldots,a_n$ be real numbers with $\sum_i a_i^2=1$. Is it true that in at least half of the partitions of $\sum a_i$ into two sums, the sums differ by at most $1$?

\item \emph{Chebyshev-type inequality:} Let $X=\sum_i a_i x_i$, where $\{x_i\}$ are uniformly distributed in $\{-1,1\}$ and independent. Is it true that $\Pr[|X| \leq \sqrt{\var(X)}] \geq 1/2$? Note that Chebyshev's inequality yields a lower bound of $0$ for this probability.

\item \emph{A ball and a cube:} Consider an $n$-dimensional ball and a smallest $n$-dimensional cube containing it. Is it true that for any pair of parallel supporting hyperplanes of the ball, at least half the vertices of the cube lie between (or on) the two hyperplanes?
\end{itemize}
In the sequel, we use the probabilistic notation $X=\sum a_i x_i$, normalizing such that $\sum a_i^2=1$. Holzman and Kleitman~\cite{HK92} proved the lower bound $\Pr[|X|<1] \geq 3/8$, which is tight for $X=\frac{1}{2}(x_1+x_2+x_3+x_4)$, and were not able to prove a stronger lower bound for $\Pr[|X| \leq 1]$. They conjectured that $\Pr[|X|\leq 1] \geq 1/2$, which would be tight, e.g., for $X=a_1 x_1+a_2 x_2$ with any $|a_1|,|a_2|<1$. This conjecture became known as \emph{Tomaszewski's conjecture} or \emph{Tomaszewski's problem}.

As is suggested by its equivalent formulations, Tomaszewski's problem and its variants naturally appear in diverse fields, including probability theory~\cite{BD15,Pinelis12}, geometric analysis~\cite{KR20}, optimization and operation research~\cite{BNR02,So09}, statistics~\cite{Pinelis94}, and theoretical computer science~\cite{DDS16,DG20,Tan12}. It became well-known, and was mentioned in lists of open problems in various fields (e.g.,~\cite{F+14,Hiriart09}).

A number of works obtained partial results toward Tomaszewski's conjecture. Ben-Tal et al.~\cite{BNR02} (who were not aware of the previous work on the conjecture and arrived to it independently, from applications to optimization) proved that $\Pr[|X|\leq 1] \geq 1/3$. Shnurnikov~\cite{Shnurnikov12} improved their lower bound to $0.36$. Boppana and Holzman~\cite{BH17} were the first to cross the $3/8$ barrier, proving a lower bound of $0.406$, and Boppana et al.~\cite{BHZ20} further improved the lower bound to $0.428$. Very recently, Dvo{\v{r}}{\'{a}}k et al.~\cite{DHT20} proved a lower bound of $0.46$, which is the best currently known bound.

Several other authors proved the conjecture in special cases: Bentkus and Dzinzalieta~\cite{BD15} proved it in the case $\max |a_i| \leq 0.16$, Hendriks and van Zuijlen~\cite{HZ17a} proved it for $n \leq 9$ (a proof-sketch in that case, using different methods, was presented earlier by von Heymann~\cite{Heymann12}), van Zuijlen~\cite{vanZuijlen11} proved it when all $a_i$'s are equal, and Toufar~\cite{Toufar18} extended his result to the case where all $a_i$'s but one are equal. In another direction, De et al.~\cite{DDS16} presented an algorithm that allows approximating $\min \{\Pr[|X|\leq 1]|X=\sum a_i x_i\}$ up to an additive error of $\epsilon$. Unfortunately, the complexity of the algorithm is $\exp(O(1/\epsilon^6))$.

\subsection{Our results}

In this paper we prove Tomaszewski's conjecture.
\begin{theorem}\label{thm:main}
Let $X=\sum_{i=1}^n a_i x_i$, where $\sum_{i=1}^{n} a_i^2=1$ and $\{x_i\}$ are independent and uniformly distributed in $\{-1,1\}$. Then
\begin{equation}\label{eq:main}
\Pr[|X|\leq 1] \geq 1/2.
\end{equation}
\end{theorem}
Our result may be interpreted within the general context of \emph{tail bounds for Rademacher sums}. A Rademacher sum is a random variable $X=\sum a_i x_i$, where the $x_i$'s are i.i.d. Rademacher random variables (i.e., are uniformly distributed in $\{-1,1\}$). Estimates on Rademacher sums were studied in numerous papers, both for their own sake (e.g.,~\cite{BGH01,Veraar08}) and for the sake of applications to statistics~\cite{Efron69} and to optimization~\cite{BNR02}. A main direction in this study is obtaining upper and lower bounds on the tail probability $\Pr[|X| > t]$, aiming at showing that the tail of any such $X$ behaves `similarly' to the tail of a Gaussian random variable with the same variance
(see, e.g.,~\cite{BD15,Eaton70,Pinelis12} for upper bounds and~\cite{HK94,MS90,Oles96} for lower bounds).

One of the upper bounds on the tail, proved by Dzindzalieta~\cite{Dzindzalieta14a}, asserts that
\[
\Pr[X>t] \leq \frac{1}{4} + \frac{1}{8}(1-\sqrt{2-\frac{2}{t^2}}), \qquad \forall 1<t \leq \sqrt{2}.
\]
when $\var(X)=1$.
Theorem~\ref{thm:main} improves over this result, showing that
\begin{equation}
\Pr[X>t] \leq 1/4, \qquad \forall t \geq 1,
\end{equation}
which is tight for all $t<\sqrt{2}$, due to $X=\frac{1}{\sqrt{2}}x_1+\frac{1}{\sqrt{2}}x_2$.

\medskip Furthermore, as was observed by Dzindzalieta~\cite{Dzindzalieta14a}, Theorem~\ref{thm:main} implies the more general:
\begin{equation}\label{eq:dzin-intro1}
\Pr[|X|<t] \geq \Pr[|X|>1/t], \qquad \forall t>0,
\end{equation}
and notably, an improvement of~\eqref{eq:main},
$
	\Pr[|X| < 1] + \frac{1}{2}\Pr[|X|=1] \geq \frac{1}{2},
$
being tight for $X= 1\cdot x_1$.


\subsection{Our tools}

The proof of Theorem~\ref{thm:main} uses four main tools.

\paragraph{A new local concentration inequality for Rademacher sums.} In~\cite{KK17}, the authors introduced several local concentration inequalities that allow comparing the probabilities $\Pr[X \in I]$ and $\Pr[X \in J]$, for segments (or rays) $I,J$. We enhance some techniques of~\cite{KK17} and prove:
\begin{theorem}[segment comparison]\label{thm:seg_compare_intro}
		Let $X = \sum_{i=1}^{n} a_{i} x_{i}$ be a Rademacher sum, and write $M=\max_i |a_i|$. For any $A, B, C, D \in \reals$ with $|A| \leq \min(B,C)$, $2M \leq C-A$ and $D-C + \min(2M, D-B) \leq B-A$, we have
		\begin{equation}
			\pr[X \in (C,D)] \leq \pr[X \in (A,B)].
		\end{equation}
\end{theorem}
We prove Theorem~\ref{thm:seg_compare_intro} by constructing an explicit injection.

\paragraph{An improved Berry-Esseen inequality for Rademacher sums.} The classical Berry-Esseen theorem~(\cite{Berry41,Esseen42}) allows approximating a sum of independent random variables by a Gaussian. If $X=\sum a_i x_i$ where all $a_i$'s are `sufficiently small', this allows deducing~\eqref{eq:main}, since for a Gaussian $Z\sim N(0,1)$ we have $\Pr[|Z|\leq 1] \approx 0.68$. Specifically, as was observed by Bentkus and Dzindzalieta~\cite{BD15}, Tomaszewski's conjecture in the case $\max |a_i| \leq 0.16$ follows from the Berry-Esseen bound.

We show that for Rademacher sums, improved Berry-Esseen type bounds (i.e., a tighter approximation by a Gaussian) can be obtained. 
We prove these new bounds via a method proposed by Prawitz~\cite{pra72} which employs the characteristic functions of $X$ and of a Gaussian to bound the difference between the distributions. Our key observation here is that in the case of Rademacher sums, Prawitz' method can be refined, yielding significantly stronger bounds. Then, we use these bounds to deduce Theorem~\ref{thm:main} in the range $\max |a_i| \leq 0.31$ (compared to $0.16$ of~\cite{BD15}) and in cases where the second-largest or the third-largest among the $a_i$'s is `sufficiently small'. We conjecture that these bounds can be improved further; see Section~\ref{sec:open}.

\paragraph{A `semi-inductive' approach, using a stopping time argument.} In their proof of the lower bound $\Pr[|X|\leq 1] \geq 1/3$, Ben-Tal et al.\cite{BNR02} introduced a `stopping time' argument, which treats $X$ as a sum $X=X'+X''$ where $X'=\sum_{i=1}^k a_i x_i$ and $X''=\sum_{i=k+1}^n a_i x_i$, and shows that if the partial sum $X'$ is `a little less than $1$', then the final sum $X$ has a decent chance of remaining less than $1$ in absolute value. Variants of this argument (which was attributed in~\cite{BNR02} to P. Van der Wal) were used in all subsequent proofs of lower bounds for Tomaszewski's problem (i.e.,~\cite{BHZ20,BH17,DHT20,Shnurnikov12}).

In our proof, we also employ a `stopping time' argument, but in a somewhat different manner. We use it to show that the statement of Theorem~\ref{thm:main} regarding $X=\sum_{i=1}^n a_i x_i$, with $a_1+a_2\geq 1$, follows from the assertion of Theorem~\ref{thm:main} for $Z=\sum_{i=1}^{m} b_i x_i$, with properly chosen $m<n$ and $\{b_i\}$. However, as it might be that $b_1+b_2 < 1$, this argument is not really inductive, and is completely useless, unless we can prove~\eqref{eq:main} in the case $a_1+a_2<1$ by different tools.

The semi-inductive approach is used once again, when treating the case $a_1 \in [0.31,0.387] \wedge a_1+a_2+a_3 \geq 1$ of Theorem~\ref{thm:main}.

\paragraph{A refinement of Chebyshev's inequality.} We make repeated use of the following, rather standard, refinement of Chebyshev's inequality:

\begin{proposition}\label{prop:Intro-Chebyshev}
Let $X$ be a symmetric (around $0$) random variable with $\var(X) = 1$, and let $c_0,\ldots,c_n,d_1,\ldots,d_m \in \mathbb{R}$ be such that
		\[
		0 = c_0 \leq c_1 \leq \ldots \leq c_n = 1 = d_0 \leq d_1 \leq \ldots \leq d_m.
		\]
Then
        \begin{equation}
		\sum_{i=0}^{n-1} (1-c_i^2) \pr[X \in (c_i, c_{i+1}]] \geq \sum_{i=1}^{m} (d_i^2-d_{i-1}^2) \pr[X \geq d_i].
		\end{equation}
\end{proposition}

\paragraph{Proof outline.} In the proof of Theorem~\ref{thm:main}, we consider several cases, according to the sizes of the $a_i$'s, and prove the assertion in each of them separately using the tools described above. In particular, the case $a_1+a_2 \geq 1$ is covered by the aforementioned semi-inductive argument, the case $\max |a_i| \leq 0.31$ is proved by the refined Berry-Esseen inequality, and the cases `in the middle' are proved via various combinations of the segment comparison argument and the refined Chebyshev inequality, sometimes using also the refined Berry-Esseen bound.

Unfortunately, this part of the proof requires somewhat grueling computations, including two cases in which a light computer-aided check is applied. For the sake of readability, we divide the proofs into their `essential' part and their `calculation' part, and relay the calculations to the appendices. 

\paragraph{Organization of the paper.}
In Section~\ref{sec:preliminaries} we introduce notation, a basic lemma, and a more detailed outline of the proof. The segment comparison argument (i.e., Theorem~\ref{thm:seg_compare_intro}) is presented in Section~\ref{sec:comparison}. In Section~\ref{sec:BE} we present the refined Berry-Esseen bounds. The semi-inductive argument for the case $a_1+a_2\geq 1$ is given in Section~\ref{sec:induction}. The rest of the cases are presented in Sections~\ref{sec:055}--\ref{sec:remaining} and Appendices~\ref{sec:31big}--\ref{sec:0.31<a_1<0.387}. We conclude with several open problems in Section~\ref{sec:open}.

\section{Preliminaries and Structure of the Proof}
\label{sec:preliminaries}

This section presents notation, a basic lemma, and the structure of the proof in more detail.

\subsection{Notation}

\paragraph{Standard notation.} \skipline

\medskip For $n \in \mathbb{N}$, $[n]$ denotes $\{1,2,\ldots,n\}$.

\medskip For $A,B,C \in \mathbb{R}$, $A+[B,C]$ denotes the segment $[A+C,B+C]$.

\medskip The shorthand $A_1,A_2,\ldots,A_m=x_1,x_2,\ldots,x_m$  denotes the $m$ equalities $A_i=x_i$, $i\in [m]$.

\medskip The shorthands LHS and RHS denote the left hand side and the right hand side of an equation (or an inequality).

\paragraph{Setting.} Throughout the paper, $X$ denotes a normalized Rademacher sum, that is, $X=\sum_{i=1}^n a_i x_i$, where $\sum_{i=1}^{n} a_i^2=1$ and $\{x_i\}$ are independent and uniformly distributed in $\{-1,1\}$. Without loss of generality, we always assume
\[
a_1 \geq a_2 \geq \ldots \geq a_n > 0.
\]
Note that sometimes, we pass from $X$ to an auxiliary random variable, which we denote by $X'=\sum_{i=1}^m b_i x_i$. In such cases, nothing is assumed on the $\{b_i\}$, unless stated otherwise explicitly.

\paragraph{Precision.} The paper contains explicit real numbers, which we present in decimal expansion. When we write these, we mean the exact value we write; we never write a rounded value and mean `a close' number. We make such roundings by writing, e.g.,  $\pi=3.1416 \pm 10^{-5}$.
		
\paragraph{Notation for segments.} We use the somewhat non-standard notation
\begin{equation}\label{eq-def:segment}
\hpr{X}{a}{b} =
\begin{cases}
\pr[X \in [a,b]] - \pr[X=a]/2 - \pr[X=b]/2,		& a<b \\
0,								                & \mrm{Otherwise}
\end{cases}.
\end{equation}
Note that we have
\begin{equation}\label{eq:additivity}
A \leq B \leq C \implies \hpr{X}{A}{C} = \hpr{X}{A}{B} + \hpr{X}{B}{C}.
\end{equation}
Similarly, for $a\leq b$ we denote $\Pr[X \in \langle a,b]] = \Pr[X \in [a,b]]-\Pr[X=a]/2$, and $\Pr[X \in [a,b \rangle]=\Pr[X \in [a,b]] - \Pr[X=b]/2$.

\subsection{A basic lemma -- elimination of variables}

The following lemma allows eliminating several $x_i$'s, by taking into account each possible value of these $x_i$'s separately. The lemma exchanges Tomaszewski's conjecture~\eqref{eq:main} by inequalities which are occasionally easier to approach.
\begin{lemma}\label{lem:nm}
Let $X = \sum_{i=1}^{n} a_i x_i$ be a Rademacher sum with  $\var(X) = 1$ and let $m < n$. Write
\begin{equation}\label{Eq:Elimination1}
\sigma = \sqrt{1-\sum_{i=1}^{m} a_i^2} \qquad \mbox{and} \qquad X' = \sum_{i=m+1}^{n} a_i' x_i, \qquad \mbox{with} \qquad a_i' = \frac{a_i}{\sigma}.
\end{equation}
Tomaszewski's assertion~\eqref{eq:main} is equivalent to
\begin{equation}\label{Eq:Elimination2}
\sum_{j=0}^{2^m-1} \pr[X' > T_j] \leq 2^{m-2},
\end{equation}
where $\{T_j\} = (1 \pm a_1 \pm \cdots \pm a_m)/\sigma$ ranges over all $2^m$ options.
\end{lemma}
	
\begin{proof}
As $X$ is symmetric,~\eqref{eq:main} is equivalent to $\pr[X > 1] \leq 1/4$. By the law of total probability,
$
2^{m} \pr[X > 1] = \sum_{j} \pr[X' > T_j],
$
implying the lemma.
\end{proof}

Lemma~\ref{lem:nm} is simple yet useful. For example, applying the lemma with $m=1$, one can see that~\eqref{eq:main} is equivalent to the inequality
\begin{equation}\label{eq:nm1}
\pr[X' \in \hleft 0, t\ri] ] \geq \pr[X' > 1/t],
\end{equation}
where
\begin{equation}\label{Eq:Prelim-Eliminate-one}
\sigma = \sqrt{1-a_1^2}, \qquad X'=\sum_{i=2}^n \frac{a_i}{\sigma}x_i, \qquad \mbox{and} \qquad t=\sqrt{\frac{1-a_1}{1+a_1}}.
\end{equation}
This readily implies the aforementioned observation of Dzindzalieta~\cite{Dzindzalieta14a} that Tomaszewski's conjecture is equivalent to the more general inequality~\eqref{eq:dzin-intro1}.


\subsection{Hard cases for the proof}\label{ssec:hard-cases}

There are several classes of Rademacher sums which are `hard to handle' with our tools. These classes motivate the partition into cases used in the proof.

One of the obstacles we have to overcome, is the difficulty in distinguishing between the probabilities $\pr[|X| \leq 1]$ and $\pr[|X| < 1]$. This obstacle appeared in previous works as well, and is probably the reason for which the `barrier' of $3/8$ (which is the tight lower bound for $\Pr[|X|<1]$) was not beaten for almost 25 years, until the work of Boppana and Holzman~\cite{BH17}.

As a result, `hard' cases for our proof are not only tightness examples for the conjecture, but also $X$'s for which $\Pr[|X|<1]$ is small.
	We list three such examples of Rademacher sums:
	\begin{enumerate}
		\item $X=(x_1+x_2)/\sqrt{2}$, and more generally, $X=\sum a_i x_i$ with $\var(X)=1$ and
		\begin{equation}\label{Eq:Aux-Hard-case}
		a_1 + \min_{x \in \spm^{[n]\sm \{1\}}} \left|\sum_{i=2}^{n} a_i x_i\right| > 1.
		\end{equation}
				\item $X = \frac{1}{2}\sum_{i=1}^{4} x_i$.
		\item $X = \frac{1}{3}\sum_{i=1}^{9} x_i$.
	\end{enumerate}
	Rademacher sums $X$ that belong to the first class are tightness examples for Theorem~\ref{thm:main}, in the sense that $\pr[|X|\leq 1]=1/2$. These are the only tightness examples we are aware of.
	The two latter classes are not tightness cases of the conjecture, but rather satisfy $\pr[|X| < 1] < 1/2$, and inevitably complicate the proof. For example, in the third case, $\pr[|X| < 1]=\frac{63}{128} < 0.493$, which demonstrates that the improved Berry-Esseen bound we prove in Section~\ref{sec:BE} is almost optimal. Indeed, while the bound implies that $\pr[|X| < 1] \geq 1/2$ holds whenever $\max_i |a_i| \leq 0.31$, the example shows that this assertion fails when $\max |a_i|=1/3$.

\subsection{Structure of the proof}

The proof of Theorem~\ref{thm:main} is split into seven cases, which we shortly overview:


\paragraph{Case~1: $a_1 \leq 0.31$.} This case is covered by the improved Berry-Esseen bound for Rademacher sums presented in Section~\ref{sec:BE}.

\paragraph{Case~2: $a_1 + a_2 \geq 1$.} This case is covered by the semi-inductive argument (which uses the `stopping time' method), and is presented in Section~\ref{sec:induction}.

\paragraph{Case~3: $a_1 \geq 0.55$ and $a_1 + a_2 \leq 1$.} The proof in this case combines the `segment comparison' argument (i.e., Theorem~\ref{thm:seg_compare_intro}) with Chebyshev's inequality. The proof in this case is a simple example of the proof strategy in some the following cases, and so we slightly detail about it.

By Lemma~\ref{lem:nm}, applied with $m=1$, in order to verify $\pr[|X| \leq 1] \geq 1/2$, it suffices to show
\begin{equation}\label{Eq:Prelim-Aux1}
\pr[|X'| \leq t] > \pr[|X'| > 1/t],
\end{equation}
 where $X',t$ are as defined in~\eqref{Eq:Prelim-Eliminate-one}. Since $a_2'$ -- the largest weight of $X'$ -- satisfies $a_2' \leq t$ (which follows from the condition $a_1+a_2 \leq 1$), Theorem~\ref{thm:seg_compare_intro} implies that $\pr[|X'| \leq t] = \Omega(t)$. As Chebyshev's inequality yields $\pr[|X'| > 1/t] \leq t^2$,~\eqref{Eq:Prelim-Aux1} follows if $t$ is sufficiently small, that is, if $a_1$ is sufficiently large. This argument applies when $a_1 \geq 0.55$, and is presented in Section~\ref{sec:055}.

\paragraph{Case~4: $a_1 \in [0.5,0.55]$ and $a_1 + a_2 \leq 1$.} The proof in this case splits according to whether $a_2$ is `large' or `small'. If $a_2$ is small, then the method of the previous case is sufficient. If $a_2$ is large, we eliminate two variables (by applying Lemma~\ref{lem:nm} with $m=2$) and prove the resulting inequality using Theorem~\ref{thm:seg_compare_intro} and the refined Chebyshev inequality (i.e., Proposition~\ref{prop:Intro-Chebyshev}). This case is presented in Section~\ref{sec:0.5<a_1<0.55}.

\paragraph{Case~5: $a_1 \in [0.31, 0.5]$ and $a_1 + a_2 + a_3 \leq 1$.} The proof in this case splits according to the sizes of $a_2$ and $a_3$.

If either of $a_2$ or $a_3$ is sufficiently small, then after elimination of one or two variables (respectively), the problem is reduced to a probabilistic inequality concerning a Rademacher sum with `sufficiently small' weights, and follows from the improved Berry-Esseen inequality.

Otherwise, we eliminate three variables and prove the assertion using Theorem~\ref{thm:seg_compare_intro} and Proposition~\ref{prop:Intro-Chebyshev}.
This case is demonstrated in Section~\ref{sec:remaining} and treated in detail in Appendix~\ref{sec:31big}.

\paragraph{Case~6: $a_1 \in [0.387, 0.5]$ and $a_1 + a_2 + a_3 \geq 1$.} The proof method in this case is superficially similar to that of the previous case. However, a subtle difference in the details makes this case simpler to handle. The proof in this case is demonstrated in Section~\ref{sec:remaining} and treated in detail in Appendix~\ref{sec:39geq}.

\paragraph{Case~7: $a_1 \in [0.31, 0.387]$ and $a_1 + a_2 + a_3 \geq 1$.} The proof in this case splits according to whether there exists some $k\geq 4$ with a `medium-sized' $a_k$.

If there is no such $a_k$, then the weights are partitioned into `large' ones and `small' ones. If there are at most four large weights, then by eliminating the four variables with largest weight, we reduce the problem to a probabilistic inequality concerning a Rademacher sum with `sufficiently small' weights that can be handled easily. Otherwise, we eliminate five variables and show that the assertion follows from  Proposition~\ref{prop:Intro-Chebyshev} and a `light' semi-inductive argument. The latter part is however somewhat cumbersome, as after eliminating five variables, we have to deal with $2^5=32$ summands simultaneously.

In the case where there exists a medium-sized $a_k$, we use an explicit bijection to prove a special segment comparison lemma that holds for it, which in turn allows deducing the assertion using Theorem~\ref{thm:seg_compare_intro} and Proposition~\ref{prop:Intro-Chebyshev}. The proof in this case is demonstrated in Section~\ref{sec:remaining} and treated in detail in Appendix~\ref{sec:0.31<a_1<0.387}.

\medskip \noindent Combination of the seven cases with induction over $n$ (where the inductive assumption is used in Cases 2 and 7), completes the proof of Theorem~\ref{thm:main}.


\section{Local Concentration Inequalities for Rademacher Sums}\label{sec:comparison}

In this section we present two concentration inequalities, that allow comparing the probabilities $\pr[X \in I]$ and $\pr[X \in J]$, where $X$ is a Rademacher sum and $I,J$ are segments or rays. Both results are used extensively throughout the paper. The first is a local concentration inequality:

\begin{theorem}[segment comparison]\label{thm:seg_compare}
	Let $X = \sum_{i=1}^{n} a_{i} x_{i}$, and write $M=\max_i |a_i|$. For all $A, B, C, D \in \reals$ such that
		\begin{equation}\label{Eq:Comparison_condition_general}
		|A| \leq \min(B,C), \qquad 2M \leq C-A, \qquad \mbox{and} \qquad D-C + \min(2M, D-B) \leq B-A,
		\end{equation}
	one has
		\begin{equation}\label{eq:seg_compare}
		\pr[X \in \hseg{C}{D}] \leq \pr[X \in \hseg{A}{B}].
		\end{equation}
\end{theorem}

The assertion of Theorem~\ref{thm:seg_compare} holds for other types of segments as well; see Appendix~\ref{app:sub:other-types}.

\medskip We note that inequalities of the same type were obtained by the authors in~\cite{KK17}, where they were used to obtain an alternative proof of another local tail inequality due to Devroye and Lugosi~\cite{DL08} and to study analytic properties of linear threshold functions. While the inequalities in~\cite{KK17} are qualitative (i.e., of the form $\pr[X \in \hseg{C}{D}] \leq c \pr[X \in \hseg{A}{B}]$ for some non-optimal constant $c$), for our purposes here a more exact inequality is required. Such an inequality is given in Theorem~\ref{thm:seg_compare}, which strictly supersedes~\cite[Lemma 3.1]{KK17}.

\medskip The usefulness of `segment comparison' in the proof of Tomaszewski's conjecture is apparent, as the conjecture itself can be rephrased as a segment comparison inequality:
\begin{equation}\label{Eq:Toma-aux}
\hpr{X}{1}{\infty} \leq \hpr{X}{0}{1}.
\end{equation}
However, Theorem~\ref{thm:seg_compare} alone is not sufficient for our needs since it allows deducing $\pr[X \in I] \leq \pr[X \in J]$ only for segments $I,J$ that satisfy, in particular, $|I| \leq |J|$.

\medskip
The second result we present is a simple-yet-powerful generalization of the classical Chebyshev's inequality, which allows handling cases where $|I|>|J|$. For example, it enables to deduce $\hpr{X}{0}{1} \geq \hpr{X}{\sqrt{2}}{\infty}$, which reminds of~\eqref{Eq:Toma-aux} but is of course much weaker.
	\begin{lemma}\label{lem:cheby}
		Let $X$ be a symmetric (around $0$) random variable with $\var(X) = 1$, and let
		\[
			0 = c_0 \leq c_1 \leq \ldots \leq c_n = 1 = d_0 \leq d_1 \leq \ldots \leq d_m \leq d_{m+1}=\infty.
		\]
		Then
		\begin{equation}\label{eq:cheby_our}
		\sum_{i=0}^{n-1} (1-c_i^2) \hpr{X}{c_i}{c_{i+1}} \geq \sum_{i=1}^{m} (d_i^2-d_{i-1}^2) \pr[X \geq d_i],
		\end{equation}
		and similarly,
		\begin{equation}\label{eq:cheby_our2}
		\sum_{i=0}^{n-1} (1-c_i^2) \hpr{X}{c_i}{c_{i+1}} \geq \sum_{i=1}^{m} (d_i^2-1) \hpr{X}{d_i}{d_{i+1}}.
		\end{equation}
	\end{lemma}
While we did not find this result in the literature, it is presumably known or even folklore. We note that a variant of Lemma~\ref{lem:cheby} was used in the recent work of Dvo{\v{r}}{\'{a}}k et al.~\cite{DHT20}.

\medskip Neither Theorem~\ref{thm:seg_compare} nor Lemma~\ref{lem:cheby} is sufficient for tackling Tomaszewski's conjecture. However, a combination of these tools allows proving the conjecture in several significant cases.

\paragraph{Organization.}
The proof of Theorem~\ref{thm:seg_compare} uses an explicit injection, which maps the event on the left hand side of~\eqref{eq:seg_compare} to the event corresponding to the right hand side. We present several auxiliary bijections that will be used to construct our injection in Section~\ref{ssec:inj-maps}, and then we prove Theorem~\ref{thm:seg_compare}
in Section~\ref{ssec:seg-compare}. The simple and standard proof of Lemma~\ref{lem:cheby} is presented in Appendix~\ref{ssec:cheby}.

\subsection{Auxiliary bijections}\label{ssec:inj-maps}
\newcommand{\PF}{\mathrm{PF}}
\newcommand{\SF}{\mathrm{SF}}
\newcommand{\RF}{\mathrm{RF}}
\newcommand{\DOM}{\mathrm{Domain}}

In this subsection we present three bijections on the discrete cube $\{-1,1\}^n$ (i.e., bijections from the discrete cube to itself) that satisfy certain desired properties. The first two bijections were already introduced in~\cite{KK17} and are presented here for the sake of completeness.

\subsubsection{Prefix flip}

\begin{lemma}[Prefix flip]\label{lem:prefix-flip}
		Let $a = (a_1, \ldots, a_n)$ be a sequence of positive real numbers, let $M=\max_i a_i$ and let $Q \geq 0$. There exists a bijection $\func{\PF_{a, Q}}{\spm^n}{\spm^n}$ such that for any $v \in \spm^n$ with
		\begin{equation}\label{eq:pf-assumption}
		X(v)\defeq \sum_{i=1}^{n} a_i v_i \geq Q/2,
		\end{equation}
		the image $w=\PF_{a, Q}(v)$ satisfies
		\begin{equation}\label{Eq:Comp-Aux1}
		X(w)=\sum_{i=1}^{n} a_i w_i \in (X(v)-Q-2M,X(v)-Q].
		\end{equation}
\end{lemma}

\paragraph{Construction of $\PF_{a, Q}$.} Let $v \in \spm^n$ satisfy~\eqref{eq:pf-assumption}. Let $k>0$ be minimal so that $\sum_{j=1}^{k} a_j v_j \geq Q/2$. Define $w=\PF_{a, Q}(v)$ by $w_i = -v_i$ for $i \leq k$, and $w_i = v_i$ for $i > k$.

\medskip \noindent We prove that $\PF_{a, Q}$ satisfies the requirements of the lemma in Appendix~\ref{ssec:prefix-flip}.
We call the function $\PF_{a, Q}$ a \emph{prefix flip}.

\subsubsection{Single coordinate flip}

	\begin{lemma}[Single coordinate flip]\label{lem:single-flip}
		Let $a = (a_1, \ldots, a_n)$ be a sequence of positive real numbers and let $M=\max_i a_i$ and $m = \min_i a_i$. There exists a bijection $\func{\SF_{a}}{\spm^n}{\spm^n}$ such that for any $v \in \spm^n$  with
		\begin{equation}\label{eq:sf-assumption}
		X(v)\defeq \sum_{i=1}^{n} a_i v_i > 0,
		\end{equation}
		the image $w=\SF_{a}(v)$ is obtained from $v$ by flipping a single coordinate from $1$ to $(-1)$, and in particular, satisfies
		\[
		X(w)=\sum_{i=1}^{n} a_i w_i \in [X(v)-2M,X(v)-2m].
		\]
	\end{lemma}

\paragraph{Construction of $\SF_{a}$.} Let $v \in \spm^n$ satisfy~\eqref{eq:sf-assumption}. Further assume that $a_1 \geq \ldots \geq a_n > 0$. Let $k > 0$ be the minimal value that maximizes the quantity $\sum_{j=1}^{k} v_j$. Define $w=\SF_{a}(v)$ by $w_i = v_i$ for all $i \neq k$, and $w_k = -v_k$.
	
\medskip \noindent We prove that $\SF_{a}$ satisfies the requirements of the lemma in Appendix~\ref{ssec:single-flip}. We call the function $\SF_{a}$ a \emph{single coordinate flip}.

\subsubsection{Recursive flip}
	\begin{lemma}[Recursive flip]\label{lem:recursive-flip}
		Let $a = (a_1, \ldots, a_n)$ be a sequence of positive real numbers and let $M=\max_i a_i$ and $m = \min_i a_i$. There exists a bijection $\func{\RF_{a}}{\spm^n}{\spm^n}$ such that any $v \in \spm^n$ and its image $w = \RF_a(v)$ satisfy:
		
		\medskip\indent \textbf{1)} If $X(v) > 0$ then $X(w) \in [X(v)-2M, X(v)-2m]$;

		\medskip\indent \textbf{2)} If $X(v) \leq 0$, then either $w=-v$ (and $X(w)=-X(v)$) or $X(w) \in [X(v) -2M, 0)$,
	
		\medskip \noindent where $X(u) \defeq \sum_{i=1}^{n} a_i u_i$ for $u \in \spm^n$.
	\end{lemma}


\paragraph{Construction of $\RF_a$.}
Define the auxiliary injection $\func{F}{\set{v\in\spm^n}{X(v)>0}}{\spm^n}$ by
\[
F(v)=-\SF_{a}(v),
\]
and note that $F$ is not defined for $\set{v}{X(v) \leq 0}$.
Define $\RF_{a}:\spm^n \to \spm^n$ by
\[
\RF_{a}(v) =
\begin{cases}
\phantom{-}\SF_{a}(v), & \quad X(v)>0 \\
-(F^{-1})^{k}(v), & \quad X(v)\leq 0
\end{cases},
\]
where $k\geq 0$ is minimal such that $(F^{-1})^{k+1}(v)$ does not exist, i.e., $(F^{-1})^{k}(v) \notin \mathrm{Image}(F)$.

\medskip
Informally, after defining $\RF_a(v)$ for all $\set{v}{X(v)>0}$ using a single coordinate flip, we would like to define $\RF_a(v)$ for any other $v$ as a simple negation: $\RF_a(v)=-v$. However, this may breach the injectivity if there exists some $w$ such that $X(w)>0$ and $\SF_a(w)=-v$. In such a case, we define $v'=\SF_a^{-1}(-v)$ and check whether $-v'$ is `vacant' (i.e., does not collide with any $\SF_a(z)$). If it is vacant, we set $\RF_a(v)=-v'$; otherwise, we continue applying $\SF_a^{-1}$ and negating until we reach a vacant value. The auxiliary function $F$ combines application of $\SF_a$ with negation, and hence, in each step we apply $F^{-1}$, as stated in the definition. We note that this construction is reminiscent of the way of constructing a bijection from two injections, used in the classical proof of the Cantor-Schr\"oder-Bernstein theorem by J. K\"onig.

\medskip 
We call the function $\RF_{a}$ a \emph{recursive flip}.	A concrete example that demonstrates the way $\RF_a$ works is presented in Appendix~\ref{app:misc:RF}, and the proof that $\RF_{a}$ satisfies the assertion of the lemma is presented in Appendix~\ref{ssec:recursive-flip}.


	\subsection{Proof of Theorem~\ref{thm:seg_compare}}\label{ssec:seg-compare}
	The proof of Theorem~\ref{thm:seg_compare} is split into two different arguments, corresponding to how the condition $D-C + \min(D-B, 2M) \leq B-A$ in~\eqref{Eq:Comparison_condition_general} is realized -- either as $D-C+2M \leq B-A$, or as $D-C + (D-B) \leq B-A$. We formulate these two cases as two separate lemmas.
	\begin{lemma}\label{lem:seg-compare1}
		Let $X = \sum_{i=1}^{n} a_{i} x_{i}$ be a Rademacher sum, and write $M=\max_i a_i$. For any $A, B, C, D \in \reals$ such that
		\begin{equation}\label{Eq:Aux-Compare-proof1}
		\min(|A|, |B|) \leq C \qquad \mbox{and} \qquad D-C+2M \leq B-A,
		\end{equation}
		we have
		\begin{equation}\label{eq:seg-compare1}
			\pr[X \in \hseg{C}{D}] \leq \pr[X \in \hseg{A}{B}].
		\end{equation}
	\end{lemma}
	
	\begin{lemma}\label{lem:seg-compare2}
		Let $X = \sum_{i=1}^{n} a_{i} x_{i}$ be a Rademacher sum, and write $M=\max_i a_i$. For any $A, B, C, D \in \reals$ such that
		\begin{equation}\label{Eq:Aux-Compare-proof2}
		|A| \leq C \qquad \mbox{and} \qquad 2\max(M, D-B) \leq C-A,
		\end{equation}
		we have
		\begin{equation}\label{eq:seg-compare2}
			\pr[X \in \hseg{C}{D}] \leq \pr[X \in \hseg{A}{B}].
		\end{equation}
	\end{lemma}
	The combination of Lemmas~\ref{lem:seg-compare1} and~\ref{lem:seg-compare2} immediately implies Theorem~\ref{thm:seg_compare}. Notice that the assumptions in both lemmas are slightly weaker than the assumptions in Theorem~\ref{thm:seg_compare}; this weakening of the assumptions will be needed in the sequel.
	
\begin{remark}
	The assertions of Lemmas~\ref{lem:seg-compare1} and~\ref{lem:seg-compare2} hold for other types of segments as well; see Appendix~\ref{app:sub:other-types}.
\end{remark}	
	
\paragraph{Notation.}
		Given a Rademacher sum $X = \sum_{i=1}^{n} a_i x_i$ and a quadruple $A,B,C,D\in \reals$ that satisfies either~\eqref{Eq:Aux-Compare-proof1} or~\eqref{Eq:Aux-Compare-proof2} (so that either Lemma~\ref{lem:seg-compare1} or Lemma~\ref{lem:seg-compare2} can be applied), or $D \leq C$, we write $\hseg{C}{D} \prec_X \hseg{A}{B}$. Using this notation, the lemmas can be rewritten as the deduction
		\[
			\hseg{C}{D} \prec_X \hseg{A}{B} \qquad  \implies \qquad \hpr{X}{C}{D} \leq \hpr{X}{A}{B}.
		\]
		
\subsubsection{Proof of Lemma~\ref{lem:seg-compare1}}
	
	\begin{proof}[Proof of Lemma~\ref{lem:seg-compare1}]
		We first show that we may assume $|A| \leq |B|$, so that~\eqref{Eq:Aux-Compare-proof1} is upgraded to
		\begin{equation}\label{Eq:Aux-Compare-proof3}
			|A| \leq C \qquad \mbox{and} \qquad D-C+2M \leq B-A.
		\end{equation}

\medskip\noindent\textbf{Upgrading to~\eqref{Eq:Aux-Compare-proof3}.}
		If $A > B$, Lemma~\ref{lem:seg-compare1} is vacant. Otherwise, if $|A| \leq |B|$ we are done. In the remaining case, we exchange $A,B,C,D$ by $-B, -A, C, D$, so that~\eqref{Eq:Aux-Compare-proof3} (and hence~\eqref{Eq:Aux-Compare-proof1}) are satisfied by the new quadruple, and conclude~\eqref{Eq:Aux-Compare-proof1} by noting $\hpr{X}{A}{B} = \hpr{X}{-B}{-A}$.
		

		

\medskip\noindent\textbf{Proving~\eqref{eq:seg-compare1}.}
		We assume $A,B,C,D$ satisfy~\eqref{Eq:Aux-Compare-proof3} and prove~\eqref{eq:seg-compare1}.

		We may also assume $Q \defeq D-B>0$, as otherwise $D\leq B$, and $[C,D] \seq [A,B]$ implies~\eqref{eq:seg-compare1}.
		
		Consider the prefix flip map $\PF_{a,Q}$. For any $v \in \spm^n$ with $X(v) \in [C,D]$, we have $X(v) \geq Q/2$, since by~\eqref{Eq:Aux-Compare-proof3},
		\[
		Q =D-B \leq D-B + 2M \leq C-A \leq 2C.
		\]
		Hence, by Lemma~\ref{lem:prefix-flip}, for any such $v$, the image $w=\PF_{a, Q}(v)$ satisfies
		\[
		X(w) \in (X(v)-Q-2M, X(v)-Q].
		\]
		As $C \leq X(v) \leq D$ and $Q=D-B$, this implies $X(w) \in (A, B]$. Finally, since $X(w)=B$ may occur only if $X(v)=D$, the injectivity of $\PF_{a, Q}$ implies
		\begin{equation}\label{eq:semi-compare}
			\pr[X \in \left[C,D\hright] \leq \pr[X \in \left(A, B\hright],
		\end{equation}
		which is even slightly stronger than the assertion~\eqref{eq:seg-compare1}. This completes the proof.
	\end{proof}

\subsubsection{Proof of Lemma~\ref{lem:seg-compare2}}

The proof of the lemma uses an explicit injection which we hereby describe.

\paragraph{Notation.}
Let $X(x)=\sum_{i=1}^n a_i x_i$ be a Rademacher sum, and $A,B,C,D,M$ be real numbers that satisfy the conditions of Lemma~\ref{lem:seg-compare2}. Note that we may assume $D>B$, as otherwise, the assertion of the lemma holds trivially.

We partition the coefficients $\{a_i\}$ into `large' and `small' ones. Let
\[
L = \li\{ \given{i}{a_{i} \geq (D-B)/2} \ri\}, \qquad S = [n] \sm L.
\]
Note that this partition depends only on the fixed parameters $\{a_i\}$ and $A,B,C,D$.

In addition, for a subset $I \subseteq [n]$ and for $x \in \spm^{n}$, we write $x_I \defeq \restrict{x}{I}$ and $a_I \defeq \restrict{a}{I}$, and correspondingly, $a_I \cdot x_I = \sum_{i \in I} a_{i} x_{i}$.

\theoremstyle{definition}
\newtheorem*{definition*}{Definition}

\begin{definition*}
Let $v \in \spm^n$ be such that $X(v) \in [C,D]$. We define $u=f(v) \in \spm^n$ in two steps. First, we denote
\[
w=g(v)=(\RF_{a_L}(v_L), v_S) \qquad \mbox{and} \qquad Q_{v} = (D-B)-(X(v)-X(w)).
\]
Then, we set
\[
u =f(v)=
\begin{cases}
w, & \quad Q_v \leq 0, \\
(w_L, \PF_{a_S, Q_{v}}(w_S)), & \quad Q_v>0.
\end{cases}
\]
\end{definition*}

\paragraph{Motivation.} We prove Lemma~\ref{lem:seg-compare2} by showing that $f\cc v\mapsto u$ injectively maps any $v$ with $X(v) \in [C,D]$ to some $u$ with $X(u) \in [A,B]$.
The single-coordinate-flip map of the large coordinates, $\SF_{a_L}$, seemingly has the same property, that is, maps any $v'$ with $X(v') \in [C,D]$ to some $w'$ with $X(w') \in [A,B]$.

However, $\SF_{a_L}$ might fail to flip a single coordinate when $a_L \cdot v'_{L} \leq 0$. In such a case, we have $a_S \cdot v'_{S} \geq C \geq (D-B)/2$, and a prefix-flip map $\PF_{a_S, D-B}\cc v' \mapsto u'$ does satisfy $X(u') \in [A,B]$.
This reasoning, of mapping $v'$ into either $w'$ or $u'$, proves that
\[
\hpr{X}{C}{D} \leq 2\cdot\hpr{X}{A}{B},
\]
as it results in a 2-to-1 map (each of $v'\mapsto w'$ and $v'\mapsto u'$ is injective).

We wish to show the stronger inequality $\hpr{X}{C}{D} \leq \hpr{X}{A}{B}$. For this, we construct $f$ in two steps. At the first step we apply on the large coordinates the recursive-flip $\RF_{a_L}$ (which is closely related to $\SF_{a_L}$) to obtain $w$. Then, at the second step, conditioned on data available in $w_L$, we choose whether to apply a prefix flip on the small coordinates or not. This results in a single 1-to-1 map. The usage of $\RF_{a_L}$ instead of $\SF_{a_L}$ is important in order for $w_L$ to have guaranteed properties, even when $a_L \cdot v'_L \leq 0$.

\begin{proof}[Proof of Lemma~\ref{lem:seg-compare2}]
		Let $X = \sum_{i=1}^{n} a_i x_i$ and $A,B,C,D,M$ be as in the statement of the lemma, and define the function $f$ on the set $\set{v}{X(v) \in [C,D]}$ as described above.
		
\paragraph{Why is $f$ injective?}
		To compute the inverse map $u \mapsto v$, first consider $u_L$ and apply the inverse recursive flip $\RF_{a_L}^{-1}$ to recover $v_L$. Then, compute $Q_{v}$ (which depends only on $v_L$ and not on $v_S$). Now, it is clear from the definition of $f$ that we have
		\[
		v =
		\begin{cases}
		(\RF_{a_L}^{-1}(u_L), u_S), & \quad Q_v \leq 0, \\
		(\RF_{a_L}^{-1}(u_L), \PF_{a_S, Q_{v}}^{-1}(u_S)), & \quad Q_v>0.
		\end{cases}
		\]

\paragraph{Why is $f$ into $[A,B]$?}
		Let $v$ satisfy $X(v) \in [C,D]$. We want to show that $u=f(v)$ satisfies $X(u) \in [A,B]$.
		We consider two cases.
		
		\paragraph{Case~1:  $a_L\cdot v_L > 0$.} In this case, since for all $i \in L$ we have $ (D-B)/2 \leq a_i \leq M$, the first property of the recursive flip (presented in Lemma~\ref{lem:recursive-flip}) implies
		\begin{equation}\label{Eq:Aux-Compare-proof4}
		X(w) \in [X(v) - 2M, X(v) - (D-B)].
		\end{equation}
		In particular, we have $Q_{v} = (D-B)-(X(v)-X(w)) \in [D-B-2M,0]$, and thus, by the definition of $f$, we set $u=f(v)=w$. By~\eqref{Eq:Aux-Compare-proof4}, we have
		\[
			X(u) = X(w) \in [X(v) - 2M, X(v) - (D-B)] \subseteq [C-2M,D-(D-B)] \subseteq [A,B],			
		\]
		where the last inclusion uses the assumption $2M \leq C-A$.
		
		\paragraph{Case~2:  $a_L\cdot v_L \leq 0$.} In this case, the second property of $\RF_{a_L}$ (in Lemma~\ref{lem:recursive-flip}) implies
		\begin{equation}\label{Eq:Aux-Compare-proof4.5}
		a_L\cdot w_L \in [a_L\cdot v_L-2M, -a_L\cdot v_L],
		\end{equation}
		and hence,
		\[
		Q_v=(D-B)-(X(v)-X(w) \in (D-B) - [2 a_L \cdot v_L, 2M] = [D-B-2M, D-B-2a_L \cdot v_L].
		\]
		We further subdivide this case into two sub-cases.
		
		\paragraph{Case~2a: $Q_v\leq 0$.} In this case, we have $Q_v \in [D-B-2M,0]$, and hence,
		\[
		X(v)-X(w) = (D-B)-Q_v \in [D-B,2M].
		\]
		As $X(v) \in [C,D]$, this implies
		\[
		X(w) \in [C-2M,D-(D-B)] \subseteq [A,B],
		\]
		where the last inclusion follows from the assumption $2M \leq C-A$.
		
		By the definition of the function $f$, we have $u=w$, and so, $X(u)=X(w) \in [A,B]$.

		\paragraph{Case~2b: $Q_v> 0$.} In this case, by the definition of $f$, we have $u=f(v)=(w_L, \PF_{a_S, Q_{v}}(w_S))$, and hence, we would like to apply
		Lemma~\ref{lem:prefix-flip} to the function $\PF_{a_S,Q_v}:w_S \mapsto u_S$. To this end, we have to prove
		\begin{equation}\label{Eq:Aux-Compare-proof5}
		a_S \cdot w_S \geq Q_{v}/2.
		\end{equation}
		To prove~\eqref{Eq:Aux-Compare-proof5}, note that the assumptions $2(D-B) \leq C-A$ and $|A| \leq C$ imply
		\begin{equation}\label{Eq:Aux-Compare-proof6}
		2C \geq C-A \geq 2(D-B) \geq D-B.
		\end{equation}	
		In addition, by~\eqref{Eq:Aux-Compare-proof4.5} we have
		\begin{equation}\label{Eq:Aux-Compare-proof7}
		X(v)-X(w)=a_L \cdot v_L -a_L \cdot w_L \geq 2a_L \cdot v_L.
		\end{equation}
		Since $a_S \cdot w_S = a_S \cdot v_S$ and $a_S \cdot v_S + a_L \cdot v_L = X(v) \in [C,D]$,~\eqref{Eq:Aux-Compare-proof6} and~\eqref{Eq:Aux-Compare-proof7} imply
		\[
		a_S \cdot w_S = X(v)- a_L \cdot v_L \geq C - a_L\cdot v_L \geq (D-B - (X(v)-X(w)))/2=Q_v/2,
		\]
		proving~\eqref{Eq:Aux-Compare-proof5}.
		
		\medskip As stated above,~\eqref{Eq:Aux-Compare-proof5} allows us to apply Lemma~\ref{lem:prefix-flip} to the function $\PF_{a_S,Q_v}:w_S \mapsto u_S$. Since $M'=\max_{i \in S} a_i \leq (D-B)/2$, the lemma implies
		\begin{align*}
		X(u)-X(w) &= a_S \cdot u_S - a_S \cdot w_S	\in (-2M'-Q_{v}, -Q_{v}] \\
		&\seq ((X(v)-X(w)) - 2(D-B), (X(v)-X(w)) - (D-B)] \\
		&= (X(v)-X(w)) + (-2(D-B), -(D-B)].
		\end{align*}
		Therefore,
		\[
		X(u) = X(v) + (X(w)-X(v)) + (X(u)-X(w)) \in X(v) + (- 2(D-B), -(D-B)].
		\]
		Since $X(v) \in [C,D]$, we have
		\[
		X(u) \in (C-2(D-B),D-(D-B)] \subseteq (A,B],
		\]
		where the last inclusion follows from the assumption $2(D-B)\leq C-A$.

		\paragraph{What happens to the endpoints?}
		Notice that in all cases, $X(u)=A$ may hold only if $X(v)=C$, and $X(u)=B$ may hold only if $X(v)=D$. Hence, the assertion of the lemma:
		\[
		\pr[X \in \hseg{C}{D}] \leq \pr[X \in \hseg{A}{B}].
		\]
		follows from the injectivity of $f\cc v \mapsto u$.
	\end{proof}

	
\section{Improved Berry-Esseen Type Inequalities for Rademacher Sums: Proving Theorem~\ref{thm:main} for \tops{$a_1 \leq 0.31$}}\label{sec:BE}
	A natural approach toward proving Tomaszewski's conjecture in the case where all the coefficients $a_i$ are small, is using the  classical Berry-Esseen theorem~(\cite{Berry41,Esseen42}), which allows approximating a sum of independent random variables by a Gaussian (i.e., a normally distributed random variable).
	\begin{theorem}[Berry-Esseen]
		Let $X_1,X_2,\ldots,X_n$ be independent random variables, such that $\forall i:\mathbb{E}[|X_i|^3]<\infty$. Let $X=\sum_{i=1}^n X_i$. Then for all $x$,
		\begin{equation}\label{Eq:Aux-BE1}
		\left|\Pr\left[\frac{X}{\sqrt{\var(X)}} \leq x\right]-\Pr[Z \leq x]\right| \leq C \cdot \frac{\sum_{i=1}^n \mathbb{E}|X_i|^3}{\var(X)^{3/2}},
		\end{equation}
		where $Z \sim N(0,1)$ is a standard Gaussian and $C$ is an absolute constant.
	\end{theorem}
\noindent	For a Rademacher sum $X=\sum a_i x_i$ with $\var(X)=1$, the theorem yields
	\[
	\left|\Pr[X \leq x]-\Pr[Z \leq x]\right| \leq C \cdot \sum_i a_i^3 \leq C \cdot \max_i |a_i| \sum_i a_i^2 =  C \cdot \max_i |a_i|,
	\]
	and consequently,
	\begin{equation}\label{Eq:Aux-BE2}
	\left|\Pr[|X| \leq 1]-\Pr[|Z| \leq 1]\right| \leq 2C \cdot \max_i |a_i|.
	\end{equation}
	The best currently known upper bound on the constant $C$ in~\eqref{Eq:Aux-BE1} is $C \leq 0.56$, obtained in~\cite{Shevtsova10}. Plugging it into~\eqref{Eq:Aux-BE2} and noting that
	$\Pr[|Z|\leq 1] \geq 0.682$,~\eqref{Eq:Aux-BE2} implies Tomaszewski's conjecture in the range $\max_i |a_i| \leq 0.162$, as was noted by Bentkus and Dzindzalieta~\cite{BD15}.
	
	Using merely the general form of the Berry-Esseen theorem, this result cannot be improved much. Indeed, it was shown by Esseen~\cite{Esseen56} that the constant $C$ in~\eqref{Eq:Aux-BE1} satisfies $C>0.409$, and hence, the best one can hope for by plugging an improved $C$ into~\eqref{Eq:Aux-BE2} is extending the range to $\max_i |a_i| \leq 0.182/0.818 \leq 0.223$.
	
	\medskip In this section we show that refined Berry-Esseen type bounds can be obtained in the special case where $X$ is a Rademacher sum.
	
	Our starting point is a smoothing inequality of Prawitz~\cite{pra72} which allows obtaining bounds on the cumulative distribution function of a random variable $X$ (i.e., $\Pr[X \leq x]$), given partial knowledge of its characteristic function $\varphi_X(t)=\mathbb{E}[e^{itX}]$. Prawitz' inequality has many applications (see Section~\ref{sec:sub:Prawitz}). In particular, Prawitz himself suggested using his inequality to bound $|\Pr[X < x]-\Pr[Z < x]|$, and Shevtsova~\cite{Shevtsova10} used his strategy to prove the best currently known bound on the constant $C$ in the Berry-Esseen theorem.
	
	When $X$ is a Rademacher sum, the strategy of Prawitz can be refined, yielding significantly better bounds than in the general case. Specifically, we obtain the following technical result.
	\begin{proposition}\label{prop:Prawitz-Rademacher}
	Let $X = \sum_i a_i x_i$ be a Rademacher sum with $a_1 \geq a_2 \geq \ldots \geq a_n>0$ and $\var(X)=1$. Then for any $T>0$, $q \in [0,1]$, and $x \in \reals$ we have
	\begin{align}\label{Eq:Our-Prawitz}
		\begin{split}
		\pr[Z < x] - \pr[X < x] &\leq
		\int_{0}^{q} \left|k(u,x,T)\right| g(Tu) \dd{u} +
		\int_{q}^{1} \left|k(u,x,T)\right| h(Tu) \dd{u} + \\
		&+ \int_{0}^{q} k(u,x,T) \exp(-(Tu)^2/2) \dd{u} +
		\int_{0}^{x} \frac{1}{\sqrt{2\pi}}\exp(-u^2/2) \dd{u},
		\end{split}
	\end{align}
where $k(u,x,T) = \frac{(1-u)\sin(\pi u - T u x)}{\sin(\pi u)} - \frac{\sin(T u x)}{\pi}$,
\[
g(v) =
\begin{cases}
\exp(-v^{2}/2) - \cos(a_1 v) ^ {1/a_1^2}, & a_1 v \leq \frac{\pi}{2}\\
\exp(-v^{2}/2)+1, & \mrm{otherwise}
\end{cases}
, \quad h(v) =
\begin{cases}
\exp(-v^{2}/2),					& a_1 v \leq \theta\\
(-\cos(a_1 v))^{1/a_1^2},	    & \theta \leq a_1 v \leq \pi\\
1,								& \mrm{otherwise}
\end{cases},
\]
$Z \sim N(0,1)$ is a standard Gaussian and $\theta=1.778 \pm 10^{-4}$ is the unique root of the function $y \mapsto \exp(-y^2/2)+\cos(y)$ in the interval $[0,\pi]$.	
\end{proposition}

While in this paper we bound ourselves to proving Proposition~\ref{prop:Prawitz-Rademacher}, it appears that one can obtain more general estimates about Rademacher sums using similar strategies; see Section~\ref{sec:open}.


\medskip Our first application of Proposition~\ref{prop:Prawitz-Rademacher} validates Tomaszewski's conjecture in the range $\max_i a_i \leq 0.31$ (compared to $0.16$ that can be obtained by the general Berry-Esseen bound):
\begin{proposition}\label{lem:be_toma}
	Let $X = \sum_i a_i x_i$ be a Rademacher sum with $\max_i a_i \leq 0.31$ and $\var(X)=1$. Then
	\begin{equation}\label{eq:small-toma2}
	\pr[X < 1] \geq \pr[Z < 1] - 0.09115,
	\end{equation}
	and consequently,
	\begin{equation}\label{eq:small-toma}
	\pr[X < 1] \geq 0.7501, \qquad \pr[|X| < 1] \geq 0.5002.
	\end{equation}
\end{proposition}

Proposition~\ref{lem:be_toma} follows from Proposition~\ref{prop:Prawitz-Rademacher} directly, by substitution of suitable parameters. We note that the condition $\max_i a_i \leq 0.31$ cannot be relaxed significantly. Indeed, this is demonstrated by $X = \frac{1}{3} \sum_{i=1}^{9} x_{i}$ having $\max_i a_i =1/3$ and $\pr[|X| < 1] = \frac{63}{128} < 0.493$.

\medskip  Our second application of Proposition~\ref{prop:Prawitz-Rademacher} is a concrete estimate which we shall use in the proof of Tomaszewski's conjecture in the range $\max_i a_i \in (0.31, 0.5)$.
	\begin{proposition}\label{prop:ad-hoc2}
		Let $X = \sum_i a_i x_i$ be a Rademacher sum with $\max_i a_i \leq 0.22$ and $\var(X) = 1$. Then, for any $x \geq 0$, we have
		\begin{equation}\label{eq:ad-hoc2}
			\pr[X \leq x] \geq \pr[Z \leq x] - \max(0.084, \pr[|Z| \leq a_1]/2),
		\end{equation}
		where $Z \sim N(0,1)$ is a standard Gaussian variable.
	\end{proposition}

\paragraph{Organization.}
In Section~\ref{sec:sub:Prawitz} we describe the inequality of Prawitz~\cite{pra72} and apply it to Rademacher sums, proving Proposition~\ref{prop:Prawitz-Rademacher}. In Section~\ref{sec:sub:BE-applications} we prove Propositions~\ref{lem:be_toma} and~\ref{prop:ad-hoc2}.
In Appendix~\ref{app:no-numeric} we show how to practically evaluate the bound~\eqref{Eq:Our-Prawitz} to a required precision.

\subsection{The smoothing inequality of Prawitz applied to Rademacher sums}
\label{sec:sub:Prawitz}

\subsubsection{Prawitz' inequality}

In~\cite{pra72}, H\r{a}kan Prawitz proposed a way for bounding the cumulative distribution function of a random variable $X$, in terms of partial information on its characteristic function $\varphi_X(t)=\mathbb{E}[e^{itX}]$. The main result of~\cite{pra72} (specifically,~\cite[(1b)]{pra72}) reads:
\begin{theorem}\label{thm:prawitz}
	Let $X$ be a real-valued random variable, and assume that the characteristic function $\varphi_X(t)=\be[\exp(i t X)]$ is given for $|t| \leq T$. Then for any $x \in \mathbb{R}$,
\begin{equation}\label{eq:prawitz1}
\pr[X < x] \geq \frac{1}{2} - v.p.\int_{-T}^{T} e^{-ixu} \frac{1}{T}  K(-u/T) \varphi_X(u) \dd{u},
\end{equation}
where
\begin{equation}\label{eq:Ku}
K(u) = \frac{1-|u|}{2} + \frac{i}{2} \li( (1-|u|) \cot(\pi u) + \frac{\sgn(u)}{\pi} \ri).
\end{equation}
\end{theorem}
The notation $v.p. \int_{-T}^{T}$, which is an abbreviation for \emph{valeur principale} integral, has the meaning $\lim_{\eps \to 0^{+}} \li[ \int_{-T}^{-\eps} + \int_{\eps}^{T} \ri]$. It is required, as $K(u) = \frac{i}{2u\pi}+O(1)$ around $u=0$, and a naive integral would diverge. Once using	$v.p. \int_{-T}^{T} = \lim_{\eps \to {0}} \li[ \int_{-T}^{-\eps} + \int_{\eps}^{T} \ri]$, the integral in~\eqref{eq:prawitz1} converges, as $K$ satisfies $K(-u) + K(u) = O(1)$, and is multiplied by a function which is $C+O(u)$ around $u=0$.

\medskip
Theorem~\ref{thm:prawitz} has numerous applications in probability theory (e.g.,~\cite{BG97,GZ14}) and in statistics (e.g.,~\cite{BGZ97,BG02,HW04}). In particular, already in~\cite[Sec.~10]{pra72}, Prawitz suggested using his method for obtaining approximation by a Gaussian variable in terms of distance between characteristic functions. A few decades later, this strategy was used by Shevtsova (\cite{Shevtsova10}, see also~\cite{KS10}) to obtain the best currently known bound on the constant $C$ in the Berry-Esseen theorem.

For more information on Theorem~\ref{thm:prawitz} and its applications, see the survey~\cite{Bobkov16}.

\subsubsection{A refined inequality for Rademacher sums}

For a Rademacher sum $X = \sum_{i=1}^{n} a_{i} x_{i}$, the characteristic function $\varphi_X$ has the convenient form \begin{equation}\label{Eq:Aux-Prawitz1}
\varphi_{X}(t) = \mathbb{E}[e^{itX}] =\prod_{i=1}^{n} \cos(a_{i} t).
\end{equation}
Since~\eqref{Eq:Aux-Prawitz1} is defined for any $t \in \mathbb{R}$, for any $T>0$ we can substitute in~\eqref{eq:prawitz1} $u/T$ in place of $u$:
\begin{equation}\label{eq:prawitz}
\pr[X < x] \geq \frac{1}{2} - v.p.\int_{-1}^{1} e^{-ixuT} K(-u) \varphi_X(uT) \dd{u}.
\end{equation}
Due to the symmetry of $X$ around $0$, we have $\varphi_{X}(-u) = \varphi_{X}(u) \in \reals$. This, together with $\overline{K(-u)}=K(u)$, implies
	\[
		\pr[X < x] \geq \frac{1}{2} - 2\int_{0}^{1} \re \li( e^{ixuT} K(u) \ri) \varphi_{X}(uT) \dd{u}.
	\]
Hence, for $Z \sim N(0,1)$, and for any $T > 0$, $q \in [0,1]$, we have
	\begin{equation}\label{eq:prawitz_renew}
	\begin{aligned}
	\pr[Z < x]- \pr[X < x]
	\leq
	& \li(\pr[Z < x] - \frac{1}{2} \ri)
	+ \int_{0}^{1} 2\re \li( e^{iTux} K(u) \ri) \varphi_{X}(uT) \dd{u} = \\
	=
	& \int_{0}^{q} 2\re \li( e^{iTux} K(u) \ri) \li( \varphi_{X}(uT) - \varphi_{Z}(uT) \ri) \dd{u} + \\
	& +\int_{q}^{1} 2\re \li( e^{iTux} K(u) \ri) \varphi_{X}(uT) \dd{u} + \\
	& +\int_{0}^{q} 2 \re \li( e^{iTux} K(u) \ri) \varphi_{Z}(uT) \dd{u} + \\
	& +\li( \pr[Z < x] - \frac{1}{2} \ri) \leq \\
	\leq
	(S_{1}=) & \int_{0}^{q} \li| 2\re \li( e^{iTux} K(u) \ri) \ri| \li| \varphi_{X}(uT) - \varphi_{Z}(uT) \ri| \dd{u} + \\
	+(S_{2}=) & \int_{q}^{1} \li| 2\re \li( e^{iTux} K(u) \ri) \ri| \li| \varphi_{X}(uT) \ri| \dd{u} + \\
	+(S_{3}=) & \int_{0}^{q} 2 \re \li( e^{iTux} K(u) \ri) \varphi_{Z}(uT) \dd{u} + \\
	+(S_{4}=) & \li( \pr[Z < x] - \frac{1}{2} \ri).
	\end{aligned}
	\end{equation}
	Notice that given $x,T,$ and $q$, the quantities $S_{3}, S_{4}$ are two constants which are, in principle, easy to compute, as $\varphi_{Z}(v)=\exp(-v^2/2)$. In the following, we obtain bounds on $S_{1}$ and $S_{2}$, and deduce Proposition~\ref{prop:Prawitz-Rademacher}.
	
	\paragraph{Bounding $\li| \varphi_{X}(uT) \ri|$.} We claim that
	\begin{equation}\label{eq:be_fx}
	\li| \varphi_{X}(v) \ri| \leq
	\begin{cases}
	\exp(-v^{2}/2),					& 0\leq a_1 v \leq \theta\\
	(-\cos(a_1 v))^{1/a_1^2},	    & \theta \leq a_1 v \leq \pi\\
	1,								& \mrm{Otherwise}
	\end{cases},
	\end{equation}
	where $\theta$ is the unique root of
	\[
	\exp(-x^2/2)+\cos(x)=0
	\]
	in the interval $[0,\pi]$. Its numerical value is $1.778\pm 10^{-4}$.
	
	\medskip \noindent
	To see this, recall that $\varphi_{X}(v) = \prod_{i} \cos(a_i v)$. Clearly, $|\varphi_{X}(v)| \leq 1$. Moreover, it is easy to check by differentiation that as long as $a_i v \leq \theta$, we have $|\cos(a_i v)| \leq \exp(-(a_i v)^2/2)$. As $a_1=\max_i a_i$, it follows that if $a_1 v \leq \theta$, then
	\begin{equation}\label{eq:aux-prawitz}
	|\varphi_{X}(v)| = |\prod_{i} \cos(a_i v)| \leq \exp(-\sum_i a_i^2 v^2 / 2) = \exp(-v^2/2).
	\end{equation}
	To handle the remaining case, $a_1 v \in (\theta, \pi]$, note that
	\begin{equation}\label{Eq:Aux-Prawitz2}
		|\varphi_X(v)| = \exp\li( \sum_i \log|\cos(a_i v)| \ri) = \exp\li( \sum_i a_i^2 \frac{\log|\cos(a_i v)|}{a_i^2} \ri)
		\leq \exp\li( \max_i \frac{\log|\cos(a_i v)|}{a_i^2} \ri),
	\end{equation}
	where the ultimate inequality holds since $\sum a_i^2 = 1$. The right hand side is maximized at $i=1$. Indeed,
	\begin{itemize}
		\item For $i$ with $a_i v \leq \theta$, by the previous case we have $\log|\cos(a_i v)|/a_i^2 \leq -v^2/2$.
		
		\item For $i$ with $a_i v \in  [\theta, \pi]$, as the function $\psi: a \mapsto \log|\cos(a)|/a^2$ increases in the range $a \in [\theta, \pi]$, we get $\psi(a_i v) \leq \psi(a_1 v)$.
	\end{itemize}
	Since $a_1=\max_i a_i$ satisfies $a_1 v \in (\theta, \pi]$, a combination of the two cases gives
	\[
	\forall i: \qquad \log|\cos(a_i v)|/a_i^2 \leq \max(-v^2/2, \log|\cos(a_1 v)|/a_1^2) = \log|\cos(a_1 v)|/a_1^2,
	\]
	yielding $\varphi_X(v) \leq |\cos(a_1 v)|^{1/a_1^2}$, as we claimed.
	
	\paragraph{Bounding $\li| \varphi_{X}(uT) - \varphi_{Z}(uT) \ri|$.} We claim that
	\begin{equation}\label{eq:be_fxfz}
	\li| \varphi_{X}(v) - \varphi_{Z}(v) \ri| \leq
	\begin{cases}
	\exp(-v^{2}/2) - \cos(a_1 v) ^ {1/a_1^2}, & 0\leq a_1 v \leq \frac{\pi}{2}\\
	\exp(-v^{2}/2)+1, & \mrm{Otherwise}
	\end{cases}.
	\end{equation}
	To see this, recall that $\varphi_{Z}(v) = \exp(-v^2/2)$, and hence, for any $v \in \mathbb{R}$,
	\[
	\li| \varphi_{X}(v) - \varphi_{Z}(v) \ri| \leq \exp(-v^2/2)+1.
	\]
	In the case $0 \leq a_1 v \leq \pi / 2$, on the one hand, similarly to~\eqref{eq:aux-prawitz},
	\[
	\varphi_X(v) = \prod_i \cos(a_i v) \leq \exp(-\sum_i a_i^2 v^2/2) = \exp(-v^2/2)=\varphi_Z(v).
	\]
	On the other hand, by virtue of~\eqref{Eq:Aux-Prawitz2} and $\cos(a_i v) \geq 0$,
	\[
	\prod_i \cos(a_i v) \geq \exp(\min_i (\log(\cos(a_i v)) / a_i^2)).
	\]
	As the function $a \mapsto \log(\cos(a))/a^2$ decreases in the range $a \in (0, \pi/2)$, one has that for all $i$, $\log(\cos(a_i v)) / a_i^2 \geq \log(\cos(a_1 v)) / a_1^2$, and thus, \[
	\varphi_Z(v) \geq \varphi_X(v)=\prod_i \cos(a_i v) \geq \exp(\min_i (\log(\cos(a_i v)) / a_i^2)) = \cos(a_1 v)^{1/a_1^2},
	\]
	yielding~\eqref{eq:be_fxfz}.
	
	\paragraph{Simplifying $2\re \li( e^{iTux} K(u) \ri)$.}
	We claim that for all $u \in (0,1)$,
	\begin{equation}\label{eq:k_def}
		k(u,x,T) \defeq 2\re \li( e^{iTux} K(u) \ri) = \frac{(1-u)\sin(\pi u - T u x)}{\sin(\pi u)} - \frac{\sin(T u x)}{\pi}.
	\end{equation}
	To verify~\eqref{eq:k_def}, notice that by definition, $K(u) = \frac{1-|u|}{2} + \frac{i}{2} \li( (1-|u|) \cot(\pi u) + \frac{\sgn(u)}{\pi} \ri)$ and $\exp(iTux) = \cos(Tux)+i\sin(Tux)$, and thus, for $u \in (0,1)$ we have
	\[
		k(u,x,T) = \frac{2\cos(Tux)}{2}(1-u) - \frac{2\sin(Tux)}{2}((1-u)\cot(\pi u)+1/\pi).
	\]
	Substituting $\cot(\pi u) = \cos(\pi u) / \sin(\pi u)$, we get
	\[
		k(u,x,T) = (1-u)\li(\cos(Tux) - \frac{\sin(Tux)\cos(\pi u)}{\sin(\pi u)}\ri) - \sin(Tux)/\pi.
	\]
	Using the identity $\sin(\alpha-\beta)=\sin(\alpha)\cos(\beta)-\sin(\beta)\cos(\alpha)$, we derive~\eqref{eq:k_def}.
	
	\paragraph{Combining the bounds.} Substituting the bounds~\eqref{eq:be_fx} and~\eqref{eq:be_fxfz} and the simplification~\eqref{eq:k_def} into~\eqref{eq:prawitz_renew}, we obtain~\eqref{Eq:Our-Prawitz}, namely, the assertion of Proposition~\ref{prop:Prawitz-Rademacher}.

	\subsection{Applications of the refined Berry-Esseen type inequalities}
	\label{sec:sub:BE-applications}
	
	\subsubsection{Tomaszewski's conjecture for \tops{$a_1 \leq 0.31$}}
	
	We prove Proposition~\ref{lem:be_toma}, which implies Tomaszewski's conjecture in the range $a_1 \leq 0.31$. 		
	\begin{proof}[Proof of Proposition~\ref{lem:be_toma}]
		Consider first the case $a_1=0.31$. Applying Proposition~\ref{prop:Prawitz-Rademacher} with $a_1=0.31$, $x=1$, $T=10$ and $q=0.4$, we obtain
		\begin{equation}\label{eq:0.31}
			\pr[Z < 1] - \pr[X < 1] \leq 0.09114 \pm 10^{-5} \leq 0.09115.
		\end{equation}
		Consequently, $\pr[X < 1] \geq \pr[Z < 1] - 0.09115 > 0.7501$, and since $X$ is a symmetric random variable,
		\[
		\pr[|X| < 1] = 2\pr[X < 1] - 1 \geq 0.5002,
		\]
		as asserted.
		
		To handle the case $a_1 < 0.31$, note that the bound in the right hand side of~\eqref{Eq:Our-Prawitz} is increasing in $a_1$, and thus, an application of Proposition~\ref{prop:Prawitz-Rademacher} with $a_1<0.31$ and the same values $x,T,q$ as above, leads to a stronger lower bound on $\pr[|X| < 1]$. This completes the proof.
	\end{proof}
	
	\subsubsection{A Berry-Esseen type inequality for Rademacher sums with \tops{$a_1 \leq 0.22$}}
	To prove Proposition~\ref{prop:ad-hoc2}, we use the following lemma, whose proof is given in Appendix~\ref{ssec:ad-hoc}.
	\begin{lemma}
		\label{lem:ad-hoc}
		Let $X = \sum_i a_i x_i$ be a Rademacher sum with $\max_i a_i \leq 0.22$ and $\var(X) = 1$. Then for every $x \geq 0.35$, we have
		\begin{equation}\label{eq:ad-hoc}
		\pr[X \leq x] \geq \pr[Z \leq x] - 0.084,
		\end{equation}
		where $Z \sim N(0,1)$ is a standard Gaussian variable.
	\end{lemma}
	The proof of Lemma~\ref{lem:ad-hoc} proceeds by applying Proposition~\ref{prop:Prawitz-Rademacher} to $X$, with suitably chosen parameters $T,q$, and obtaining a slightly stronger version of~\eqref{eq:ad-hoc} for a finite set of $x$'s. Choosing this set of $x$'s fine enough,~\eqref{eq:ad-hoc} follows for all $x \geq 0.35$, by the monotonicity of $x\mapsto \pr[X \leq x]$.
	
	\medskip
	
	Proposition~\ref{prop:ad-hoc2} follows from Lemma~\ref{lem:ad-hoc} and Lemma~\ref{lem:seg-compare2}.
	
	\begin{proof}[Proof of Proposition~\ref{prop:ad-hoc2}]
		The proof is split according to the value of $x$.
		For $x \in [0, a_1)$ we have
		\[
		\pr[X \leq x] \geq 1/2 \geq \pr[Z \leq x] - \pr[|Z| \leq a_1]/2,
		\]
		implying~\eqref{eq:ad-hoc2}. For $x \in [0, 0.2]$, we have $\pr[Z \leq x] \leq 0.58$, and thus, $\pr[X \leq x] \geq 1/2 \geq \pr[Z \leq x] - 0.084$, as asserted. For $x \geq 0.35$, the assertion~\eqref{eq:ad-hoc2} follows directly from Lemma~\ref{lem:ad-hoc}.
		
		\medskip Hence, it is left to prove the assertion for $x \in [\max(a_1, 0.2), 0.35]$. We show that in this range, $\pr[X \leq x] \geq \pr[Z \leq x] - 0.084$.
		
		Indeed, applying Lemma~\ref{lem:seg-compare2} to $X$, with the parameters $A,B,C,D,M = -x, x, x+\eps, 2x+\eps/2, a_1$ and letting $\eps\to 0^{+}$, we obtain
		$
		\pr[|X| \leq x] \geq \pr[X \in (x, 2x]],
		$
		and thus,
		\[
		\pr[X \in \langle 0,x ] ] \geq \frac{1}{3} \pr[X \in \langle 0,2x] ].
		\]
		(Note that the parameters $A,B,C,D,M$ satisfy the assumptions of Lemma~\ref{lem:seg-compare2} since $x \geq a_1$.)
		Applying Lemma~\ref{lem:ad-hoc} with the parameter $2x$ (which can be done, as by assumption, $2x \geq 0.4> 0.35$), and using the symmetry of $X$, we get
		\[
		\pr[X \leq x] \geq 1/2 + \frac{1}{3}\pr[X \in \hleft 0, 2x\ri]] \geq \frac{1}{2} + \frac{1}{3}(\pr[Z \leq 2x] - 0.584).
		\]
		Therefore, in order to complete the proof it is sufficient to show that
		\begin{equation}\label{eq:ad-aux}
		\frac{1}{2} + \frac{1}{3}(\pr[Z \leq 2x] - 0.584) \geq \pr[Z \leq x] - 0.084,
		\end{equation}
		The inequality~\eqref{eq:ad-aux} indeed holds for all $x \in [0.2, 0.35]$. To see this, note that the function $x \mapsto \pr[Z \leq 2x] / 3 - \pr[Z \leq x]$ is decreasing in $[0.2,0.35]$, and thus, it suffices to verify~\eqref{eq:ad-aux} for $x=0.35$. At $x=0.35$, the inequality holds, completing the proof.
	\end{proof}


\section{Theorem~\ref{thm:main} for \tops{$a_1 + a_2 \geq 1$}, via a Semi Inductive Argument}\label{sec:induction}

In this section we prove the following result.
\begin{proposition}\label{prop:semi-inductive}
	Let $3 \leq n \in \mathbb{N}$. Assume that for any Rademacher sum $Z=\sum_{i=1}^m a_i x_i$ with $m < n$ and $\var(Z)=1$, we have $\pr[|Z|\leq 1] \geq 1/2$.
	
	Let $X=\sum_{i=1}^n a_i x_i$ be a Rademacher sum, such that $\var(X)=1$ and $a_1+a_2 \geq 1$. Then $\pr[|X|\leq 1] \geq 1/2$.
\end{proposition}
There is no restriction in assuming $n \geq 3$, as the assertion of Theorem~\ref{thm:main} (namely, $\pr[|X|\leq 1] \geq 1/2$) holds trivially for Rademacher sums with $n \leq 2$. Proposition~\ref{prop:semi-inductive} is only `semi-inductive' in the sense that it assumes that Theorem~\ref{thm:main} holds for all $m <n$, with no restriction on the $a_i$'s, and deduces Theorem~\ref{thm:main} for $m=n$ only in the case $a_1+a_2 \geq 1$. Hence, it can be used only if we resolve the case $a_1+a_2<1$ by a different argument, as we do in the following sections.

The proof relies on a `stopping time' argument that is reminiscent of the stopping time argument of Ben-Tal et al.~\cite{BNR02} used in all recent works on Tomaszewski's conjecture.

\paragraph{Elimination of two variables.} We start with the following variant of Lemma~\ref{lem:nm}.
\begin{lemma}\label{cor:nm2}
	Let $X = \sum_{i=1}^{n} a_{i} x_i = a_1 x_1 + a_2 x_2 + \sigma X'$ with $\var(X) = \var(X') = 1$ (so that $\sigma=\sqrt{1-a_1^2-a_2^2}$). The assertion
	\[
	\pr[|X| \leq 1] \geq 1/2
	\]
	is equivalent to the following inequality involving $X'$:
	\begin{equation}\label{eq:gt1_ineq}
	\pr\li[ X' \in [L_{1}, L_{2}] \ri] \geq \pr \li[ X' > R_{1} \ri] + \pr \li[ X' > R_{2} \ri],
	\end{equation}
	where
	\[
	L_1, L_2 = \frac{a_1+a_2-1}{\sigma}, \frac{1-a_1+a_2}{\sigma},
	\]
	and
	\[
	R_1, R_2 = \frac{1+a_1-a_2}{\sigma}, \frac{1+a_1+a_2}{\sigma}.
	\]
\end{lemma}
\begin{proof}
	The assertion follows immediately from Lemma~\ref{lem:nm} with $m=2$, since  $\pr[X' \in [L_1, L_2]] = 1 - \pr[X' > -L_1] - \pr[X' > L_2]$.
\end{proof}

\noindent Note that if $a_1 \geq a_2$ and $a_1+a_2 \geq 1$ then the parameters $L_1,L_2,R_1,R_2$ satisfy
\begin{equation}\label{Eq:Aux-Induction1}
0 \leq L_1 < L_2 \leq R_1 < R_2.
\end{equation}
This sequence of inequalities will be used several times in the sequel.

\paragraph{The semi-inductive stopping time argument}

Let $X = \sum_{i=1}^{n} a_i x_i$ (with $n \geq 3$) be a Rademacher sum with $\var(X)=1$, and write $X = a_1 x_1 + a_2 x_2 + \sigma X'$, as in Lemma~\ref{cor:nm2}. For $i \geq 3$, write $a_i' = a_i / \sigma$ so that $X' = \sum_{i=3}^{n} a_i' x_i$. By Lemma~\ref{cor:nm2}, in order to deduce $\pr[|X|\leq 1]\geq 1/2$, it suffices to verify~\eqref{eq:gt1_ineq}.

Write $X'$ as the sum of two random variables $X'=Y'+Z'$, as follows. Let the random variable $k$ to be the minimal index with $\sum_{i=3}^{k} a_{i}'x_{i} \geq L_{1}$, and set
\[
Y' \defeq \sum_{i=3}^{k} a_{i}'x_{i}, \qquad \mbox{and} \qquad Z' \defeq \sum_{i=k+1}^n a_{i}' x_{i}.
\]
If no such index exists, let $k=n$ and set $Z'=0$.

\medskip By lemma~\ref{cor:nm2}, in order to prove $\pr[|X| \leq 1] \geq 1/2$, it suffices to prove~\eqref{eq:gt1_ineq}, that is
\begin{equation}\label{eq:gt11}
\pr\li[ Y'+Z' \in [L_{1}, L_{2}] \ri] \geq \pr \li[ Y'+Z' > R_{1} \ri] + \pr \li[ Y'+Z' > R_{2} \ri].
\end{equation}
We shall show that~\eqref{eq:gt11} holds even if we condition on any possible value of $k,Y'$. This is clearly sufficient, due to the law of total probability.

For any specific assignment of $k$ and $Y'$, the above inequality is a probabilistic inequality involving the random variable $Z'$. We consider two cases:
\begin{itemize}
	\item \emph{Case~1: $Z' \equiv 0$.} We show that in this case,~\eqref{eq:gt11} holds as its right hand side is $0$.
	
	\item \emph{Case~2: $Z' \not \equiv 0$.} In this case, we show that~\eqref{eq:gt11} follows from the inequality~\eqref{eq:gt1_ineq} applied to the Rademacher sum $Z'/\sqrt{\var(Z')}$. To show that the latter inequality holds, we note that when $k$ is fixed, $Z'$ is a Rademacher sum on the $n-k$ variables $x_{k+1},\ldots,x_n$. Hence, by applying Lemma~\ref{cor:nm2} in the inverse direction, we may infer~\eqref{eq:gt1_ineq} for $Z'$ from the assertion $\Pr[|Z|\leq 1]\geq 1/2$ for an appropriate Rademacher sum $Z$ on $n-k+2$ variables, which holds due to the inductive hypothesis, since $k \geq 3$.
\end{itemize}

\noindent For the proof of~\eqref{eq:gt11}, we observe the following relation between $Y'$ and $\var(Z' | Y')$.
\begin{claim}
	\label{claim:aux-induction}
	Let $X',Y',Z'$ be as defined above, assume $Y' \geq L_1$ and let $s=(1-\sum_{i=3}^k (a'_i)^2)^{1/2}$. We have
	\begin{equation}\label{eq:gt1_sy}
	Y'(Y' - L_{1}) \leq 1 - s^2.
	\end{equation}
\end{claim}

\begin{proof}
By the definition of $k$, we have $Y'-L_{1} \leq a_{k}' \leq a_{k-1} \leq \ldots \leq a_3'$. Thus,
\[
1-s^{2} = \sum_{i=3}^{k} (a'_{i})^{2} \geq \min_{3 \leq i \leq k} \li\{ a_{i}'\ri\} \sum_{i=3}^{k} a_{i}' \geq (Y'-L_{1})
\sum_{i=3}^{k} a'_{i} x_{i} = (Y'-L_{1}) Y',
\]
as asserted.
\end{proof}

\paragraph{The case $Z' \equiv 0$.} We observe that in this `singular' case, the right hand side of~\eqref{eq:gt11} is zero, and hence the inequality trivially holds.
Indeed, $Z' \equiv 0$ occurs in one of two cases:
\begin{itemize}
	\item There does not exist $k$ such that $Y' = \sum_{i=3}^{k} a_{i}'x_{i} \geq L_{1}$. In this case, $Y'+Z'=\sum_{i=3}^{n} a_{i}'x_{i} < L_1$, and thus, the right hand side of~\eqref{eq:gt11} is clearly equal to zero by~\eqref{Eq:Aux-Induction1}.
	
	\item The minimal $k$ such that $Y' \defeq \sum_{i=3}^{k} a_{i}'x_{i} \geq L_{1}$ is $k=n$.
	By~\eqref{Eq:Aux-Induction1}, in order to show that the r.h.s. of~\eqref{eq:gt11} is equal to zero, it is sufficient to prove that $Y'=Y'+Z' < R_1$. To see this, observe that by Claim~\ref{claim:aux-induction}, $Y'(Y'-L_1) \leq 1$, and hence $Y' < R_1$ follows from $R_1(R_1 - L_1) > 1$ (recall $0 \leq L_1 \leq R_1$ by~\eqref{Eq:Aux-Induction1}). This latter inequality reads as $2(1+a_1-a_2)(1-a_2)/\sigma^2 > 1$, where $\sigma^2=1-a_1^2-a_2^2$ and $a_1 \geq a_2$, and follows by
	\[
		2(1+a_1-a_2)(1-a_2)=\sigma^2+(a_1-a_2)(2+a_1-a_2)+(2a_2-1)^2/2+1/2 > \sigma^2.
	\]
\end{itemize}

\paragraph{The case $Z' \not \equiv 0$.} Denote $s = (\var(Z'))^{1/2} = (\sum_{i=k+1}^n (a'_i)^2)^{1/2}$, so that $\frac{1}{s} Z'$ is a Rademacher sum with variance~1 on $n-k$ variables. We have to prove~\eqref{eq:gt11} which reads as
\begin{equation}\label{eq:gt1_need}
\pr\li[ \frac{Z'}{s} \in \li[ \frac{L_{1}-Y'}{s}, \frac{L_{2}-Y'}{s}\ri] \ri] \geq \pr \li[ \frac{Z'}{s} > \frac{R_{1}-Y'}{s} \ri] + \pr \li[ \frac{Z'}{s} > \frac{R_{2}-Y'}{s} \ri].
\end{equation}
(Note that we assume $k,Y'$ are fixed, and hence, the probabilities in~\eqref{eq:gt1_need} depend only on $Z'$.)
We would like to deduce~\eqref{eq:gt1_need} from the assertion $\Pr[|Z|\leq 1]\geq 1/2$ for an auxiliary Rademacher sum $Z$ on $n-k+2$ variables, which holds due to the inductive hypothesis (since $n-k+2<n$).

\medskip To this end, we pick $b_1, b_2 \in \mathbb{R}$ such that
\begin{equation}\label{Eq:Aux-Induction3}
\frac{L_{1}-Y'}{s}, \frac{L_{2}-Y'}{s} \eqdef L_{1}', L_{2}' = \frac{b_1+b_2-1}{\sigma'}, \frac{1-b_1+b_2}{\sigma'},
\end{equation}
where $\sigma' = \sqrt{1-b_1^2 -b_2^{2}}$. (Concrete values of the possibly negative $b_1, b_2$ are given below.) We define a Rademacher sum $Z = b_1 z_1 + b_2 z_2 + \frac{\sigma'}{s} Z'$, let
\[
R_{1}', R_{2}' = \frac{1+b_1-b_2}{\sigma'}, \frac{1+b_1+b_2}{\sigma'},
\]
and show that $R'_1$ and $R'_2$ satisfy
\begin{equation}\label{eq:gt1_RR}
R_{1}' \leq \frac{R_{1}-Y'}{s} \qquad \mbox{and} \qquad R_{2}' \leq \frac{R_{2}-Y'}{s}.
\end{equation}
As $Z$ is a Rademacher sum with variance~1 on $n-k+2<n$ variables, the induction hypothesis together with Lemma~\ref{cor:nm2} yields
\[
\pr\li[\frac{Z'}{s} \in [L'_{1}, L'_{2}] \ri] \geq \pr \li[\frac{Z'}{s} > R'_{1} \ri] + \pr \li[\frac{Z'}{s} > R'_{2} \ri],
\]
which implies~\eqref{eq:gt1_need} via~\eqref{eq:gt1_RR}. So, it is only left to show that $Z$ is well defined and that~\eqref{eq:gt1_RR} holds.

\paragraph{Why is $Z$ well-defined?}

To show that there exist $b_1,b_2 \in \mathbb{R}$ that satisfy~\eqref{Eq:Aux-Induction3}, let
\[
b_1 = \frac{2 + 2L_1' L_2'}{2 + (L'_1)^2 + (L'_2)^2} \qquad  \mbox{and} \qquad b_2 = \frac{(L'_2)^2 - (L'_1)^2}{2 + (L'_1)^2 + (L'_2)^2}.
\]
A direct computation shows that
\begin{equation}\label{Eq:Aux-Induction4}
\sigma' = \sqrt{1-b_1^2-b_2^2} = \frac{2(L'_2-L'_1)}{2+(L'_1)^2+(L'_2)^2},
\end{equation}
and in particular, $b_1^2+b_2^2 < 1$ (via~\eqref{Eq:Aux-Induction1} and the definition of $L_1', L_2'$ in~\eqref{Eq:Aux-Induction3}).
A further direct computation shows that the right equality in~\eqref{Eq:Aux-Induction3} holds as well.

\paragraph{Proving~\eqref{eq:gt1_RR}.}
To verify the two inequalities in~\eqref{eq:gt1_RR}, notice that by adding $L_{2}'$ to both sides of the first, and $-L_1'$ to both sides of the second, and recalling the definitions of $L_1,L_2,R_1,R_2$ and $L_1', L_2', R_1', R_2'$, these inequalities are respectively equivalent to:
\[
\frac{s}{\sigma'} + Y' \leq \frac{1}{\sigma} \qquad \mbox{and} \qquad \frac{s}{\sigma'} \leq \frac{1}{\sigma}.
\]
As $Y' \geq L_1 \geq 0$ by~\eqref{Eq:Aux-Induction1}, the former inequality clearly implies the latter; hence, we focus only on it. Note that similarly to~\eqref{Eq:Aux-Induction4}, we have
\[
\sigma = \frac{2(L_2-L_1)}{2 + L_1^2 + L_2^2}.
\]
In addition, by~\eqref{Eq:Aux-Induction3}, we have $L'_2-L'_1 = (L_2-L_1)/s$. Using this and substituting the values of $\sigma,\sigma'$, the inequality we seek to prove $\frac{s}{\sigma'} + Y' \leq \frac{1}{\sigma}$, reads as
\[
\frac{s^{2}(2+L_1'^2+L_2'^2)}{2(L_2-L_1)} + Y' \leq \frac{(2+L_1^2+L_2^2)}{2(L_2-L_1)}.
\]
Multiplying by $2(L_2-L_1)$ and substituting the values of $L_1', L_2'$ from~\eqref{Eq:Aux-Induction3}, we reduce to showing
\[
2s^2 + (L_1-Y')^2 + (L_2-Y')^2 + 2Y'(L_2-L_1) \leq 2 + L_1^2 + L_2^2,
\]
Since $s^2 \leq 1-Y'(Y' - L_{1})$ by Claim~\ref{claim:aux-induction}, it is sufficient to prove that
\[
2 - 2Y'(Y'-L_1) + (L_1-Y')^2 + (L_2-Y')^2 + 2Y'(L_2-L_1) \leq 2+L_1^2+L_2^2.
\]
Simplifying this inequality, one sees it is equivalent to the inequality $2Y' \cdot L_{1} \geq 0$, which indeed holds since by~\eqref{Eq:Aux-Induction1}, we have $0 \leq L_{1} \leq Y'$. This completes the proof of Proposition~\ref{prop:semi-inductive}.


\section{Theorem~\ref{thm:main} for \tops{$a_1 \geq 0.55$}, \tops{$a_1+a_2 < 1$}}\label{sec:055}
	In the previous sections we handled Theorem~\ref{thm:main} in the cases $a_1 \leq 0.31$ (Section~\ref{sec:sub:BE-applications}), and $a_1+a_2 \geq 1$ (Section~\ref{sec:induction}, semi-inductively). In this section we handle the case $(a_1 \geq 0.55) \wedge (a_1 + a_2 < 1)$.
	
	
	\paragraph{Elimination step.} Let $X = \sum_{i=1}^{n} a_i x_i$ be a Rademacher sum with $\var(X)=1$, $a_1 \geq (1+\sqrt{8})/7$, and $a_1+a_2 < 1$. (Note that $1+\sqrt{8}/7 \leq 0.55$.) 
	We want to prove $\pr[|X| \leq 1] \geq 1/2$.
	
	\medskip By the case $m=1$ of Lemma~\ref{lem:nm}, it is sufficient to prove that
	\begin{equation}\label{Eq:Aux6.1}
	\pr[X' \in \hleft 0, t\ri] ] \geq \pr[X' > 1/t],
	\end{equation}
	where
	\[
	\sigma = \sqrt{1-a_1^2}, \qquad a'_i = \frac{a_i}{\sigma}, \qquad X'=\sum_{i=2}^n a'_i x_i, \qquad \mbox{and} \qquad t=\sqrt{\frac{1-a_1}{1+a_1}}=\frac{1-a_1}{\sigma}
	\]
	(see~\eqref{eq:nm1} and~\eqref{Eq:Prelim-Eliminate-one} above). As we assume $a_1 + a_2 < 1$, we have
	\begin{equation}\label{Eq:Aux6.2}
	\forall i\cc a_i' < t.
	\end{equation}
	
	\paragraph{Segment comparison step.}
	
	\begin{claim}\label{Cl:Aux6.1}
		Let $X',t$ be as defined above. Then:
		\begin{enumerate}
			\item[(a)] $\hpr{X'}{t}{2t} \leq 2\hpr{X'}{0}{t},$ and
			
			\item[(b)] $\forall k \geq 2 \cc \hpr{X'}{kt}{(k+1)t} \leq 4\hpr{X'}{0}{t}$.
		\end{enumerate}
	\end{claim}
	
	\begin{proof}\skipline

\noindent	(a) Applying Lemma~\ref{lem:seg-compare2} to $X'$, with the parameters $A,B,C,D,M=-t,t,t,2t,a_2'$, we get
		\[
		\hpr{X'}{t}{2t} \leq \hpr{X'}{-t}{t}.
		\]
		(Note that $2M\leq C-A$ follows from~\eqref{Eq:Aux6.2}). By the symmetry of $X'$, this implies~(a).
	
\medskip \noindent
		(b) Applying Lemma~\ref{lem:seg-compare1} to $X'$, with the parameters $A,B,C,D,M=-t, 2t, kt, (k+1)t, a_2'$, we get $\hpr{X'}{kt}{(k+1)t} \leq \hpr{X'}{-t}{2t}$. Using (a) as
		\[
		\hpr{X'}{-t}{2t} = \hpr{X'}{-t}{t} + \hpr{X'}{t}{2t} \leq 4 \hpr{X'}{0}{t},
		\]
		the assertion (b) follows.
	\end{proof}
	
\paragraph{Chebyshev-type inequality step.} Applying the inequality~\eqref{eq:cheby_our} to $X'$, with
	\[
		c_0, c_1, c_2, c_3, \ldots, c_{\lceil 1/t \rceil} = 0, t, 2t, 3t, \ldots, 1 \qquad \mbox{and} \qquad d_0, d_1 = 1, 1/t,
	\]
	we obtain
	\[
	\sum_{k=0}^{\lceil 1/t \rceil-1} (1-(k t)^2) \hpr{X'}{k t}{(k+1)t} \geq \Big(\frac{1}{t^2}-1 \Big) \pr[X' \geq 1/t].
	\]
	By Claim~\ref{Cl:Aux6.1}, this implies
	\[
		\Big( 1 + 2(1-t^2) + 4\sum_{k=2}^{\lceil 1/t \rceil -1} \li( 1-(kt)^2\ri) \Big) \hpr{X'}{0}{t}
		\geq
		\Big( \frac{1}{t^2} - 1 \Big) \pr[X' \geq 1/t].
	\]
	Dividing both sides by $1/t^2-1$, we obtain an inequality of the form
	\[
	C_{t} \cdot \hpr{X'}{0}{t} \geq \pr[X' \geq 1/t].
	\]
	This inequality implies~\eqref{Eq:Aux6.1}, provided  $C_{t} \leq 1$. Hence, it is left to verify:
	\begin{equation}\label{eq:055fin}
		C_{t} \isleq 1, \quad \text{with} \quad C_{t} = \frac{1 + 2(1-t^2) + 4\sum_{k=2}^{\lceil 1/t \rceil-1} \li( 1-(kt)^2\ri)}{(1/t^2) - 1}, \quad 0 < t \leq \sqrt{1-2^{-1/2}}.
	\end{equation}
	(Note that the condition $0< t \leq (1-2^{-1/2})^{1/2}$ follows from the assumption $a_1 \in \li[ (1+\sqrt{8})/7, 1\ri)$, via $t=\sqrt{(1-a_1)/(1+a_1)}$. This is the only place where the assumption $a_1 \geq 0.55$ is used.) The derivation of~\eqref{eq:055fin} is presented in Appendix~\ref{app:055fin}.


\section{Theorem~\ref{thm:main} for \tops{$0.5 \leq a_1 \leq 0.55$}, \tops{$a_1 + a_2 < 1$}}\label{sec:0.5<a_1<0.55}
	The proof of Theorem~\ref{thm:main} in this range is split into two cases: Small $a_2$ and large $a_2$. The threshold between the cases depends on $a_1$, being  $(a_1-3+\sqrt{25+10a_1-63a_1^2})/8$.
	
	\subsection{The case \tops{$a_2 \leq (a_1-3+\sqrt{25+10a_1-63a_1^2})/8$}}
	
	Let $X = \sum a_i x_i$ be a Rademacher sum with $\var(X)=1$, $a_1 \in [0.5, 0.55]$, $a_1+a_2 < 1$, and $a_2 \leq (a_1-3+\sqrt{25+10a_1-63a_1^2})/8$. The proof that $\pr[|X| \leq 1] \geq 1/2$ is almost identical to the argument of Section~\ref{sec:055}, as we explain below.
	
	\paragraph{Elimination step.} Like in Section~\ref{sec:055}, we note that it is sufficient to prove $\pr[X' \in \hleft 0, t\ri] ] \geq \pr[X' > 1/t]$, where $X',t$ are defined as in~\eqref{Eq:Aux6.1}.
	
	\paragraph{Segment comparison step.} Instead of Claim~\ref{Cl:Aux6.1}, we use the following comparisons:	
	\begin{equation}\label{eq:05rep}
	\hpr{X'}{t}{3t/2} \leq \hpr{X'}{0}{t} \quad \mbox{and} \quad \hpr{X'}{3t/2}{1} \leq \hpr{X'}{0}{t}.
	\end{equation}
	The first inequality follows from Lemma~\ref{lem:seg-compare2}, applied to $X'$ with the parameters $A,B,C,D,M = 0, t, t, 3t/2, a_2/\sigma$, and the second inequality follows from Lemma~\ref{lem:seg-compare1}, applied to $X'$ with the parameters $A,B,C,D,M = 0, t, 3t/2, 1, a_2/\sigma$. To show that Lemmas~\ref{lem:seg-compare2} and~\ref{lem:seg-compare1} indeed can be applied (i.e., that the assumptions of the lemmas are satisfied), it is sufficient to verify:
	\begin{equation}\label{eq:05saux}
	\begin{gathered}
	\frac{2a_2}{\sigma} \isleq t, \qquad \quad (1 - 3t/2) + \frac{2a_2}{\sigma} \isleq t\\
	\text{with:}\qquad a_1 \in [0.5, 0.55],\quad\sigma = \sqrt{1-a_1^2},\quad\\ t=\sqrt{\frac{1-a_1}{1+a_1}},\quad a_2 \leq \frac{a_1-3+\sqrt{25+10a_1-63a_1^2}}{8}.
	\end{gathered}
	\end{equation}
	These inequalities are proved in Appendix~\ref{app:05saux}.
	
	\paragraph{Chebyshev-type inequality step.} Note that as $a_1 \geq 1/2$, we have $t=\sqrt{(1-a_1)/(1+a_1)} \leq 1/\sqrt{3}< 2/3$.
	Applying the inequality~\eqref{eq:cheby_our} to $X'$, with
	\[
	c_0, c_1, c_2, c_3 = 0, t, 3t/2, 1 \qquad \mbox{and} \qquad d_0, d_1 = 1, 1/t,
	\]
	and using~\eqref{eq:05rep}, we obtain
	\[
	(1 + (1-t^2) + (1-(3t/2)^2)) \hpr{X'}{0}{t} \geq \left(\frac{1}{t^2}-1\right) \pr[X' \geq 1/t].
	\]	
	Hence, the task of deducing~\eqref{Eq:Aux6.1} boils down to verifying:
	\begin{equation}\label{eq:05rep2}
	C_t' \isleq 1, \qquad \text{with:}\qquad C_t' = \frac{1 + (1-t^2) + (1-(3t/2)^2)}{(1/t^2)-1},\quad t^2 \in (0, 1/3].
	\end{equation}
	This verification is done in Appendix~\ref{app:05fin2}.
	
	\subsection{The case \tops{$a_2 \geq (a_1-3+\sqrt{25+10a_1-63a_1^2})/8$}}
	
	Let $X = \sum a_i x_i$ be a Rademacher sum with $\var(X)=1$, $a_1 \in [0.5, 0.55]$, $a_1+a_2 < 1$, and $a_2 \geq (a_1-3+\sqrt{25+10a_1-63a_1^2})/8$. The proof of $\pr[|X| \leq 1] \geq 1/2$ is similar to the above strategy, but this time, two variables are eliminated.
	
	\paragraph{Elimination step.} By Lemma~\ref{cor:nm2}, it it sufficient to show that
	\begin{equation}\label{eq:05need}
		\hpr{X'}{0}{-L_1} + \hpr{X'}{0}{L_2} \geq \pr[X' > R_1] + \pr[X' > R_2],
	\end{equation}
	where $X', \sigma, L_1, L_2, R_1, R_2$ are as defined in Lemma~\ref{cor:nm2}. (But this time, since $a_1+a_2<1$ we have $L_1<0$). Note that unlike the previous subsection, $X'$ depends on $n-2$ variables.
	
	\paragraph{Auxiliary estimates.} We use several auxiliary estimates on $L_2,R_1,R_2$:
	\begin{equation}\label{eq:05fin}
	\begin{gathered}
	L_2 \isgeq \frac{2}{3}, \qquad R_1 \isgeq \max(\sqrt{3 - L_2^2},\sqrt{2}), \qquad R_2 \isgeq \max(\sqrt{5 - 2L_2^2},\sqrt{3}),\\
	\text{with:}\qquad L_2 = \frac{1-a_1+a_2}{\sigma},\quad R_1 = \frac{1+a_1-a_2}{\sigma},\quad R_2 = \frac{1+a_1+a_2}{\sigma}, \\
	\sigma = \sqrt{1-a_1^2-a_2^2},\quad
	a_1 \in [0.5, 0.55],\quad
	a_1+a_2 < 1,\\
	a_2 \geq \li( a_1-3+\sqrt{25+10a_1-63a_1^2}\ri)/8.
	\end{gathered}
	\end{equation}
	These inequalities are proved in Appendix~\ref{app:05fin}.
	
	\paragraph{The sub-case $L_2 \geq 1$.} Applying the Chebyshev-type inequality~\eqref{eq:cheby_our} to $X'$, with $c_0,c_1=0,1$ and $d_1,d_2,d_3=1,\sqrt{2},\sqrt{3}$, we obtain
	\[
	\hpr{X'}{0}{1} \geq \pr[X' \geq \sqrt{2}] + \pr[X' \geq \sqrt{3}],
	\]
	which implies~\eqref{eq:05need} via~\eqref{eq:05fin}.
	
	\paragraph{The sub-case $L_2 < 1$.} Applying Lemma~\ref{lem:seg-compare2} to $X'$, with the parameters $A,B,C,D,M = 0, L_2, L_2, 1, a_{3}/\sigma$, we obtain
	\begin{equation}\label{Eq:Aux7.1}
	\hpr{X'}{L_2}{1} \leq \hpr{X'}{0}{L_2}.
	\end{equation}
	Notice that the assumptions of Lemma~\ref{lem:seg-compare2} are satisfied, as $2(D-B)=2(1-L_2)\leq L_2=C-A$ by~\eqref{eq:05fin}, and
	\[
	2M \leq \frac{2a_2}{\sigma} < \frac{1-a_1+a_2}{\sigma} = L_2 =C-A,
	\]
	by the assumption $a_1+a_2 < 1$.
	
	\medskip Applying the Chebyshev-type inequality~\eqref{eq:cheby_our} to $X'$, with $c_0,c_1,c_2=0,L_2,1$ and $d_1,d_2,d_3=1,\sqrt{3-L_2^2},\sqrt{5-2L_2^2}$, we obtain
	\begin{equation}\label{eq:05cheby}
	\begin{split}
		\hpr{X'}{0}{L_2} &+ (1-L_2^2)\hpr{X'}{L_2}{1} \geq \\
		&\geq (2-L_2^2)\pr\left[X'\geq \sqrt{3-L_2^{2}}\right] + (2-L_2^2)\pr\left[X' \geq \sqrt{5-2L_2^{2}}\right].
	\end{split}
	\end{equation}
	By~\eqref{Eq:Aux7.1}, this implies
	\[
	\hpr{X'}{0}{L_2} \geq \pr\left[X'\geq \sqrt{3-L_2^{2}}\right] + \pr\left[X' \geq \sqrt{5-2L_2^{2}}\right],
	\]
	which, in turn, implies~\eqref{eq:05need} via~\eqref{eq:05fin}. This completes the proof.

\section{The remaining case: $a_1 \in (0.31, 0.5)$}\label{sec:remaining}

Our proof is most involved in this range, although there are no remarkable tightness examples with $a_1 \in (0.31, 0.5)$ (except for $X = \frac{1}{3} \sum_{i=1}^9 x_i$, being the `lightest' appearing in Section~\ref{ssec:hard-cases}).

The full proof in this range is deferred to appendices~\ref{sec:31big},~\ref{sec:39geq}, and~\ref{sec:0.31<a_1<0.387}, corresponding to three main subcases. The general structure of the proof is to condition on the values of a few largest weights (say, $a_1, a_2, a_3$), usually through elimination (Lemma~\ref{lem:nm}), and to prove Tomaszewski's assertion~\eqref{eq:main} regardless of the values we condition on.
While the careful proof is somewhat cumbersome, it was plotted by considering several specific fixings of the large weights (say, $a_1,a_2,a_3 = 0.4, 0.3, 0.2$) and proving the assertion~\eqref{eq:main} under these fixings. Then, the proof was generalized to capture any such fixing.

We demonstrate the proof by presenting several (not fully) representative such fixings, and proving~\eqref{eq:main} under them.

\subsection{Subcase $a_1 + a_2 + a_3 \leq 1$}\label{ssec:rem-case1}
The following example concisely demonstrates all the steps in the proof of this case (given in Appendix~\ref{sec:31big}).

Assume
\[
	a_1, a_2, a_3 = 0.36, 0.2, 0.15,
\]
so that 3-elimination (Lemma~\ref{lem:nm}) reduces us to proving the following inequality for all Rademacher sums $X'$ with $\var(X')=1$, whose largest weight is $\leq 0.15/(1-0.36^2-0.2^2-0.15^2)^{1/2} < 0.17$:
\begin{equation}\label{eq:demo1}
\begin{gathered}
	\hpr{X'}{0}{0.32} + \hpr{X'}{0}{0.65} + \hpr{X'}{0}{0.76} + \hpr{X'}{0}{1.1}
	\\ \geq \\
	\pr[X' > 1.12] + \pr[X' > 1.45] + \pr[X' > 1.56] + \pr[X' > 1.9]
\end{gathered}
\end{equation}
Using Chebyshev's inequality~\eqref{eq:cheby_our}, we lower bound an expression similar to the LHS of~\eqref{eq:demo1}:
\begin{equation}\label{eq:demo1-cheby}
\begin{gathered}
	\hpr{X'}{0}{0.32} + 0.9\hpr{X'}{0.32}{0.65} +\\
	0.6\hpr{X'}{0.65}{0.76} + 0.5\hpr{X'}{0.76}{1}
	\\ \geq \\
	(0.25\pr[X' > 1.12] + 0.25\pr[X' > 1.23]) + 0.5\pr[X' > 1.42] +\\
	(0.4\pr[X' > 1.56]  + 0.1\pr[X' > 1.6]) + 0.5\pr[X' > 1.75].
\end{gathered}
\end{equation}
Can we deduce~\eqref{eq:demo1} from~\eqref{eq:demo1-cheby}? Not immediately. If we multiply~\eqref{eq:demo1-cheby} by $2$ and see what is `missing' in order to deduce~\eqref{eq:demo1}, we get an inequality weaker than:
\[
	\hpr{X'}{0}{0.32} + \hpr{X'}{0}{0.65} \geq 0.5\hpr{X'}{1.12}{1.23} + 0.2\hpr{X'}{1.56}{1.6}.
\]
Recall that the largest weight of $X'$ is at most $0.17$, so using segment comparison (specifically, Lemma~\ref{lem:seg-compare1}), we can prove this inequality by showing $\hpr{X'}{1.12}{1.23} \leq \hpr{X'}{0}{0.65}$ and $\hpr{X'}{1.56}{1.6} \leq \hpr{X'}{0}{0.65}$.

\subsection{Subcase $a_1 + a_2 + a_3 \geq 1$ and $a_1 \in (0.387, 0.5)$}\label{ssec:rem-case2}
The proof in this range is similar to that of Section~\ref{ssec:rem-case1}, but the details are much simpler. We again give a demonstration that captures the essence of the proof (given in Appendix~\ref{sec:39geq}).

Assume
\[
	a_1, a_2, a_3 = 0.4, 0.35, 0.3,
\]
so that 3-elimination (Lemma~\ref{lem:nm}) reduces us to prove the following inequality for all Rademacher sums $X'$ with $\var(X')=1$, whose largest weight is $\leq 0.3/(1-0.4^2-0.35^2-0.3^2)^{1/2} < 0.38$:
\begin{equation}\label{eq:demo2}
\begin{gathered}
	\hpr{X'}{0.07}{0.69} + \hpr{X'}{0}{0.82} + \hpr{X'}{0}{0.94}
	\\ \geq \\
	\pr[X' > 1.57] + \pr[X' > 1.7] + \pr[X' > 1.83] + \pr[X' > 2.58]
\end{gathered}
\end{equation}
We prove~\eqref{eq:demo2} even without the $\hpr{X'}{0.07}{0.69}$ term. Using Chebyshev's inequality~\eqref{eq:cheby_our}:
\begin{equation}\label{eq:demo2-cheby}
\begin{gathered}
	\hpr{X'}{0}{0.82} + 0.33\hpr{X'}{0.82}{0.94} + 0.12\hpr{X'}{0.94}{1}
	\\ \geq \\
	0.7 \pr[X' > 1.31] + 0.7 \pr[X' > 1.55] + 0.7 \pr[X' > 1.77] + 0.7 \pr[X' > 1.95].
\end{gathered}
\end{equation}
Can we deduce~\eqref{eq:demo2} from~\eqref{eq:demo2-cheby}? Not immediately. If try to deduce~$0.7\cdot$\eqref{eq:demo2} from~\eqref{eq:demo2-cheby} we see that what is `missing', is weaker than:
\[
	0.4\hpr{X'}{0}{0.82}-0.12\hpr{X'}{0.94}{1} \geq 0.
\]
Using segment comparison, and specifically Lemma~\ref{lem:seg-compare1}, barely applicable since $(1-0.94) + 2\cdot 0.38 = 0.82$, we conclude $\hpr{X'}{0.94}{1} \leq \hpr{X'}{0}{0.82}$. Also, we are equally satisfied with
\[
	\hpr{X'}{0.94}{1} \leq \hpr{X'}{-0.82}{0.82} \leq 2\hpr{X'}{0}{0.82},
\]
which liberally follows from Lemma~\ref{lem:seg-compare1} (and whose analog holds true in the entire range).

\subsection{Subcase $a_1 + a_2 + a_3 \geq 1$ and $a_1 \in (1/3, 0.387)$}

The proof in this case is a bit more complicated than in the other cases, and the demonstration we give does not capture the entire set of arguments we use. However it represents perhaps the most exotic argument, and features a semi-inductive argument along with a 5-elimination.

Assume
\[
a_1 = a_2 = a_3 = a_4 = a_5 = 0.34,
\]
so that 5-elimination (Lemma~\ref{lem:nm}) reduces us to proving the following inequality for all Rademacher sums $X'$ with $\var(X')=1$, whose largest weight is $\leq 0.34/(1-5\cdot 0.34^2)^{1/2} < 0.6$:
\begin{equation}\label{eq:demo3}
\begin{gathered}
	4\hpr{X'}{0}{1.01} + 5\hpr{X'}{0.04}{1.01} + \hpr{X'}{1.08}{2.06}
	\\ \geq \\
	\pr[X' > 1.01] + 9\pr[X' > 2.06] + 5\pr[X' > 3.1] + \pr[X' > 4.15].
\end{gathered}
\end{equation}
We prove~\eqref{eq:demo3} even with only the first term on the LHS. Using Chebyshev's inequality~\eqref{eq:cheby_our}:
\begin{equation}\label{eq:demo3-cheby}
	3\hpr{X'}{0}{1} \geq 9\pr[X' > 2] + 15\pr[X' > 3].
\end{equation}
In order to deduce~\eqref{eq:demo3} from~(\ref{eq:demo3-cheby}), we have to show
\begin{equation}\label{eq:demo3-final}
	\hpr{X'}{0}{1.01} \geq \pr[X' > 1.01].
\end{equation}
This inequality is actually not easy to prove. However, it is implied by Tomaszewski's assertion~\eqref{eq:main} for the variable $X'$ which depends on $n-5$ weights (where $n$ is the number of weights in the original Rademacher sum $X$ we discuss)! We comment that in general we should use~\eqref{eq:dzin-intro1} instead of~\eqref{eq:main} to prove~\eqref{eq:demo3-final}, which is possible as~\eqref{eq:dzin-intro1} for $X'$ is implied from Tomaszewski's assertion~\eqref{eq:main} on $n-4$ weights.

\section{Open Problems}\label{sec:open}

As was mentioned in the introduction, we believe the methods developed in this paper can be applied to obtain further results on the distribution of Rademacher sums. We conclude the paper with several related open questions.

\paragraph{Tail bounds for Rademacher sums.} Consider the following general problem.
\newtheorem{problem}[theorem]{Problem}
\begin{problem}
Let $\mathcal{X}$ be the class of all Rademacher sums with variance $1$. Characterize the following function, defined for all $x \in \reals$:
\[
	F(x) = \sup_{X \in \mathcal{X}} \pr[X > x].
\]
\end{problem}
\noindent While different sub-cases of this problem were studied in many papers (see, e.g.,~\cite{Dzindzalieta14,Pinelis12} and the numerous references therein),
only a few exact results on it are known (e.g.,~\cite{BD15,Pinelis12}). Theorem~\ref{thm:main} continues the series of exact results, showing that $F(x) = 1/4$ for all $x \in [1, \sqrt{2})$.

A well-known conjecture, due to Hitczenko and Kwapie\'{n}~\cite{HK94}, concerns $F(x)$ for $x = -1$.
\begin{conjecture}[\cite{HK94}]\label{conj:olesz}
	Let $X = \sum a_i x_i$ be a Rademacher sum with $\var(X) = 1$. Then
	\[
		\pr[X \geq 1] \geq \frac{7}{64}.
	\]
\end{conjecture}
Conjecture~\ref{conj:olesz} is an evident counterpart of Tomaszewski's conjecture -- while the latter states that $\pr[X > 1]$ must be somewhat small, the former declares that $\pr[X \geq 1]$ must be somewhat large.

The best currently known result toward Conjecture~\ref{conj:olesz} is $\pr[X > 1] \geq 1/20$ (whenever $X \not \equiv x_1$), proved by Oleszkiewicz~\cite{Oles96} more than 20 years ago. Using our methods (specifically, Proposition~\ref{prop:Prawitz-Rademacher} and Lemma~\ref{lem:nm} with $m \leq 3$) and additional tools, Dvo{\v{r}}{\'{a}}k and the second author~\cite{DK21} proved the stronger bound $\pr[X \geq 1] \geq 6/64$, along with the sharp bound $\pr[X > 1] \geq \frac{1}{16}$.

\paragraph{Improved Berry-Esseen type bounds for Rademacher sums.}
Proposition~\ref{prop:ad-hoc2} shows that for a Rademacher sum $X=\sum a_i x_i$ with $\var(X) = 1$ and $0 \leq a_i\leq a_1=0.22$, and for any $x \geq 0$,
\[
\pr[X \leq x] \geq \pr[Z \leq x] - \pr[Z \in \langle 0, a_1]],
\]
where $Z$ is a standard Gaussian. The following conjecture is a natural extension:
\begin{conjecture}
	Let $X = \sum a_i x_i$ be a Rademacher sum with $\var(X) = 1$ and $\forall i \cc 0 < a_i \leq a_1$. Let $Z \sim N(0, 1)$ be a standard Gaussian, and $x \in \reals$. Then
	\begin{equation}\label{eq:be_conj}
		|\pr[X \leq x] - \pr[Z \leq x]| \leq \pr[Z \in (0, a_1)] < \frac{1}{\sqrt{2\pi}} a_1.
	\end{equation}
\end{conjecture}
Note that first inequality in~\eqref{eq:be_conj} is tight, as is demonstrated by $X = \sum_{i=1}^{n} x_i / \sqrt{n}$ for $n$ odd, and $x=1/\sqrt{n}-\eps$.
We note that it follows from~\cite[Theorem~1.3 and Remark~1.4(c)]{MS18} that for $X = \sum_{i=1}^{n} x_i / \sqrt{n}$, this inequality holds for any $n,x$.

It appears that our methods can be used to prove the conjecture in part of the range, namely,
$\pr[Z \leq x] - \pr[X \leq x] \leq \pr[Z \in (0, a_1)]$ for all $x \geq 0$, in a way similar to the proof of Proposition~\ref{prop:ad-hoc2}.

\section*{Acknowledgements}

We thank Tom Kalvari, Jiange Li, Lunz Mattner, and Jeremy Schiff, and especially Ron Holzman, for inspiring discussions and useful suggestions.

\begin{appendices}

\section{Proofs for Section~\ref{sec:comparison}}

\subsection{Proof of Lemma~\ref{lem:prefix-flip}}\label{ssec:prefix-flip}
\begin{proof}
		We define $\PF_{a, Q}(v)$ for all $v\in \spm^n$ that satisfy $X(v) \geq Q/2$, and show that the map is injective, into $\spm^n$, and satisfies~\eqref{Eq:Comp-Aux1}. Then, $\PF_{a, Q}$ can be arbitrarily completed to a bijection on $\spm^n$ that satisfies the assertion of the lemma.

		Given $v$, consider the partial sums $s_k(v)=\sum_{j=1}^{k} a_j v_j$, and let $k \in [n]$ be minimal such that $s_k(v) \geq Q/2$. Note that $k$ is well-defined, since $s_n(v)= X(v) \geq Q/2$. Define $w = \PF_{a, Q}(v)$ by
		\[
			w_i =
			\begin{cases}
				-v_i, & \quad i \leq k \\
	  \phantom{-}v_i, & \quad i > k
			\end{cases}.
		\]
		It turns out that $k$, and hence $v$, can be recovered from $w$ (which implies that $\PF_{a, Q}$ is injective):
		\[
		k = \min\set{i\in [n]}{\sum_{j=1}^{i} a_j w_j \leq -Q/2}, \qquad \mbox{and} \qquad
		v_i =
			\begin{cases}
				-w_i, & \quad i \leq k \\
	  \phantom{-}w_i, & \quad i > k
			\end{cases}.
		\]
		Hence, to complete the proof we have to show that for all $v$ and $w=\PF_{a, Q}(v)$, we have $X(w) \in (X(v) - Q - 2M, X(v) - Q]$.
		
		To see this, note that $X(v) - X(w) = 2s_k(v)$. As $k$ is minimal with the property $s_k(v) \geq Q/2$, and as $M=\max_i a_i$, we have $s_k(v)\in [Q/2, Q/2+M)$,
		and hence,
		\[
		X(w) = X(v)-2s_k(v) \in (X(v) - Q - 2M, X(v) - Q],
		\]
		as asserted.
	\end{proof}

\subsection{Proof of Lemma~\ref{lem:single-flip}}\label{ssec:single-flip}
\begin{proof}
		We define $\SF_{a}(v)$ for all $v\in \spm^n$ that satisfy $X(v) > 0$, and show that the map is injective, into $\spm^n$, and that $v$ and $w=\SF_{a}(v)$ differ in exactly one coordinate. Then, $\SF_{a}$ can be arbitrarily completed to a bijection on $\spm^n$ that satisfies the assertion of the lemma. Without loss of generality, assume $a_1 \geq a_2 \geq \ldots \geq a_n > 0$.

		Given $v$, consider its partial sums $S_i(v) = \sum_{j=1}^{i} v_j$, where $S_0(v)=0$. (Note that these partial sums differ from the sums $s_i(v)$ considered in Lemma~\ref{lem:prefix-flip}). Let $k \in [n]$ be minimal such that $S_k(v) \geq S_j(v)$, for all $j \in [n]$ (i.e., the lowest amongst the indices in which the maximal partial sum is attained). Then, define $w = \SF_{a}(v)$ by
		\[
			w_i =
			\begin{cases}
				-v_i, & \quad i = k \\
	  \phantom{-}v_i, & \quad i \neq k
			\end{cases}.
		\]
		Let us temporarily suppose that $S_k(v) > 0$; we prove this assertion at the end of the proof.
		
		First, we observe that $\SF_{a}(v)$ is obtained from $v$ by flipping a `positive' coordinate (i.e., $v_k=1$). To see this, notice that if $k>1$, then $v_k=1$, as otherwise $S_{k-1}(v) > S_{k}(v)$, contradicting the definition of $k$. If $k=1$, one has $v_1 = S_1(v) > 0$, yielding, once again, $v_k=1$.
		
		Second, we observe that $k$, and hence $v$, can be recovered from $w$, via the formula
			\begin{equation}\label{eq:kformula}
			k = 1 + \max\set{0 \leq i < n}{\forall j \in [n] \cc S_{i}(w) \geq S_{j}(w)}, \mbox{ and} \quad
			v_i =
			\begin{cases}
				-w_i, & \quad i = k \\
	  \phantom{-}w_i, & \quad i \neq k
			\end{cases}.
			\end{equation}
		To verify~\eqref{eq:kformula}, we consider two cases.
	\begin{itemize}
		\item If $k>1$, then since $v_k=1$, we have
		\[
		S_{k-1}(w) = S_{k-1}(v) = S_{k}(v)-1.
		\]
		Moreover, for all $i \geq k$ one has $S_i(w) = S_i(v)-2 < S_{k-1}(w)$. Similarly, by the definition of $k$, we have $S_{i}(v) < S_{k}(v)$ for all $i < k$, and thus, $S_{i}(w) = S_{i}(v) \leq S_{k}(v) - 1 = S_{k-1}(v)$. Hence, $(k-1)$ is indeed the highest index $i$ in which $S_{i}(w)$ is maximal, as asserted in~\eqref{eq:kformula}.
		
		\item If $k=1$, then by the definition of $k$, for any $i \in [n]$, we have $S_i(v)\leq S_1(v)=v_1=1$. Thus, for any $i \geq 1$, we have $S_i(w)=S_i(v)-2<0$. Thus, $0=S_0(w)>S_i(w)$ for all $i$, which again confirms~\eqref{eq:kformula}.
	\end{itemize}
		Hence, $\SF_{a}(v)$ is an injection, as asserted.
		
		\medskip To conclude the proof, we show that for the chosen $k$, we have $S_{k}(v) > 0$. By the definition of $k$, this assertion is equivalent to $\max_{i\in[n]} S_i(v) > 0$. The latter follows from Abel's summation formula using the assumptions $X(v)=\sum_{i=1}^{n} a_i v_i > 0$ and $a_{i}-a_{i+1} \geq 0$, with $a_{n+1}=0$:
		\[
			0 < X(v) = \sum_{i=1}^{n} a_i v_i = \sum_{i=1}^{n} (a_{i}-a_{i+1}) S_i(v) \leq a_1 \max_{i\in[n]} S_i(v).
		\]
		This completes the proof.
	\end{proof}

\subsection{Proof of Lemma~\ref{lem:recursive-flip}}\label{ssec:recursive-flip}
\begin{proof}[Proof of Lemma~\ref{lem:recursive-flip}]
		Let $\RF_a$ and $F$ be defined as in Section~\ref{ssec:inj-maps}.

		\paragraph{Why is $\RF_{a}$ well defined?} To show $\RF_{a}$ is well-defined, we have to prove that for any $v$ with $X(v)\leq 0$, there exists $k \geq 0$ such that $F^{-k}(v) \notin \mathrm{Image}(F)$.
		
		Indeed, if there was no such $k$, we would enter a loop (as $\spm^n$ is finite), that is, $(F^{-1})^{k}(v) = (F^{-1})^{k'}(v)$ for some $k < k'$.
		Applying $F$ on both sides, $k$ times, we would get $v = (F^{-1})^{k''}(v)$,
		with $k''>0$, contradicting the assumption $v \notin \DOM(F)$.
		
		\paragraph{Why is $\RF_a$ a bijection?}
		It is sufficient to show $\RF_{a}$ is an injection. Suppose $\RF_a(v)=\RF_a(u)$. We consider three cases.
		\begin{itemize}
			\item If $X(v) > 0$ and $X(u) > 0$, then $v=u$ because $\SF_a$ is injective, due to Lemma~\ref{lem:single-flip}.
			
			\item If $X(v) \leq 0$ and $X(u) > 0$, then we have
			$-(F^{-1})^{k}(v) = \RF_a(v) = \RF_a(u) = -F(u)$, with $k$ as above. In
			particular, $(F^{-1})^{k}(v) = F(u)$, and so $(F^{-1})^{k+1}(v)$
			exists, contradicting the definition of $k$.
			
			\item If $X(v) \leq 0$ and $X(u) \leq 0$, then we have $(F^{-1})^{k}(v) = (F^{-1})^{k'}(u)$. Without loss of generality, assume $k' \geq k$, and thus, after applying $F$ to both sides, $k$ times, we obtain $v =
			(F^{-1})^{k''}(u)$, with $k'' \geq 0$. If $k''=0$ then $v=u$ and we
			are done. Otherwise, we get a contradiction, since $v \notin
			\DOM(F)$.
		\end{itemize}
		
		\paragraph{Why does $\RF_{a}$ satisfy the asserted properties?}
			The first property of $\RF_{a}$ holds directly by Lemma~\ref{lem:single-flip}. We show that the second property holds as well. Let $v$ be such that $X(v) \leq 0$, and so, $w \defeq \RF_a(v) = -(F^{-1})^{k}(v)$ for some $k \geq 0$. We consider three cases.
		\begin{itemize}
			\item If $k=0$, then $w= -v$ and $X(w) = -X(v)$.
			
			\item If $k = 1$, then $-w=F^{-1}(v)$, and so $X(-w) > 0$ and
			\[
			\RF_a(-w) = \SF_a(-w) = -F(-w)= -v.
			\]
			Thus, from the \emph{first property} of $\RF_a$ (i.e., Lemma~\ref{lem:single-flip}), we have
			\[
			-X(v) = X(-v)=X(\RF_a(-w)) \in \li[ X(-w) -2M, X (-w) - 2m \ri].
			\]
			This implies
			$X(v) \in [X(w)+2m,X(w)+2M]$, and hence,
			\[
			X(w) \in [X(v) - 2M, X(v) - 2m].
			\]
			
			\item If $k > 1$, let $u = (F^{-1})^{k-1}(v)$, so that $-w=F^{-1}(u)$. Clearly, $u \in \DOM(F)$ and so $X(u) > 0$.
			By the \emph{first property} of $\RF_a$ (i.e., Lemma~\ref{lem:single-flip}), we have
			\[
				0 > - X(u) = X(-u) \geq X(\RF_a^{-1}(-u))-2M = X(F^{-1}(u))-2M = X(-w)-2M.
			\]
			Moreover, $(-w) \in \DOM(F)$ implies $X(-w) > 0$. Finally, since $X(w)=-X(-w)$, we have $X(w) \in (-2M, 0)$. This completes the proof.
			\end{itemize}
\end{proof}

\subsection{An example of the bijection Recursive Flip}
\label{app:misc:RF}


\begin{example}
	Let $X(v)=a_1 v_1+a_2 v_2 +a_3 v_3$, where $a_1 \geq a_2 \geq a_3>0$ and $a_1>a_2+a_3$. Clearly, $M=\max_i a_i=a_1$ and $m=\min_i a_i=a_3$. The set $\{v:X(v)>0\}$ consists of
	\[
	v^1=(1,1,1), \quad v^2=(1,1,-1), \quad v^3=(1,-1,1), \quad v^4=(1,-1,-1).
	\]
	For these vectors, we define $\RF_a(v)=\SF_a(v)$ and $F(v)=-\SF_a(v)$, and so, we have
	\[
	\begin{gathered}
	\RF_a(v^1)=(1,1,-1),\quad \RF_a(v^2)=(1,-1,-1), \\
	\RF_a(v^3)=(-1,-1,1),\quad \RF_a(v^4)=(-1,-1,-1),
	\end{gathered}
	\]
	and
	\[
	F(v^1)=(-1,-1,1), \quad  F(v^2)=(-1,1,1), \quad F(v^3)=(1,1,-1), \quad F(v^4)=(1,1,1).
	\]
	Now, we consider the four remaining vectors:
	\[
	w^1=(-1,-1,-1), \quad w^2=(-1,-1,1), \quad w^3=(-1,1,-1), \quad w^4=(-1,1,1).
	\]
	We have $w^1,w^3 \notin \mathrm{Image}(F)$, and hence, we set
	\[
	\RF_a(w^1)=-w^1=(1,1,1), \quad \mbox{and} \quad \RF_a(w^3)=-w^3=(1,-1,1).
	\]
	For $w^2$, we have $F^{-1}(w^2)=v^1=(1,1,1) \in \mathrm{Image}(F)$, and thus, we move forward to $F^{-2}(w^2)=F^{-1}(v^1)=v^4=(1,-1,-1) \notin \mathrm{Image}(F)$. Hence, we define
	\[
	\RF_a(w^2)=-F^{-2}(w^2)=(-1,1,1).
	\]
	Similarly, for $w^4$ we have $F^{-2}(w^4) = (1,-1,1) \notin \mathrm{Image}(F)$, and thus, we define
	\[
	\RF_a(w^4)=-F^{-2}(w^4)=(-1,1,-1).
	\]
	It is easy to see that $\RF_a$ is indeed a bijection that satisfies the assertion of Lemma~\ref{lem:recursive-flip}.
\end{example}

\subsection{Proof of Lemma~\ref{lem:cheby}}\label{ssec:cheby}

	Let $X$ be a symmetric (around $0$) random variable with $\var(X)=1$. Observe that:
	\begin{equation}\label{eq:basic_cheby}
		\be \li[ (1-X^{2}) \one_{\{|X| < 1\}} \ri] = \be \li[ (X^{2}-1) \one_{\{|X| > 1\}} \ri].
	\end{equation}
	The equality~\eqref{eq:basic_cheby} holds, since by linearity of expectation,
	\[
	\be \li[ (1-X^{2}) \one_{\{|X| < 1\}} \ri] - \be \li[ (X^{2}-1) \one_{\{|X| > 1\}} \ri] = \be \li[ (1-X^{2}) \one_{\{|X| < 1\} \cup \{|X| > 1\}} \ri] = \be \li[ (1-X^{2})\ri] = 0,
	\]
	where the ultimate equality is satisfied as $\be[X^2]=\var(X) + \be[X]^2=1+0^2=1$.
	
	Notice that~\eqref{eq:basic_cheby} immediately implies Chebyshev's inequality. Indeed, by upper bounding the left hand side of~\eqref{eq:basic_cheby} with $\pr[|X| < 1]$ and lower bounding the right hand side of~\eqref{eq:basic_cheby} by $(t^2-1) \pr[|X| \geq t]$ for any $t \geq 1$, one obtains $\pr[|X| < 1] \geq (t^2-1) \pr[|X| \geq t]$, or equivalently,
		\[
			(t^2-1)\pr[|X| \geq t] + \pr[|X| \geq 1] \leq 1,
		\]
	yielding the Chebyshev bound $\pr[|X| \geq t] \leq 1/t^2$.
	
	\medskip It is however clear that there is extra freedom in the derivation of this inequality. Specifically, for any set of real numbers
		\[
		0 = c_0 \leq c_1 \leq \ldots \leq c_n = 1 = d_0 \leq d_1 \leq \ldots \leq d_m \leq d_{m+1}=\infty,
		\]
		we can upper bound the left hand side of~\eqref{eq:basic_cheby} as
		\begin{equation}\label{eq:cheby_lhs}
			\be \li[ (1-X^{2}) \one_{\{|X| < 1\}} \ri] \leq \sum_{i=0}^{n-1} (1-c_i^2) \pr[|X| \in [c_i, c_{i+1})],
		\end{equation}
		and lower bound the right hand side as
		\begin{equation}\label{eq:cheby_rhs}
			\be \li[ (X^{2}-1) \one_{\{|X| > 1\}} \ri] \geq \sum_{i=1}^{m} (d_i^2-1) \pr[|X| \in [d_i,d_{i+1})] = \sum_{i=1}^{m} (d_i^2-d_{i-1}^2) \pr[|X| \geq d_i].
		\end{equation}
		Using the symmetry of $X$, we may deduce
		\begin{equation}
		\sum_{i=0}^{n-1} (1-c_i^2) \hpr{X}{c_i}{c_{i+1}} \geq \sum_{i=1}^{m} (d_i^2-1) \hpr{X}{d_i}{d_{i+1}},
		\end{equation}
		and similarly,
		\begin{equation}
		\sum_{i=0}^{n-1} (1-c_i^2) \hpr{X}{c_i}{c_{i+1}} \geq \sum_{i=1}^{m} (d_i^2-d_{i-1}^2) \pr[X \geq d_i],
		\end{equation}
		proving Lemma~\ref{lem:cheby}.

\subsection{Comparison lemmas for other types of segments}
\label{app:sub:other-types}

One may consider $9$ types of segments $T(a,b)$, characterized by having each of its ends $a,b$ open, closed, or semi-open (the latter meaning that the probability at that end counts as $1/2$, which we denote by $\langle$ or $\rangle$).
		
\begin{itemize}
	\item It turns out that if $\hseg{C}{D} \prec_X \hseg{A}{B}$ (that is, if $A,B,C,D$ satisfy the assumptions of either Lemma~\ref{lem:seg-compare1} or Lemma~\ref{lem:seg-compare2}), then for any segment type $T$ we have
			\[
			\pr[X \in T(C,D)] \leq \pr[X \in T(A,B)].
			\]
			To see this, one can track the proofs of Lemmas~\ref{lem:seg-compare1} and~\ref{lem:seg-compare2} and notice that all injections we use, map $v$ to $w$ in such a way that if $X(v) \in (C,D)$ then $X(w) \in (A,B)$, if $X(v)=C$ then $X(w)\in [A,B)$, and if $X(v)=D$ then $X(w) \in (A,B]$. The only exception is in Lemma~\ref{lem:seg-compare1} in the case $A < B < -A$, which needs a separate handling. In this case, the proof yields that if $X(v) \in [C,D)$ then $X(w) \in (A,B)$ and if $X(v)=D$ we might have $X(w) = A$. This can be fixed by tweaking $\PF_{a,Q}$ to flip the first prefix \emph{strictly} exceeding $Q/2$ (instead of being greater or equal).
			
			In this paper, we use the current observation only through Lemma~\ref{lem:seg-compare2} to assert $\pr[X \in (C,D)] \leq \pr[X \in (A,B)]$ in the proof of Lemma~\ref{lem:mid-ai} (in particular, the `exception case' is not used in the paper).
			
			\item Sometimes, we even have
			\[
			\pr[X \in S(C,D)] \leq \pr[X \in T(A,B)]
			\]
			for different segment types $S,T$, as is demonstrated in~\eqref{eq:semi-compare}.

		\end{itemize}

\section{Proofs for Section~\ref{sec:BE}}
\label{app:no-numeric}

\subsection{Proof of Lemma~\ref{lem:ad-hoc}}\label{ssec:ad-hoc}
	\begin{proof}[Proof of Lemma~\ref{lem:ad-hoc}]
		Since in the proof we use the value of $a_1$ only through substitution in Proposition~\ref{prop:Prawitz-Rademacher}, we may assume $a_1 = 0.22$ rather than $a_1 \leq 0.22$ (as otherwise, the lower bound only gets better).
		
		Consider first the range $x \geq 1.65$. Applying Proposition~\ref{prop:Prawitz-Rademacher} with $a_1=0.22$, $x=1.65$, $T = 14.5$, and $q=0.4$, we obtain
		\begin{equation}\label{eq:165}
			\pr[Z < 1.65] - \pr[X < 1.65] < 0.0314,
		\end{equation}
		and consequently, $\pr[X < 1.65] > 0.919$. In particular, for all $x \geq 1.65$ we have
		\[
		\pr[X \leq x] > 1 - 0.084 \geq \pr[Z \leq x] - 0.084,
		\]
		as claimed in the lemma.
		
		Now, let $x \in [0.35, 1.65]$. Instead of delicately analysing the bound~\eqref{Eq:Our-Prawitz} as $x$ varies, it is sufficient to show that for some finite sequence $\{x_i\}_{i=0}^{m}$ with $0.35 = x_0 < x_1 < \ldots < x_m = 1.65$,~\eqref{Eq:Our-Prawitz} implies $\pr[Z < x_{i+1}] - \pr[X < x_i] < 0.084$. Indeed, this clearly gives
		\[
		\pr[Z \leq x] - \pr[X \leq x] < 0.084, \qquad  \forall x \in [x_i, x_{i+1}].
		\]
		Since
		\begin{equation}\label{eq:ad-hoc-proof}
		\pr[Z < x_{i+1}] - \pr[X < x_i] = \pr[Z \in [x_i, x_{i+1})] + (\pr[Z < x_i] - \pr[X < x_i]),
		\end{equation}
		we may construct such a sequence $\{x_i\}_{i=0}^{m}$ by starting with $x_0=0.35$, and for each $i \geq 0$, choosing $x_{i+1}>x_i$ sufficiently close to $x_i$, so that
		\begin{equation}\label{eq:ad-hoc-proof2}
		\pr[Z \in [x_i, x_{i+1})] < 0.084 - (\pr[Z < x_i] - \pr[X < x_i]).
		\end{equation}
		This would readily imply $\pr[Z < x_{i+1}] - \pr[X < x_i] < 0.084$ by~\eqref{eq:ad-hoc-proof}.
		
		\medskip
		A computer check\footnote{A computer program verifying~\eqref{eq:ad-hoc-proof2} for the 93 elements of the sequence $\{x_i\} $, as well as the inequalities~\eqref{eq:0.31} and~\eqref{eq:165}, is provided in  \url{https://github.com/IamPoosha/tomaszewski-problem/blob/master/formal_verification.py}. All inequalities hold with a spare of at least $2\cdot 10^{-5}$, allowing tolerating precision errors of up to $\pm 10^{-5}$.},
		in which $\pr[Z < x_i] - \pr[X < x_i]$ is bounded from above using Proposition~\ref{prop:Prawitz-Rademacher}, applied with $a_1=0.22$, $T=14.5$, $q=0.4$, and $x=x_i$, shows that~\eqref{eq:ad-hoc-proof2} is satisfied for the sequence
		\begin{equation}
		\begin{aligned}
			\{x_i\}_{i=0}^{93} =
			(
			& 0.35, 0.358, 0.366, 0.374, 0.38, 0.386, 0.39, 0.395, 0.399, 0.403, 0.406, 0.409, 0.412, \\
			& 0.415, 0.417, 0.419, 0.421, 0.423, 0.425, 0.427, 0.428, 0.429, 0.43, 0.431, 0.432, 0.433, \\
			& 0.434, 0.435, 0.436, 0.437, 0.438, 0.439, 0.44, 0.441, 0.442, 0.443, 0.444, 0.445, 0.446, \\
			& 0.447, 0.448, 0.449, 0.45, 0.451, 0.452, 0.453, 0.454, 0.455, 0.456, 0.457, 0.458, 0.459, \\
			& 0.46, 0.461, 0.462, 0.463, 0.464, 0.466, 0.468, 0.47, 0.472, 0.474, 0.476, 0.478, 0.481, \\
			& 0.484, 0.487, 0.49, 0.494, 0.499, 0.504, 0.51, 0.517, 0.526, 0.537, 0.55, 0.567, 0.589, \\
			& 0.61, 0.63, 0.65, 0.67, 0.69, 0.71, 0.73, 0.76, 0.8, 0.85, 0.91, 0.98, 1.07, 1.2, 1.38, 1.65
			).
		\end{aligned}
		\end{equation}
		This completes the proof.
		\end{proof}


\subsection{Numeric integration in our proofs}

In the proofs of Proposition~\ref{lem:be_toma} and of Lemma~\ref{lem:ad-hoc}, we evaluate multiple times the right hand side of~\eqref{Eq:Our-Prawitz}, with various parameters. The proof requires precision of $\pm 10^{-5}$ in the results of the evaluations.

The \texttt{Python3} program we used for the evaluation (provided in  \url{https://github.com/IamPoosha/tomaszewski-problem/blob/master/formal_verification.py}) uses the open-source procedure \texttt{scipy.integrate.quad}, wrapping standard adaptive integrators from the \textsc{QUADPACK} library, which estimates the total integration error as being below $10^{-10}$. In addition, there are implicit numerical errors of order $2^{-53} \approx 10^{-16}$, caused by the finite precision of double-precision floating point numbers. These estimates are well below the allowed $\pm 10^{-5}$ error, which we obtain by requiring all inequalities to hold with a `safety margin' of $2\cdot 10^{-5}$. As the functions involved in the numerical integration do not oscillate excessively in the integration range, relying on numeric integration in our scenario is standard.

Nevertheless, since the numeric integration is not a 100\% rigorous proof, the accompanied program makes also a slower, more straightforward evaluation of the integrals via Riemann sums. In this subsection we analyze the error of this evaluation and show it is bounded, as required, by $10^{-5}$.
	
	\medskip
	Given a piecewise-differentiable function $f(u)$ defined on a finite domain $(a,b)$, together with a bound on its derivative $|f'(u)| \leq B$ for all $u\in (a,b)$, it is easy to check that for any $N \in \pintegers$,
	\begin{equation}\label{eq:integral}
	\li| \int_{a}^{b} f(u) \dd{u} - \frac{b-a}{N}\cdot \sum_{k=1}^{N} f\li(a + \frac{2k-1}{2N}(b-a)\ri) \ri| \leq \frac{B(b-a)^2}{4N}.
	\end{equation}
	By choosing $N$ large enough (as a function of $B$), one may evaluate $\int_{a}^{b} f(u) \dd{u}$ using~\eqref{eq:integral}, with any required precision rate. We henceforth compute a bound $B$ associated with the integrals involved in~\eqref{Eq:Our-Prawitz}, and subsequently, re-compute the evaluations of~\eqref{Eq:Our-Prawitz}, using~\eqref{eq:integral} with a sufficiently large $N$.
	
	\medskip
	For the sake of convenience, let us recall~\eqref{Eq:Our-Prawitz}. We have to evaluate the right hand side of the inequality	\begin{equation}\label{eq:prawitz_again}
	\begin{aligned}
	\pr[Z < x] - \pr[X < x] \leq
	& (S_1' = ) \int_{0}^{q} |k(u,x,T)| g(Tu) \dd{u} + \\
	& (S_2' = ) \int_{q}^{1} |k(u,x,T)| h(Tu) \dd{u} + \\
	& (S_3' = ) \int_{0}^{q} k(u,x,T) \exp(-(Tu)^2/2) \dd{u} + \\
	& (S_4' = ) \int_{0}^{x} \frac{1}{\sqrt{2\pi}}\exp(-u^2/2) \dd{u},
	\end{aligned}
	\end{equation}
	where $k(u,x,T) = \frac{(1-u)\sin(\pi u - T u x)}{\sin(\pi u)} - \frac{\sin(T u x)}{\pi}$,
	\[
	g(v) =
	\begin{cases}
	\exp(-v^{2}/2) - \cos(a_1 v) ^ {1/a_1^2}, & a_1 v \leq \frac{\pi}{2}\\
	\exp(-v^{2}/2)+1, & \mrm{otherwise}
	\end{cases}
	, \quad h(v) =
	\begin{cases}
	\exp(-v^{2}/2),					& a_1 v \leq \theta\\
	(-\cos(a_1 v))^{1/a_1^2},	    & \theta \leq a_1 v \leq \pi\\
	1,								& \mrm{otherwise}
	\end{cases}.
	\]
	In the proof of Proposition~\ref{lem:be_toma}, we substitute into~\eqref{eq:prawitz_again} the parameters $a_1=0.31$, $x=1$, $T=10$, and $q=0.4$.	In the proof of Lemma~\ref{lem:ad-hoc}, the parameters are $a_1 = 0.22$, $T=14.5$, $q=0.4$, and various values $x \in [0, 1.65]$. In order to estimate the integrals $S_1', S_2', S_3', S_4'$ by the method described above, we upper bound the absolute value of the derivatives of the functions appearing in~\eqref{eq:prawitz_again}, for our choices of the parameters. The derivative bound corresponding to $S_i'$ is called $B_i$.
	
	\paragraph{Regarding $B_4$} One can easily check that the derivative $\frac{\dd{}}{\dd{u}} \exp(-u^2/2)/\sqrt{2\pi}$ is negative for $u \geq 0$, and is minimized at $u=1$. This gives the bound
	\[
	B_4 \leq \exp(-1/2)/\sqrt{2\pi} < 1/4.
	\]
	
	\paragraph{Regarding $B_3$} In order to give a bound $B_3$ on the derivative, we use $(fg)' = f' g + f g'$, implying $|(fg)'| \leq |f'||g| + |f||g'|$. We clearly have $|\exp(-(Tu)^2/2)| \leq 1$, and by the argument of the previous case (applied with $Tu$ in place of $u$),
	\[
	\Big|\frac{\dd{}}{\dd{u}} \exp(-(Tu)^2/2)\Big| \leq \exp(-1/2)T \leq 2T/3.
	\]
	To bound $k(u,x,T)$, recall the simplification~\eqref{eq:k_def} and notice that as $|\sin(v)| \leq |v|$ for all $v \in \reals$, and $\pi u (1-u) \leq \sin(\pi u)$ when $u \in [0, 1]$, we have
	\[
	|k(u, x, T)| \leq \frac{|\pi - T x|}{\pi} + \frac{T u x}{\pi} \leq 1 + 2Tx/\pi.
	\]
	To bound $|\frac{d}{du} k(u, x, T)|$, we write $s=\pi - Tx$, and obtain
	\begin{align}\label{Eq:Aux-Numeric1}
	\begin{split}
	\frac{\dd{}}{\dd{u}} k(u, x, T) + \frac{Tx}{\pi}\cos(T u x) =
	\ &
	\frac{(1-u) s \cos(s u)}{\sin(\pi u)} - \frac{(1-u) \pi \cos(\pi u) \sin(s u)}{\sin(\pi u)^2} - \frac{\sin(s u)}{\sin(\pi u)}
	= \\
	&=
	\frac{u(1-u)}{\sin(\pi u)} \bigg( s^2 \cdot \frac{s u \cos(s u) - \sin(s u)}{(s u)^2} + \\
	& \qquad +
	s \cdot \frac{\sin(s u)}{s u}\cdot \frac{(1-2u)\sin(\pi u) - \pi u(1-u) \cos(\pi u)}{u(1-u)\sin(\pi u)} \bigg).
	\end{split}
	\end{align}
	It can be verified that for any $u \in [0,1]$ we have
	\[
	|(1-2u)\sin(\pi u) - \pi u(1-u) \cos(\pi u)| \leq u(1-u)\sin(\pi u) \qquad \mbox{and} \qquad \pi u (1-u) \leq \sin(\pi u),
	\]
	and for all $v \in \reals$, we have
	\[
	|v\cos(v)-\sin(v)| \leq v^2 / 2 \qquad \mbox{and} \qquad |\sin(v)| \leq |v|.
	\]	
	Plugging these inequalities into~\eqref{Eq:Aux-Numeric1}, we obtain the bound
	\[
	\Big|\frac{\dd{}}{\dd{u}} k(u, x, T)\Big| \leq \frac{1}{\pi}(|s|^2/2 + |s| + Tx) \leq \frac{1}{\pi}((Tx)^2/2 + \pi^2) = (Tx)^2/(2\pi) + \pi.
	\]
	We conclude
	\begin{align*}
	B_3 &\leq \max_{u \in [0,q]} \left(|k(u,x,T)| \Big|\frac{d}{du}\exp(-(Tu)^2/2)\Big| + \Big|\frac{d}{du}k(u,x,T)\Big| \exp(-(Tu)^2/2) \right) \leq  \\
	&\leq
	(1+2Tx/\pi) \cdot 2T/3 + ((Tx)^2/(2\pi) + \pi) \cdot 1.
	\end{align*}
	
	\paragraph{Regarding $B_2$} Using the previous bounds, together with the bounds $|h(Tu)| \leq 1$ and
	\[
	\frac{\dd{}}{\dd{u}}h(v) =
	\begin{cases}
	-v \exp(-v^2/2),								& a_1 v < \theta\\
	(-\cos(a_1 v))^{1/(a_1^2)-1} \sin(a_1 v) / a_1,	& \theta < a_1 v < \pi\\
	0,												& \pi < a_1 v
	\end{cases}
	\qquad \implies \qquad \li|\frac{\dd{}}{\dd{u}}h(Tu)\ri| \leq T,
	\]
	we conclude
	\[
	B_2 \leq (1+2Tx/\pi) \cdot T + ((Tx)^2/(2\pi) + \pi) \cdot 1.
	\]
	The inequality $|\cos(u)|^{1/(a_1^2)-1} \sin(u) \leq 1$ we use, holds for any $u \in \reals$ and $a_1 \in (0,1)$. This can be checked easily for $a_1 \in \{0.22, 0.31\}$, which are the only cases we need.
	
	\paragraph{Regarding $B_1$} Using the previous bounds, together with the bounds
	\[
	|g(Tu)| \leq 1+\exp(-(\pi/(2a_1))^2/2) < 1.1
	\]
	and
	\[
	\frac{\dd{}}{\dd{u}}g(v) =
	\begin{cases}
	\cos(a_1 v)^{1/(a_1^2)-1} \sin(a_1 v) / a_1-v \exp(-v^2/2),	& a_1 v < \pi / 2\\
	-v \exp(-v^2/2),				& \pi / 2 < a_1 v\\
	\end{cases}
	\quad \implies \quad \li|\frac{\dd{}}{\dd{u}}g(Tu)\ri| \leq T,
	\]
	we conclude
	\[
	B_1 \leq (1+2Tx/\pi) \cdot T + ((Tx)^2/(2\pi) + \pi) \cdot 1.1.
	\]

	\medskip The aforementioned \texttt{Python3} program  verifies~\eqref{eq:0.31},~\eqref{eq:165}, and the inequality~\eqref{eq:ad-hoc-proof2} for the 93 elements of the sequence $\{x_i\} $, via~\eqref{eq:integral} with the bounds $B_1, B_2, B_3, B_4$ given above. In all cases, the inequalities hold with a safety margin (i.e., difference between the two sides) of at least $2 \cdot 10^{-5}$. This allows the program to estimate the integrals to within $10^{-5}$ (total) additive error, while leaving a spare error bound of $10^{-5}$ for any other numeric inaccuracies occurring outside of the numeric integration procedure \texttt{integrate} implementing~\eqref{eq:integral}. The running time of this verification is less than 10 minutes on a modern machine.

\section{Theorem~\ref{thm:main} for \tops{$a_1 \in (0.31, 0.5)$}, \tops{$a_1+a_2+a_3 \leq 1$}}\label{sec:31big}

The proof of Theorem~\ref{thm:main} in this range splits into three cases. If $a_2$ is sufficiently small, elimination of one variable and the improved Berry-Esseen bound of Section~\ref{sec:BE} are sufficient for proving the assertion. If $a_3$ is sufficiently small, elimination of two variables and the improved Berry-Esseen bound do the work. The hard case is when both $a_2$ and $a_3$ are not very small. To handle this case, we eliminate three variables and combine segment comparison with a Chebyshev-type inequality. Unfortunately, the proof is quite long and rather cumbersome.

\subsection{The case \tops{$a_2 \leq 0.19$}}
	We begin with eliminating one variable (i.e., applying Lemma~\ref{lem:nm} with $m=1$). The lemma implies that in order to derive $\pr[|X|\leq 1] \geq 1/2$, it is sufficient to prove
	\begin{equation}\label{Eq:Aux8.1}
	\pr[X'>t] + \pr[X' > 1/t] \leq 1/2,
	\end{equation}
	where $X',t$ are defined as in~\eqref{Eq:Aux6.1}.
	
	The maximal weight of $X'$ is $a_2'$, which satisfies $a_2' \leq 0.19 / \sqrt{3/4} < 0.2195$. Hence, applying Proposition~\ref{prop:intro-BE} to $X'$ gives $\pr[X' > x] \leq \pr[Z > x] + 0.088$ for any $x \geq 0$. Thus,~\eqref{Eq:Aux8.1} follows from
	\begin{equation}\label{eq:31a2s}
	\begin{gathered}
		\pr[Z > t] + \pr[Z > 1/t] \isleq 1/2 - 2\cdot 0.088,\\
		\text{with:}\qquad a_1 \in (0.31, 0.5), \quad t=\sqrt{\frac{1-a_1}{1+a_1}},\quad Z\sim N(0,1).
	\end{gathered}
	\end{equation}
	This inequality is verified in Appendix~\ref{app:31a2s}. Note that the proof in this case does not use the assumption $a_1+a_2+a_3 \leq 1$. Thus, the assertion holds in the case $(a_1 \in (0.31,0.5)) \wedge (a_2 \leq 0.19) \wedge (a_1+a_2+a_3 > 1)$ as well.

\subsection{The case \tops{$a_2 \geq 0.19$}, \tops{$a_3 \leq 0.15$}}\label{ssec:a3015}
	We begin with eliminating two  variables. By Lemma~\ref{lem:nm}, applied with $m=2$, in order to prove $\pr[|X| \leq 1] \geq 1/2$ for $X = \sum_{i=1}^{n} a_{i} x_{i}$, it is sufficient to verify
	\begin{equation}\label{eq:E2}
		\sum_{k=1}^{4} \pr[X' > T_k] \leq 1,
	\end{equation}
	with $X' = \sum_{i=3}^{n} a_i' x_i$, $\sigma=\sqrt{1-a_1^2-a_2^2}$, $a_i' = a_i / \sigma$ and $T_k = (1\pm a_1 \pm a_2) / \sigma$ as in Lemma~\ref{lem:nm}.
	The maximal weight of $X'$ is $a_3'$, which satisfies $a_3' \leq 0.15 / \sqrt{1/2} = 0.15 \sqrt{2}$. Hence, an application of Proposition~\ref{prop:ad-hoc2} to $X'$ gives
	\begin{align*}
	\pr[X' \leq x] &\geq \pr[Z \leq x] - \max(0.084, \pr[|Z| \leq a'_3]/2) \\
	&\geq \pr[Z \leq x] - \max(0.084, \pr[|Z| \leq 0.15 \sqrt{2}]/2) \\
	& = \pr[Z \leq x] - 0.084,
	\end{align*}
	and thus, $\pr[X' > x] \leq \pr[Z > x] + 0.084$, for all $x \geq 0$.
	As $T_k \geq 0$ for all $k$ (since $a_1+a_2 \leq 2a_1 \leq 1$),~\eqref{eq:E2} follows from the inequality
	\begin{equation}\label{eq:31a3s}
	\begin{gathered}
		\sum_{k=1}^{4} \pr[Z > T_k] \isleq 1 - 4\cdot 0.084, \\
		\text{with:}\quad Z\sim N(0,1), \quad 0.19 \leq a_2 \leq a_1 \leq 1/2, \quad \sigma = \sqrt{1-a_1^2-a_2^2}, \quad T_k = \frac{1\pm a_1 \pm a_2}{\sigma}.
	\end{gathered}
	\end{equation}
	This inequality is verified in Appendix~\ref{app:31a3s}. Note that like in the previous case, the proof does not use the assumption $a_1+a_2+a_3 \leq 1$. Thus, the assertion holds in the case $(a_1 \in (0.31,0.5)) \wedge (a_2 \geq 0.19) \wedge (a_3 \leq 0.15) \wedge (a_1+a_2+a_3 > 1)$ as well.

\subsection{The case \tops{$a_2 \geq 0.19$}, \tops{$a_3 \geq 0.15$}}
	
	\paragraph{Elimination step.} We begin with a variant of the `variable elimination lemma', eliminating three variables. This lemma will be used several more times in the sequel.
	Note that in our case, as $a_1 < 1/2$, we have $n > 4$, and thus, we may eliminate three variables and have $\sigma > 0$.
	
	\begin{lemma}\label{cor:nm3}
		Let $n \geq 4$. Let $X = \sum_{i=1}^{n} a_{i} x_i$ be a Rademacher sum with $\var(X)=1$, and write $X = a_1 x_1 + a_2 x_2 + a_3 x_3 + \sigma X'$, such that $\var(X') = 1$ (hence, $\sigma=\sqrt{1-a_1^2-a_2^2-a_3^2}$). The assertion
		\[
			\pr[|X| \leq 1] \geq 1/2
		\]
		is equivalent to the following inequality involving $X'$:
		\begin{equation}\label{eq:nm3}
			\begin{split}
			\pr[X' \in [-L_1, L_2]]
			& + \pr[X' \in [-L_3, L_4]] \\
			& \geq \\
			\pr\li[ X' > R_1\ri] + \pr\li[ X' > R_2\ri]
			&+ \pr\li[ X' > R_3\ri] + \pr\li[ X' > R_4\ri],
			\end{split}	
		\end{equation}
		where
		\[
			L_1, L_2, L_3, L_4 = \frac{1-a_1-a_2-a_3}{\sigma},
			\frac{1-a_1-a_2+a_3}{\sigma}, \frac{1-a_1+a_2-a_3}{\sigma}, \frac{1-|a_1-a_2-a_3|}{\sigma}
		\]
			and
		\[
			R_1, R_2, R_3, R_4 = \frac{1+|a_1-a_2-a_3|}{\sigma}, \frac{1+a_1-a_2+a_3}{\sigma},
			\frac{1+a_1+a_2-a_3}{\sigma}, \frac{1+a_1+a_2+a_3}{\sigma}.
		\]
	\end{lemma}
	\begin{proof}
		The assertion follows from Lemma~\ref{lem:nm}, applied to $X$ with the parameter $m=3$, and the equality
		\[
			\pr[X' \in [-L_1, L_2]] + \pr[X' \in [-L_3, L_4]]  = 2 - \sum_{j=1}^{4} \pr[X' > L_j].
		\]
	\end{proof}
	
	As we assume $a_1+a_2+a_3 \leq 1$, we have $L_1, L_2, L_3, L_4 \geq 0$, and thus, by symmetry of $X'$, the inequality~\eqref{eq:nm3} we have to prove follows from
	\begin{equation}\label{eq:sum3_less_1_req}
		\sum_{i=1}^{4} \hpr{X'}{0}{L_i} \geq \sum_{i=1}^{4} \hpr{X'}{R_i}{\infty} = \sum_{i=1}^{4} i \cdot \hpr{X'}{R_i}{R_{i+1}},
	\end{equation}
	with $R_5=\infty$.

	\paragraph{Auxiliary estimates.} To proceed, we use the following auxiliary estimates on $L_i,R_i$.
	\begin{equation}\label{eq:31check-cheby}
	\begin{gathered}
	L_3 \isgeq 0, \qquad L_4 \isgeq \sqrt{1/2},\qquad \forall i \in \{1,2,3,4\}\cc R_i \isgeq \sqrt{(1+i)/2},
	\\
	\text{with:}
	\qquad a_1 \in [0.31, 0.5],
	\quad a_2 \in [0.19, a_1],
	\quad a_3 \in [0.15, a_2],
	\quad a_1+a_2+a_3 \leq 1,
	\\
	\sigma = \sqrt{1-a_1^2-a_2^2-a_3^2},
	\\
	L_3 = \frac{1-a_1+a_2-a_3}{\sigma},
	\quad L_4 =	\frac{1-|a_1-a_2-a_3|}{\sigma},
	\quad R_1 = \frac{1+|a_1-a_2-a_3|}{\sigma},
	\\
	R_2 = \frac{1+a_1-a_2+a_3}{\sigma}
	\quad R_3 = \frac{1+a_1+a_2-a_3}{\sigma},
	\quad R_4 = \frac{1+a_1+a_2+a_3}{\sigma},
	\end{gathered}
	\end{equation}
	These inequalities are proved in Appendix~\ref{app:31check-cheby}.
	
	\paragraph{Chebyshev-type inequality step.} We reduce~\eqref{eq:sum3_less_1_req} into a more convenient inequality, by applying a Chebyshev-type inequality. Denote $\bar{L}_3=\min(L_3,1)$ and $\bar{L}_4=\min(L_4,1)$. Let
	\begin{equation}\label{eq:31cds}
	\begin{gathered}
		c_0, c_1, c_2, c_3, c_4 = 0, \bar{L}_3, \max(\bar{L}_3, \sqrt{1/2}), \bar{L}_4, 1
		\\
		d_0, d_1, d_2, d_3, d_4 = 1, R_1, \max(R_1, \sqrt{3/2}), R_2, \max(R_2, \sqrt{4/2})
		\\
		d_5, d_6, d_7, d_8, d_9 = R_3, \max(R_3, \sqrt{5/2}), R_4, \max(R_4, \sqrt{6/2}), \infty.
	\end{gathered}
	\end{equation}
	Note that by the definition of $\bar{L}_3,\bar{L}_4$ and~\eqref{eq:31check-cheby}, we have
	\[
	0 = c_0 \leq c_1 \leq c_2 \leq c_3 \leq c_4 = 1 = d_0 \leq d_1 \leq \ldots \leq d_8 \leq d_{9}=\infty.
	\]
	Hence, we may apply the Chebyshev-type inequality~\eqref{eq:cheby_our2} to $X'$, with the parameters $c_0,\ldots,c_4$ and $d_0,\ldots,d_9$, to obtain
	\begin{equation}\label{eq:31schebyshev}
	\begin{gathered}
	1 \cdot \hpr{X'}{0}{\bar{L}_3} +
	(1- \bar{L}_3^2 )\hpr{X'}{\bar{L}_3}{\max(\bar{L}_3,\sqrt{1/2})} + \\
	+(1-\max(\bar{L}_3,\sqrt{1/2})^2) \hpr{X'}{\max(\bar{L}_3,\sqrt{1/2})}{\bar{L}_4} + (1-\bar{L}_4^2)\hpr{X'}{\bar{L}_4}{1} \\ \geq \\
	\sum_{i=1}^4 \Big( (R_i^2-1)\hpr{X'}{R_i}{\max(R_i,\sqrt{(i+2)/2})} + \\
	+\big(\max \big(R_i,\sqrt{(i+2)/2}\big)^2-1\big) \hpr{X'}{\max(R_i,\sqrt{(i+2)/2})}{R_{i+1}} \Big),
	\end{gathered}
	\end{equation}
	with $R_5=\infty$. We now claim that the following inequality, coupled with~\eqref{eq:31schebyshev}, implies~\eqref{eq:sum3_less_1_req}.
	\begin{equation}\label{eq:31sneed}
	\begin{gathered}
	\hpr{X'}{0}{L_1} + \hpr{X'}{0}{L_2} + \\
	(2L_3^2-1)\hpr{X'}{L_3}{\sqrt{1/2}} +
	(2L_4^2-2)\hpr{X'}{L_4}{\sqrt{1}} + \\
	\geq \\
	(3-2R_1^2)\hpr{X'}{R_1}{\sqrt{3/2}} +
	(4-2R_2^2)\hpr{X'}{R_2}{\sqrt{2}} + \\
	(5-2R_3^2)\hpr{X'}{R_3}{\sqrt{5/2}} +
	(6-2R_4^2)\hpr{X'}{R_4}{\sqrt{3}}.
	\phantom{..}
	\end{gathered}
	\end{equation}
	To see this, note that if
	\begin{equation}\label{Eq:Aux8.2}
	0 \leq L_3 \leq \sqrt{1/2} \leq L_4 \leq 1 \leq R_1 \leq \sqrt{3/2} \leq R_2 \leq \sqrt{4/2} \leq R_3 \leq \sqrt{5/2} \leq R_4 \leq \sqrt{6/2},
	\end{equation}
	then~\eqref{eq:31schebyshev} reads
	\begin{align*}
	1\cdot \hpr{X'}{0}{L_3} +
	(1- L_3^2 & )\hpr{X'}{L_3}{\sqrt{1/2}} +
	\\
	\frac{1}{2}\hpr{X'}{\sqrt{1/2}}{L_4} +
	(& 1-L_4^2)\hpr{X'}{L_4}{1}
	\\ &\geq \\
	(R_1^2-1)     \hpr{X'}{R_1}{\sqrt{3/2}} &+
	\frac{1}{2}   \hpr{X'}{\sqrt{3/2}}{R_2} +
	\\
	(R_2^2-1)     \hpr{X'}{R_2}{\sqrt{2}} &+
	1             \hpr{X'}{\sqrt{2}}{R_3} +
	\\
	(R_3^2-1)     \hpr{X'}{R_3}{\sqrt{5/2}} &+
	\frac{3}{2}   \hpr{X'}{\sqrt{5/2}}{R_4} +
	\\
	(R_4^2-1)     \hpr{X'}{R_4}{\sqrt{3}} &+
	2             \hpr{X'}{\sqrt{3}}{\infty},
	\end{align*}
	and thus, we have
	\[
	2\eqref{eq:31schebyshev}+\eqref{eq:31sneed}=\eqref{eq:sum3_less_1_req}.
	\]
	If~\eqref{Eq:Aux8.2} is not satisfied, then $2\eqref{eq:31schebyshev}+\eqref{eq:31sneed}$ is even stronger than~\eqref{eq:sum3_less_1_req}, and implies it. To see this, observe that~\eqref{Eq:Aux8.2} can fail in three possible ways:
	\begin{itemize}
		\item \emph{$R_i > \sqrt{(i+2)/2}$ for some $1 \leq i \leq 4$.} In this case, the contribution of the region $\langle R_i, R_{i+1} \rangle$ to the right hand side of $2\eqref{eq:31schebyshev}+\eqref{eq:31sneed}$ becomes
		\begin{align*}
		2(R_i^2-1) \hpr{X'}{R_i}{R_{i+1}} &\geq 2((\sqrt{(i+2)/2})^2-1) \hpr{X'}{R_i}{R_{i+1}} \\
		&= i \cdot \hpr{X'}{R_i}{R_{i+1}},
		\end{align*}
		compared to $i \cdot \hpr{X'}{R_i}{R_{i+1}}$ in the right hand side of~\eqref{eq:sum3_less_1_req}, while the contribution of the other regions remains unchanged. Hence, in this case the inequality $2\eqref{eq:31schebyshev}+\eqref{eq:31sneed}$ implies~\eqref{eq:sum3_less_1_req}.
		
		\item \emph{$L_3>1$ or $L_4>1$.} If $L_3 > 1$, then the left hand side of $2\eqref{eq:31schebyshev}+\eqref{eq:31sneed}$ becomes
		\[
		\hpr{X'}{0}{L_1}+\hpr{X'}{0}{L_2}+2\hpr{X'}{0}{1},
		\]
		which is not larger than the left hand side of~\eqref{eq:sum3_less_1_req}, being $\sum_{i=1}^4 \hpr{X'}{0}{L_i}$. As the right hand side is unchanged,~\eqref{eq:sum3_less_1_req} follows in this case as well.

		\medskip Similarly, if $L_3 \leq 1<L_4$, then the contribution of the region $\langle \max(L_3,\sqrt{1/2}), L_4 \rangle$ to the left hand side of $2\eqref{eq:31schebyshev}+\eqref{eq:31sneed}$ becomes
		\[
		2(1-\max(L_3,\sqrt{1/2})^2) \hpr{X'}{\max(L_3,\sqrt{1/2})}{1},
		\]
		compared to
		$\hpr{X'}{\max(L_3,\sqrt{1/2})}{L_4}$ in the left hand side of~\eqref{eq:sum3_less_1_req}. The former is no larger than the latter, since $2(1-\max(L_3,\sqrt{1/2})^2) \leq 1$ and $1 \leq L_4$. The contribution of the other regions remains unchanged, so again~\eqref{eq:sum3_less_1_req} follows.
		
		\item $\sqrt{1/2} \leq L_3 \leq L_4 \leq 1.$
		In this case, the contribution of the region $\langle L_3, L_4 \rangle$ to the left hand side of $2\eqref{eq:31schebyshev}+\eqref{eq:31sneed}$ becomes
		\[
		2(1-L_3^2) \hpr{X'}{L_3}{L_4} \leq \hpr{X'}{L_3}{L_4},
		\]
		compared to
		$\hpr{X'}{L_3}{L_4}$ in the left hand side of~\eqref{eq:sum3_less_1_req}. The former is no larger than the latter, since $2(1-L_3^2) \leq 1$. The contribution of the other regions remains unchanged, so~\eqref{eq:sum3_less_1_req} follows.
	\end{itemize}
	Therefore, in order to prove~\eqref{eq:sum3_less_1_req}, it is sufficient to show~\eqref{eq:31sneed}, which we rewrite in the form
	\begin{equation}\label{eq:31sneed2}
		\frac{1}{2}\hpr{X'}{-L_1}{L_1} + \hpr{X'}{0}{L_2} \geq \sum_{i=1}^{6} c_i \hpr{X'}{d_i}{e_i},
	\end{equation}
	with \[d_1,d_2,d_3,d_4,d_5,d_6=L_3,L_4,R_1,R_2,R_3,R_4, \qquad c_i = i-2d_i^2, \qquad e_i=\sqrt{i/2}.
	\]
	Notice we might have $e_i \leq d_i$, in which case $\hpr{X'}{d_i}{e_i}=0$.

	\paragraph{Segment comparison step.} In order to prove~\eqref{eq:31sneed2}, we use the following lemma, which is proved in Appendix~\ref{app:31sfin}:
	\begin{lemma}\label{lem:31sfin}
		Let $X',L_i,R_i, c_i,d_i,e_i$ be defined as above. We have:
		\begin{enumerate}
			\item $\forall i\cc \hseg{d_i}{e_i} \prec_{X'} \hseg{0}{L_2}$, and thus Theorem~\ref{thm:seg_compare} implies: $\hpr{X'}{d_i}{e_i} \leq \hpr{X'}{0}{L_2}$.
			\item $\sum_{i=1}^{6} \max(c_i, 0) \leq 3/2$.
			\item $\sum_{i\in B(X')} \max(c_i, 0) \leq 1$, where $B(X') = \li\{\given{i \in [6]}{\hseg{d_i}{e_i} \not \prec_{X'} \hseg{-L_1}{L_1}}\ri\}$.
		\end{enumerate}
	\end{lemma}
\medskip\noindent
	Observe that~\eqref{eq:31sneed2} is trivially implied by combination of the following two inequalities:
	\begin{equation}\label{eq:31saux1}
		\sum_{i \in B(X')} c_i \hpr{X'}{d_i}{e_i} \leq C \cdot \hpr{X'}{0}{L_2},
	\end{equation}
	with $C = \sum_{i \in B(X')} \max(c_i,0)$, and
	\begin{equation}\label{eq:31saux2}
		\sum_{i \not\in B(X')} c_i \hpr{X'}{d_i}{e_i} \leq (1-C)\hpr{X'}{0}{L_2} + \frac{1}{2}\hpr{X'}{-L_1}{L_1}.
	\end{equation}
	We deduce both inequalities from  Lemma~\ref{lem:31sfin}.
	Inequality~\eqref{eq:31saux1} follows immediately from the first item of Lemma~\ref{lem:31sfin}.
	We now reason about~\eqref{eq:31saux2}.
	
	Let $p = \min(\hpr{X'}{0}{L_2}, \hpr{X'}{-L_1}{L_1})$. Notice that the first item of Lemma~\ref{lem:31sfin}, together with the definition of $B(X')$, implies that
	\[
	\forall i \notin B(X'): \qquad \hpr{X'}{d_i}{e_i} \leq p.
	\]
	Hence, we may deduce~\eqref{eq:31saux2}, as
	\begin{equation}
	\begin{aligned}
		\sum_{i \not\in B(X')} c_i \hpr{X'}{d_i}{e_i} &\leq \sum_{i \not\in B(X')} \max(c_i, 0) \cdot p
		\stackrel{\mrm{(a)}}{\leq} (1-C)p + \frac{1}{2}p \leq
		\\
		& \stackrel{\mrm{(b)}}{\leq} (1-C)\hpr{X'}{0}{L_2} + \frac{1}{2}\hpr{X'}{-L_1}{L_1},
	\end{aligned}
	\end{equation}
	where inequality~(a) uses the second item of Lemma~\ref{lem:31sfin} and the definition of $C$, and inequality~(b) follows from the definition of $p$, via $1-C \geq 0$, which is a rephrasing of the third item of Lemma~\ref{lem:31sfin}. This completes the proof.


\section{Theorem~\ref{thm:main} for \tops{$a_1 \in [0.387, 0.5]$}, \tops{$a_1 + a_2 + a_3 \geq 1$}}\label{sec:39geq}

The proof of Theorem~\ref{thm:main} in this range is similar to -- but much easier than -- the proof in the range $a_1 \in (0.31,0.5), a_1+a_2+a_3 \leq 1$, presented in Appendix~\ref{sec:31big}. If $a_3$ is sufficiently small, then the Berry-Esseen argument of Appendix~\ref{sec:31big} yields the assertion. If $a_3$ is not very small, the proof goes by elimination of three variables and combination of Chebyshev-type inequalities and segment comparison. The difference from Appendix~\ref{sec:31big} is that the slightly different assumptions on the parameters allow for a simple Chebyshev-type argument to work. A similar argument fails in the range of Appendix~\ref{sec:31big} (specifically, the inequality~\eqref{eq:39gglue} below, does not hold there), and thus, an exhausting detour is needed.

\subsection{The case \tops{$a_3 \leq 0.15$}}
	
	By assumption, we have $a_1+a_2+a_3 \geq 1$ and $a_3 \leq a_2 \leq a_1 \leq 0.5$, and thus,
	$a_2 \geq 0.25$. Hence, if $a_3 \leq 0.15$ then the assertion $\pr[|X| \leq 1] \geq 1/2$ follows from the argument of Appendix~\ref{ssec:a3015}, which applies whenever $(a_1 \in (0.31,0.5)) \wedge (a_2 \geq 0.19) \wedge (a_3 \leq 0.15)$, as was shown at the end of Appendix~\ref{ssec:a3015}.

\subsection{The case \tops{$a_3 \geq 0.15$}}
	
	\paragraph{Elimination step.} By Lemma~\ref{cor:nm3}, it is sufficient to prove
	\begin{equation}\label{eq:39gneed}
	\begin{gathered}
	\hpr{X'}{-L_1}{L_2} + \hpr{X'}{0}{L_3} + \hpr{X'}{0}{L_4}
	\\
	\geq
	\\
	\hpr{X'}{R_1}{\infty} + \hpr{X'}{R_2}{\infty} + \hpr{X'}{R_3}{\infty} + \hpr{X'}{R_4}{\infty},
	\end{gathered}
	\end{equation}
	where $\sigma, X', L_1,L_2,L_3,L_4,R_1,R_2,R_3,R_4$ are as defined in Lemma~\ref{cor:nm3}.
	
	\paragraph{Auxiliary estimate.} To proceed, we use the following auxiliary inequality in $L_i,R_i$.
	\begin{equation}\label{eq:39gglue}
	\begin{gathered}
	\max\li\{ 1-L_3^2, 3/2-L_4^2, 1/2 \ri\} \isleq \min \li\{ \frac{R_1^2-1}{1}, \frac{R_2^2-1}{2}, \frac{R_3^2-1}{3}, \frac{R_4^2-1}{4} \ri\},\\
	\text{with:}
	\qquad a_1 \in [0.387, 0.5],
	\quad 0.15 \leq a_3 \leq a_2 \leq a_1,
	\quad a_1+a_2+a_3 \geq 1,
	\\
	\sigma = \sqrt{1-a_1^2-a_2^2-a_3^2},
	\\
	L_3 = \frac{1-a_1+a_2-a_3}{\sigma},
	\quad L_4 =	\frac{1+a_1-a_2-a_3}{\sigma},
	\quad R_1 = \frac{1-a_1+a_2+a_3}{\sigma},
	\\
	R_2 = \frac{1+a_1-a_2+a_3}{\sigma}
	\quad R_3 = \frac{1+a_1+a_2-a_3}{\sigma},
	\quad R_4 = \frac{1+a_1+a_2+a_3}{\sigma}.
	\end{gathered}
	\end{equation}
	The proof of~\eqref{eq:39gglue} is presented in Appendix~\ref{app:39gglue}.
	
	\paragraph{Chebyshev-type step.}
	Let
	\begin{equation}
	\begin{gathered}
	c = 1/\max\li\{ 1-L_3^2, 3/2 - L_4^2, 1/2 \ri\}, \qquad \mbox{and}\\
	d = 1/\min \li\{ R_1^2-1, (R_2^2-1)/2, (R_3^2-1)/3, (R_4^2-1)/4 \ri\}.
	\end{gathered}
	\end{equation}
	We show that
	\begin{equation}\label{eq:39glhs}
		\text{LHS of~\eqref{eq:39gneed}} \geq \frac{c}{2} \cdot \be[(1-X'^2) \one\{|X'| < 1\}],
	\end{equation}
	and
	\begin{equation}\label{eq:39grhs}
		\text{RHS of~\eqref{eq:39gneed}} \leq \frac{d}{2} \cdot \be[(X'^2-1) \one\{|X'| > 1\}].
	\end{equation}
	As by~\eqref{eq:39gglue} we have $c \geq d$, the assertion~\eqref{eq:39gneed} follows from~\eqref{eq:39glhs} and~\eqref{eq:39grhs} by the Chebyshev-type equality~\eqref{eq:basic_cheby}.
	
	\paragraph{Proving~\tops{\eqref{eq:39glhs}}.} We consider three sub-cases.
	
	\medskip \noindent \textbf{Sub-case~1: $L_4 \geq L_3 \geq 1$.} In this case, it is clear that
	\[
	\text{LHS of~\eqref{eq:39gneed}} \geq
	2\hpr{X'}{0}{1} \geq \be[(1-X'^2) \one\{|X'| < 1\}],
	\]
	implying~\eqref{eq:39glhs} (by noting $c/2=1$).
	
	\medskip \noindent \textbf{Sub-case~2: $L_4 \geq 1 > L_3$.} In this case, the Chebyshev-type inequality~\eqref{eq:cheby_lhs} with $c_0,c_1,c_2=0, L_3, 1$ yields
	\[
		2\hpr{X'}{0}{L_3} + 2(1-L_3^2)\hpr{X'}{L_3}{1}
		\geq \be[(1-X'^2) \one\{|X'| < 1\}].
	\]
	The left hand side can be bounded from above by
	\[
	\max(1, 2(1-L_3^2))\cdot \li( \hpr{X'}{0}{L_3} + \hpr{X'}{0}{L_4}\ri),
	\]
	and thus~\eqref{eq:39glhs} follows (by noting $c/2=1/\max(1, 2(1-L_3^2))$).
	
\medskip \noindent \textbf{Sub-case~3: $L_4 < 1$.} In this case,~\eqref{eq:cheby_lhs} with $c_0,c_1,c_2,c_3 = 0, L_3, L_4, 1$ yields
	\begin{equation}\label{eq:39aux1}
	\begin{gathered}
	2\hpr{X'}{0}{L_3} + 2(1-L_3^2)\hpr{X'}{L_3}{L_4} + 2(1-L_4^2)\hpr{X'}{L_4}{1}
	\\
	\geq
	\be[(1-X'^2) \one\{|X'| < 1\}].
	\end{gathered}
	\end{equation}
	We claim that $\hpr{X'}{L_4}{1} \leq \hpr{X'}{-L_3}{L_3} = 2\hpr{X'}{0}{L_3}$, and consequently, the LHS of~\eqref{eq:39aux1} is upper bounded by
	\[
		 \max\li\{1+2(1-L_4^2), 2(1-L_3^2)\ri\} \cdot \li( \hpr{X'}{0}{L_3} + \hpr{X'}{0}{L_4}\ri),
	\]
	thus implying~\eqref{eq:39glhs} (by noting $c/2 = 1/\max(3-2L_4^2, 2-2L_3^2)$).
	
	\medskip \noindent The claim $\hpr{X'}{L_4}{1} \leq \hpr{X'}{-L_3}{L_3}$ follows from a  segment comparison argument -- namely,  Lemma~\ref{lem:seg-compare1}, applied to $X'$ with the parameters $A,B,C,D,M = -L_3, L_3, L_4, 1, a_4/\sigma$. To see that the lemma can be applied, we have to check that $D-C+2M \leq B-A$. Using $M \leq a_3/\sigma$, this follows from
	\[
	1-L_4 + 2a_3 / \sigma \leq 2L_3.
	\]
	Rearranging, we have to prove $\sigma \leq 3 - a_1 + a_2 - 5a_3$. This indeed holds, since
	\[
		\sigma \leq \sqrt{1-3a_3^2} \leq 5/2 - 4a_3 \leq 3 - a_1 +a_2-5a_3,
	\]
	where the ultimate inequality holds as $a_1+a_3 \leq 1/2+a_2$, and the penultimate inequality follows from
	\[
		(5/2-4a_3)^2 - (1-3a_3^2) = \frac{(1-2a_3)(21-38a_3)}{4} \geq 0.
	\]

	\paragraph{Proving~\tops{\eqref{eq:39grhs}}.}
	The Chebyshev-type inequality~\eqref{eq:cheby_rhs} with the parameters
	\[d_0,d_1,d_2,d_3,d_4 = 1, \sqrt{1+1/d}, \sqrt{1+2/d}, \sqrt{1+3/d}, \sqrt{1+4/d},
	\]
	where $d$ is as defined above, yields
	\[
		\be[(X'^2-1) \one\{|X'| > 1\}] \geq \frac{1}{d} \sum_{i=1}^{4} \pr[|X'| \geq \sqrt{1+i/d}].
	\]
	This implies~\eqref{eq:39grhs}, provided that  $R_i \geq \sqrt{1+i/d}$ holds for all $i\in [4]$. The latter reads $1/d \leq (R_i^2-1)/i$, which indeed follows from the definition of $d$.

\section{Theorem~\ref{thm:main} for \tops{$a_1 \in [0.31, 0.387]$}, \tops{$a_1 + a_2 + a_3 \geq 1$}}\label{sec:0.31<a_1<0.387}
The proof of Theorem~\ref{thm:main} in this region is a bit more intricate than the proof in the other regions. A reason for the extra difficulty is that after eliminating the variables $x_1,x_2,x_3$ by applying Lemma~\ref{cor:nm3}, we have to show the LHS of~\eqref{eq:nm3} is $\geq$ the RHS, where the LHS includes the term $\pr[X' \in [-L_1, L_2]]$, with $L_1 < 0$. It is harder for us to exploit this quantity, as the segment $[-L_1, L_2]$ does not contain $0$, which is essential for an approach based on the Chebyshev-type inequality~\eqref{eq:basic_cheby}. In Appendix~\ref{sec:39geq} we overcome this difficulty by just neglecting the term $\pr[X \in [-L_1, L_2]]$. However, this is not possible in the case $a_1 \approx 1/3$, where this term is inherently needed.

In the proof, we consider two cases, which we further subdivide into two sub-cases each.
\begin{enumerate}
	\item \textbf{There are intermediate-sized weights.} The two sub-cases of this case are:
	\begin{itemize}
		\item There exists $k \geq 4$ with $a_k \in [a_1+a_2+a_3-1, 1-a_1-a_2]$.
		
		\item There exist $k > j > 1$ with $a_2+a_j+a_k \leq 1$ and $a_j,a_k \geq 1-2a_1$.
	\end{itemize}
	
	\item \textbf{There are no intermediate-sized weights.} The two sub-cases of this case are:
	\begin{itemize}
		\item There are at most 4 `large' weights of size $\geq 1-a_1-a_2$.
		
		\item There are at least 5 `large' weights.
	\end{itemize}
\end{enumerate}
The proof in the first three sub-cases goes through elimination of three variables, similarly to the proofs in Appendices~\ref{sec:31big} and~\ref{sec:39geq}. The only significant difference is in the \emph{segment comparison} step, which is somewhat more complex and differs between the three sub-cases. The strategy in the fourth sub-case is different. It goes through elimination of five variables and a combination of a Chebyshev-type inequality with a semi-inductive argument.

\subsection{Case~1: There are intermediate-sized weights}
	\subsubsection{Sub-case~1: There exists \tops{$k \geq 4$} with \tops{$a_k \in [a_1+a_2+a_3-1, 1-a_1-a_2]$}}\label{ssec:31mid}
		\paragraph{Elimination step.} By Lemma~\ref{cor:nm3}, it is sufficient to prove
		\begin{equation}\label{Eq:Aux10.1}
		\begin{gathered}
		\pr[X' \in [-L_1, L_2]] + \pr[X' \in [-L_3, L_4]] \\ \geq \\
		\pr\li[ X' > R_1\ri] + \pr\li[ X' > R_2\ri]
		+ \pr\li[ X' > R_3\ri] + \pr\li[ X' > R_4\ri],
		\end{gathered}
		\end{equation}
		where $\sigma, X', L_1,L_2,L_3,L_4,R_1,R_2,R_3,R_4$ are as defined in Lemma~\ref{cor:nm3}.
		
		\paragraph{Segment comparison step.} The following segment comparison lemma plays an important role in the proof, allowing us to take into account the segment $[-L_1, L_2]$ in a Chebyshev-type approach.
		\begin{lemma}\label{lem:mid-ai}
			Let $X = \sum_{i=1}^{n} a_i x_i$ be a Rademacher sum, such that $\var(X) = 1$, $a_1 + a_2 + a_3 \geq 1$, $0 \leq a_n \leq \ldots \leq a_2 \leq a_1 \leq 0.387$. Suppose there exists $k \geq 4$ with $a_k \in [a_1+a_2+a_3-1, 1-a_1-a_2]$.
			Let $\sigma = \sqrt{1-a_1^2-a_2^2-a_3^2}$, $a_i'=a_i/\sigma$, $X' = \sum_{i=4}^{n} a_i' x_i$ and
			\[
				L_1 = \frac{1-a_1-a_2-a_3}{\sigma}, \qquad L_2 = \frac{1-a_1-a_2+a_3}{\sigma}, \qquad L_4 = \frac{1+a_1-a_2-a_3}{\sigma}
			\]
			be as in Lemma~\ref{lem:31sfin}. Then
			\begin{equation}\label{eq:exotic}
				\Pr[X' \in (L_4,1)] \leq 2 \Pr[X' \in (-L_1,L_2]].
			\end{equation}
		\end{lemma}
		\begin{proof}
			As $\Pr[X' \in (L_4,1)] = \sum_{b \in \spm} \pr[X' \in (L_4,1) \andd x_k = b]$, it is clearly sufficient to show
				\begin{equation}\label{eq:exotic-need}
				\forall b \in \spm \cc \pr[X' \in (L_4,1) \andd x_k = b] \leq \Pr[X' \in (-L_1,L_2]].
			\end{equation}
			We handle these two cases (corresponding to the value of $b$) separately.
			
			\paragraph{Proving~\eqref{eq:exotic-need} for $b=-1$.}
			Applying Lemma~\ref{lem:seg-compare2} to the Rademacher sum $X''=X'-a_k' x_k$, with the parameters $A,B,C,D=L_1+a_k', L_2+a_k', L_4+a_k', 1+a_k'$ and $M \leq a_3'$, we get
			\begin{equation}\label{Eq:Aux10.2}
				\pr[X' \in (L_4,1) \andd x_k = -1] \leq \pr[X' \in (L_1,L_2) \andd x_k = -1].
			\end{equation}
			Note that we apply the lemma with open segments instead of half-open segments; the lemma indeed holds in this setting, as is shown in Appendix~\ref{app:sub:other-types}. To verify that the assumptions of the lemma are satisfied, note that the assumption $0 \leq A \leq C$ holds since $a_k \geq a_1+a_2+a_3-1$, the assumption $2M \leq C-A$ holds since $a_1 \geq a_3$, and the assumption $D-C+D-B \leq B-A$ is equivalent to $a_2-a_3 \leq 1 - \sigma$, and follows from
			\[
				a_2 - a_3 \leq a_1 - (1-2a_1) \leq 3\cdot 0.387 - 1 < 1 - \sqrt{2/3} \leq 1-\sigma,
			\]
			where the first and last inequalities are implied by $a_1+a_2+a_3\geq 1$ (the latter, via the Cauchy-Schwarz inequality).
			By~\eqref{Eq:Aux10.2}, the assertion~\eqref{eq:exotic-need} for $b=-1$ will follow once we show
			\begin{equation}\label{Eq:Aux10.25}
			\pr[X' \in (L_1,L_2] \andd x_k = -1] \leq \pr[X' \in (-L_1,L_2]].
			\end{equation}
			We prove this by constructing an explicit injective map. Let
			\[
			\Omega = \set{z=(z_4,\ldots,z_n) \in \{-1,1\}^{n-3}}{X'(z) \in (L_1,L_2] \wedge z_k=-1},
			\]
			and
			\[
			\Omega' =\set{z=(z_4,\ldots,z_n) \in \{-1,1\}^{n-3}}{X'(z) \in (-L_1,L_2]}.
			\]
			We define $f: \Omega \to \Omega'$ by setting $(f(z))_i=z_i$ for all $i \neq k$, and
			\[
			(f(z))_k =
			\begin{cases}
			\phantom{-}z_k, & \quad X'(z) \in (-L_1,L_2] \\
			-z_k, & \quad X'(z)\not \in (-L_1,L_2]
			\end{cases}.
			\]
			That is, we flip the $k$'th coordinate of $z$ iff $X'(z) \not \in (-L_1,L_2]$, and leave the other coordinates unchanged.
			
			It is clear that $f$ is injective. To see that $\mathrm{Range}(f) \subset \Omega'$, notice that
			\begin{equation}\label{Eq:Aux10.3}
			-2L_1 \leq 2a_k' \leq L_2-(-L_1),
			\end{equation}
			where the first inequality holds since $a_k \geq a_1+a_2+a_3-1$ and the second inequality holds since $a_k \leq 1-a_1-a_2$. As when $X'(z) \not \in (-L_1,L_2]$ we have $X'(f(z))=X'(z)+2a_k'$, the inequality~\eqref{Eq:Aux10.3} implies
			\[
			X'(z) \in (L_1,-L_1] \Longrightarrow X'(f(z)) \in (-L_1,L_2].
			\]
			Since when $X'(z) \in (-L_1,L_2]$ we have $f(z)=z$, the assertion~\eqref{Eq:Aux10.25} follows.
			
			\paragraph{Proving~\eqref{eq:exotic-need} for $b=1$.} Notice that
			\[
			\pr[X' \in (L_4,1) \andd x_k = 1] = \pr[X' \in (L_4-2a_k',1-2a_k') \andd x_k = -1].
			\]
			We proceed by proving $\pr[X' \in (L_4-2a_k',1-2a_k') \andd x_k = -1] \leq \Pr[X' \in (-L_1,L_2]]$ by a slight variation of the proof of~\eqref{eq:exotic-need} for $b=-1$  presented above.
			By~\eqref{Eq:Aux10.25}, it suffices to prove
			\begin{equation}\label{Eq:Aux10.4}
				\pr[X' \in (L_4-2a_k',1-2a_k') \andd x_k = -1] \leq \pr[X' \in (L_1,L_2] \andd x_k = -1].
			\end{equation}
			This inequality is equivalent to
			$\pr[X'' \in (L_4-a_k',1-a_k')] \leq \pr[X'' \in (L_1+a_k',L_2+a_k']]$,
			where $X''=X'-a_k' x_k$, as defined above. By subtracting $\pr[X'' \in (L_4-a_k', L_2+a_k']]$ from both sides, the latter is equivalent to
			\begin{equation}\label{Eq:Aux10.5}
				\pr[X'' \in (L_2+a_k',1-a_k')] \leq
				\pr[X'' \in (L_1+a_k',L_4-a_k']].
			\end{equation}
			To prove~\eqref{Eq:Aux10.5}, we apply Lemma~\ref{lem:seg-compare2} to $X''$, with the parameters  $A,B,C,D=L_1+a_k',L_4-a_k', L_2+a_k', 1-a_k'$ and $M \leq a_3'$, and open segments instead of half-open segments (Appendix~\ref{app:sub:other-types}). To verify that the assumptions of the lemma are indeed satisfied, note that the assumption $0 \leq A \leq C$ holds since $a_k \geq a_1+a_2+a_3-1$, the assumption $2M \leq C-A = 2a_3'$ holds trivially, and the assumption $D-C+D-B \leq B-A$ is equivalent to $\sigma \leq 1+a_1-a_2$, being clear.
			This completes the proof.
		\end{proof}
		
		\paragraph{Auxiliary estimate.} To proceed, we use the following auxiliary inequality in $L_i,R_i$.
		\begin{equation}\label{eq:exotic-final}
		\begin{gathered}
				\max\li\{1/2, \frac{3-2L_4^2}{3}\ri\} \isleq \min\li\{R_1^2-1, \frac{R_2^2-1}{2}, \frac{R_3^2-1}{3}, \frac{R_4^2-1}{4}\ri\},
				\\
				\text{with:}
				\qquad a_1 \in [0.31, 0.387],
				\quad a_3 \leq a_2 \leq a_1,
				\quad a_1+a_2+a_3 \geq 1,
				\\
				\sigma = \sqrt{1-a_1^2-a_2^2-a_3^2},
				\quad L_4 =	\frac{1+a_1-a_2-a_3}{\sigma},
				\quad R_1 = \frac{1-a_1+a_2+a_3}{\sigma},
				\\
				R_2 = \frac{1+a_1-a_2+a_3}{\sigma}
				\quad R_3 = \frac{1+a_1+a_2-a_3}{\sigma},
				\quad R_4 = \frac{1+a_1+a_2+a_3}{\sigma},
				\end{gathered}
				\end{equation}
		The proof of~\eqref{eq:exotic-final} is presented in Appendix~\ref{app:exotic-final}.
		
		\paragraph{Chebyshev-type step.}
		Let
		\begin{equation}
		\begin{gathered}
		c = 1/\max\li\{ (3-2L_4^2)/3, 1/2 \ri\}, \qquad \mbox{and}\\
		d = 1/\min \li\{ R_1^2-1, (R_2^2-1)/2, (R_3^2-1)/3, (R_4^2-1)/4 \ri\}.
		\end{gathered}
		\end{equation}
		We show that
		\begin{equation}\label{eq:ex-lhs}
		\text{LHS of~\eqref{Eq:Aux10.1}} \geq \frac{c}{2} \cdot \be[(1-X'^2) \one\{|X'| < 1\}],
		\end{equation}
		and
		\begin{equation}\label{eq:ex-rhs}
		\text{RHS of~\eqref{Eq:Aux10.1}} \leq \frac{d}{2} \cdot \be[(X'^2-1) \one\{|X'| > 1\}].
		\end{equation}
		As by~\eqref{eq:exotic-final} we have $c \geq d$, the assertion~\eqref{Eq:Aux10.1} follows from~\eqref{eq:ex-lhs} and~\eqref{eq:ex-rhs} by the Chebyshev-type equality~\eqref{eq:basic_cheby}.
		
		\paragraph{Proving~\tops{\eqref{eq:ex-lhs}}.}
			First, we claim that
			\[
			L_3 \geq 1/\sqrt{2}.
			\]
			To see this, note that the assumption $a_1+a_2+a_3 \geq 1$ implies, via the Cauchy-Schwarz inequality, $a_1^2+a_2^2+a_3^2 \geq 1/3$, and thus $\sigma = \sqrt{1-a_1^2-a_2^2-a_3^2} \leq \sqrt{2/3}$. Using this, along with the assumptions $a_1 \leq 0.387$ and $a_2 \geq a_3$, we obtain
			\[
			L_3 = \frac{1-a_1+a_2-a_3}{\sigma} \geq \frac{1-a_1}{\sqrt{2/3}} \geq \frac{1-0.387}{\sqrt{2/3}} > \frac{1}{\sqrt{2}}.
			\]
			As $L_3 \geq 1/\sqrt{2}$, we have
			\begin{equation}\label{Eq:Aux10.6}
			\pr[X' \in [-L_3,L_4]] \geq 2\pr[X' \in \langle 0,L_3]] + 2\max\{0,1-L_3^2\} \cdot \pr[X' \in (L_3,L_4]].
			\end{equation}
			We consider two sub-cases.
			
			\paragraph{Sub-case~1: $L_4 \geq 1$.} Let $\bar{L}_3 = \min\{L_3,1\}$. The Chebyshev-type inequality~\eqref{eq:cheby_lhs}, applied with the parameters $c_0,c_1,c_2=0, \bar{L}_3, 1$, yields
			\begin{equation*}
			2\pr[X'\in \hleft 0, \bar{L}_3 \ri]] + 2(1-\bar{L}_3^2)\pr[X' \in (\bar{L}_3, 1)]
				\geq \be[(1-X'^2) \one\{|X'|<1\}],
			\end{equation*}
			which together with~\eqref{Eq:Aux10.6} and the definition of $c$ implies~\eqref{eq:ex-lhs}.
			
			\paragraph{Sub-case~2: $L_4 < 1$.} Since $L_3 \geq \sqrt{1/2}$, the Chebyshev-type inequality~\eqref{eq:cheby_lhs}, applied with the parameters $c_0,c_1,c_2,c_3=0, L_3, L_4, 1$, implies
			\[
				c \pr[X' \in \langle 0,L_3]] +   \pr[X' \in (L_3,L_4]] + c(1-L_4^2)\pr[X' \in (L_4,1)]  \geq \frac{c}{2} \cdot \be[(1-X'^2) \one\{|X'|<1\}].
			\]
			Thus, in order to deduce~\eqref{eq:ex-lhs} it suffices to show
			\begin{equation}\label{eq:ex-lhs2}
				2\pr[X' \in \langle 0,L_3]]
				+ \pr[X' \in (-L_1,L_2]] \geq c
				\pr[X' \in \langle 0,L_3]] + c(1-L_4^2) \pr[X' \in (L_4,1)].
			\end{equation}
			To prove this, note that  Lemma~\ref{lem:mid-ai} implies
			\[
			\pr[X' \in (L_4,1)] \leq 2\pr[X' \in (-L_1,L_2]] \leq 2\pr[X' \in \langle 0,L_3]].
			\]
			If $2c(1-L_4^2) \leq 1$, then~\eqref{eq:ex-lhs2} follows as $c \leq 2$ by definition. Otherwise, it suffices to check
			\[
				(2 - (2c(1-L_4^2)-1))\pr[X' \in \langle 0,L_3]] \geq c\pr[X' \in \langle 0,L_3]].
			\]
			This inequality follows from $(2 - (2c(1-L_4^2)-1)) \geq c$, which holds by the definition of $c$. This completes the proof.

 		\paragraph{Proving~\tops{\eqref{eq:ex-rhs}}.}
			Inequality~\eqref{eq:ex-rhs} is the same as~\eqref{eq:39grhs} and is proved in the same way. (Note that the slightly different assumptions on $X'$ here do not affect the proof.)

	\subsubsection{Sub-case~2: There exist \tops{$k > j > 1$} with \tops{$a_2+a_j+a_k \leq 1$} and \tops{$a_j,a_k \geq 1-2a_1$}}
	\label{ssec:31mid2}

			\paragraph{Elimination step.} By Lemma~\ref{cor:nm3}, it is sufficient to prove
			\begin{equation}\label{eq:mid-need}
				\hpr{X'}{-L_1}{L_2} + \hpr{X'}{0}{L_3} + \hpr{X'}{0}{L_4} \geq \sum_{i=1}^{4} \pr[X' > R_i],
			\end{equation}
			with
			\[
			X' = \sum_{i \in [n]\sm \{1,j,k\}} \frac{a_i}{\sigma}x_i, \qquad \sigma = \sqrt{1-a_1^2-a_j^2-a_k^2},
			\]
			and
			\begin{equation*}
			\begin{gathered}
				L_1, L_2, L_3, L_4 = \frac{1-a_1-a_j-a_k}{\sigma}, \frac{1-a_1-a_j+a_k}{\sigma}, \frac{1-a_1+a_j-a_k}{\sigma}, \frac{1+a_1-a_j-a_k}{\sigma} \\
				R_1, R_2, R_3, R_4 = \frac{1-a_1+a_j+a_k}{\sigma}, \frac{1+a_1-a_j+a_k}{\sigma}, \frac{1+a_1+a_j-a_k}{\sigma}, \frac{1+a_1+a_j+a_k}{\sigma}.
			\end{gathered}
			\end{equation*}
			
			\paragraph{Auxiliary inequalities.} To proceed, we use the following auxiliary inequalities in $L_i,R_i$.
			\begin{equation}\label{eq:L34Ri}
			\begin{gathered}
			L_4 \geq L_3 \isgeq \sqrt{1/2}, \qquad \forall i\in \{1,2,3,4\}\cc R_i \isgeq \sqrt{1+(i/2)} \\
			\text{with:}\qquad 1-2a_1 \leq a_k \leq a_j \leq a_1 \leq 0.387,\quad \sigma = \sqrt{1-a_1^2-a_j^2-a_k^2}, \\
			L_3 = \frac{1-a_1+a_j-a_k}{\sigma}, \quad L_4 = \frac{1+a_1-a_j-a_k}{\sigma}, \quad R_1 = \frac{1-a_1+a_j+a_k}{\sigma}, \\
			R_2 = \frac{1+a_1-a_j+a_k}{\sigma}, \quad R_3 = \frac{1+a_1+a_j-a_k}{\sigma}, \quad R_4 = \frac{1+a_1+a_j+a_k}{\sigma}.
			\end{gathered}
			\end{equation}
			The proof of~\eqref{eq:L34Ri} is presented in Appendix~\ref{app:L34Ri}.
			
			\paragraph{Segment comparison step.} By Lemma~\ref{lem:seg-compare2}, applied to the Rademacher sum $X'$ with the parameters $A,B,C,D = -L_1, L_2, L_4, 1$ and $M \leq a_2/\sigma$, we have
			\begin{equation}\label{Eq:Aux10.7}
			\hpr{X'}{L_4}{1} \leq \hpr{X'}{-L_1}{L_2}.
			\end{equation}
			To verify that the assumptions of the lemma are satisfied, note that the assumption $A \leq C$, $-A \leq C$ holds since $a_j+a_k < 1$, the assumption $2M \leq C-A$ (being $2a_2/\sigma \leq (2-2a_j-2a_k)/\sigma$) holds since $a_2+a_j+a_k \leq 1$, and the assumption $D-C+D-B \leq B-A$ is equivalent to $\sigma \leq 2-a_1-2a_j$, which holds as
			\[
			(2-a_1-2a_j)^2 - \sigma^2 = (a_k^2 - (1-2a_1)^2) + (a_1-a_j)(8-9a_1-5a_j) + 5(0.4-a_1)(2-3a_1) > 0.
			\]
			
			\paragraph{Chebyshev-type step.}
			Denote $\bar{L}_3=\min\{L_3,1\}$ and $\bar{L}_4=\min\{L_4,1\}$. The Chebyshev-type inequality~\eqref{eq:cheby_our}, applied with the parameters $c_0,c_1,c_2,c_3=0, \bar{L}_3, \bar{L}_4, 1$ and $d_0,d_1,d_2,d_3,d_4=1, \sqrt{3/2}, \sqrt{4/2}, \sqrt{5/2}, \sqrt{6/2}$, gives
			\begin{equation}\label{Eq:Aux10.8}
				\begin{gathered}
					2\hpr{X'}{0}{\bar{L}_3} + 2(1-\bar{L}_3^2)\hpr{X'}{\bar{L}_3}{\bar{L}_4} + 2(1-\bar{L}_4^2)\hpr{X'}{\bar{L}_4}{1}
					\\
					\geq 2 \cdot \sum_{i=1}^{4} \frac{1}{2} \pr[X' \geq \sqrt{1+(i/2)}].
				\end{gathered}
			\end{equation}
			By~\eqref{eq:L34Ri} and~\eqref{Eq:Aux10.7},
			\[
			\hpr{X'}{-L_1}{L_2} \geq \hpr{X'}{L_4}{1} \geq \max\{0,2(1-L_4^2)\}\hpr{X'}{L_4}{1}.
			\]
			Using in addition the inequality $L_3 \geq \sqrt{1/2}$ that holds by~\eqref{eq:L34Ri}, we obtain
			\begin{align*}
				\begin{split}
				\text{LHS of}~\eqref{eq:mid-need} &\geq 2\hpr{X'}{0}{L_3} + \max\{0,2(1-L_3^2)\}\hpr{X'}{L_3}{L_4} \\
				&\qquad \qquad+ \max\{0,2(1-L_4^2)\}\hpr{X'}{L_4}{1}\\ &\geq \text{LHS of}~\eqref{Eq:Aux10.8}.
				\end{split}
			\end{align*}	
			On the other hand, as $\forall i: R_i \geq \sqrt{1+(i/2)}$ by~\eqref{eq:L34Ri}, we have
			\begin{align*}
					\text{RHS of}~\eqref{Eq:Aux10.8}=\sum_{i=1}^{4} \pr[X' \geq \sqrt{1+(i/2)}] \geq \sum_{i=1}^{4} \pr[X' \geq R_i] \geq \text{RHS of}~\eqref{eq:mid-need}.
			\end{align*}	
			Therefore,~\eqref{eq:mid-need} follows from~\eqref{Eq:Aux10.8}.

	\subsection{Case~2: There are no intermediate-sized weights}
		In Appendix~\ref{ssec:31mid} we covered the case where there exists a weight $a_i$ with $a_i \in [a_1+a_2+a_3-1, 1-a_1-a_2]$. We hence assume the inexistence of such weights. That is, we may partition the weights into `big' and `small' ones, $B$ and $S$:
		\begin{itemize}
		\item $B = \set{i\in [n]}{a_i > 1-a_1-a_2}$.     (Notice that $\{1,2,3\} \subset B$, as $a_1+a_2+a_3 \geq 1$.)		
		\item $S = \set{i\in [n]}{a_i < a_1+a_2+a_3-1}$. (Notice that $S = [n] \sm B$.)
		\end{itemize}
		We divide this case into two sub-cases, according to the size of $B$.

		\subsubsection{Sub-case~1: \tops{$|B| \leq 4$}}
		The proof in this case is very similar to the proof in Appendix~\ref{ssec:31mid2}, except for a slightly more complicated segment comparison step.
		
		\paragraph{Elimination step.} By Lemma~\ref{cor:nm3}, it is sufficient to prove
		\begin{equation}\label{Eq:Aux10a.1}
		\begin{gathered}
		\hpr{X'}{-L_1}{L_2} + \hpr{X'}{0}{L_3} + \hpr{X'}{0}{L_4}
		\\
		\geq
		\\
		\hpr{X'}{R_1}{\infty} + \hpr{X'}{R_2}{\infty} + \hpr{X'}{R_3}{\infty} + \hpr{X'}{R_4}{\infty},
		\end{gathered}
		\end{equation}
		where $\sigma, X', L_1,L_2,L_3,L_4,R_1,R_2,R_3,R_4$ are as defined in Lemma~\ref{cor:nm3}.
		
		\medskip \noindent Note that by~\eqref{eq:L34Ri}, we have
		\begin{equation}\label{Eq:Aux10a.2}
		L_4 \geq L_3 \geq \sqrt{1/2} \qquad \mbox{and} \qquad \forall i \cc R_i \geq \sqrt{1+i/2}.
		\end{equation}
		(The assertion~\eqref{eq:L34Ri} applies whenever the weights $a_1,a_j,a_k$ of the three eliminated variables satisfy $1-2a_1 \leq a_k \leq a_j$. This holds for $a_1,a_2,a_3$, since $a_1+a_2+a_3 \geq 1$ by assumption.)
		
		\paragraph{Segment comparison step.} Like in Appendix~\ref{ssec:31mid2}, we claim that
		\begin{equation}\label{Eq:Aux10a.3}
		\hpr{X'}{L_4}{1} \leq \hpr{X'}{-L_1}{L_2}.
		\end{equation}
		We would like to deduce this inequality from Lemma~\ref{lem:seg-compare1}, but for using the lemma we need a good upper bound on the maximal weight of the Rademacher sum it is applied to. In order to obtain such a bound, we eliminate also the variable $x_4$ and use the fact that due to the definition of $B,S$ and the assumption $|B|\leq 4$, we have
		\begin{equation}\label{Eq:Aux10a.4}
		\forall i \geq 5 \cc a_i \leq a_1+a_2+a_3-1.
		\end{equation}
		The argument goes as follows. To prove~\eqref{Eq:Aux10a.3}, it suffices to check
		\[
		\forall b \in \spm \cc \pr[X' \in \hseg{L_4}{1} \andd x_4 = b] \leq \pr[X' \in \hseg{-L_1}{L_2} \andd x_4 = b].
		\]
		Defining $a_i'=a_i/\sigma$ and $X'' = X'-a'_4 x_4 = \sum_{i\geq 5} a_i' x_i$, this boils down to showing
		\[
		\forall b \in \spm \cc \hpr{X''}{L_4-b a_4'}{1-b a_4'} \leq \hpr{X''}{-L_1-b a_4'}{L_2-b a_4'}.
		\]
		This follows from Lemma~\ref{lem:seg-compare1}, applied to the Rademacher sum $X''$ with the parameters
		\[
		A,B,C,D,M = -L_1-b a_4', L_2-b a_4', L_4-b a_4', 1-b a_4', a_5'.
		\]
		Let us verify that the assumptions of the lemma are satisfied. By~\eqref{Eq:Aux10a.4}, all weights of $X''$ are indeed bounded by $M = a_5' \leq (a_1+a_2+a_3-1)/\sigma$. To verify the assumption $C \geq \min\{|A|,|B|\}$, it is sufficient to check $0 \leq B \leq C$. $B \leq C$ is equivalent to $L_2 \leq L_4$ which holds as $a_1 \geq a_3$, and $B > 0$ holds since $a_1+a_2-a_3+a_4 < 1$. Finally, the assumption
		$D-C + 2M \leq B-A$ is equivalent to $\sigma \leq 5-3a_1-5a_2-3a_3$, which holds since
		\begin{equation*}
			\begin{aligned}
				(5-3a_1-5a_2-3a_3)^2 - \sigma^2 =\  & (a_2-a_3)(30-18a_1-40a_2-10a_3) +\\
				& (a_1-a_2)(80-114a_1-66a_2) + (2-4a_1)(12-31a_1) > 0,
			\end{aligned}
		\end{equation*}
		where all expressions here are nonnegative as $a_3 \leq a_2 \leq a_1 \leq 0.387$.
		
		\paragraph{Chebyshev-type step.}
		The Chebyshev-type step is almost identical to the corresponding step in Appendix~\ref{ssec:31mid2}, and thus we describe it very briefly.
		
		\medskip \noindent Denoting $\bar{L}_3=\min\{L_3,1\}$ and $\bar{L}_4=\min\{L_4,1\}$, the Chebyshev-type inequality~\eqref{eq:cheby_our}, applied with $c_0,c_1,c_2,c_3=0, \bar{L}_3, \bar{L}_4, 1$ and $d_0,d_1,d_2,d_3,d_4=1, \sqrt{3/2}, \sqrt{4/2}, \sqrt{5/2}, \sqrt{6/2}$, gives
		\begin{equation}\label{Eq:Aux10a.5}
		\begin{gathered}
		2\hpr{X'}{0}{\bar{L}_3} + 2(1-\bar{L}_3^2)\hpr{X'}{\bar{L}_3}{\bar{L}_4} + 2(1-\bar{L}_4^2)\hpr{X'}{\bar{L}_4}{1}
		\\
		\geq 2 \cdot \sum_{i=1}^{4} \frac{1}{2} \pr[X' \geq \sqrt{1+(i/2)}].
		\end{gathered}
		\end{equation}
		The LHS of~\ref{Eq:Aux10a.1} is $\geq$ the LHS of~\eqref{Eq:Aux10a.5} due to~\eqref{Eq:Aux10a.2} and~\eqref{Eq:Aux10a.3}, while the RHS of~\eqref{Eq:Aux10a.1} is $\leq$ the RHS of~\eqref{Eq:Aux10a.5} due to~\eqref{Eq:Aux10a.2}. Therefore,~\eqref{Eq:Aux10a.1} follows from~\eqref{Eq:Aux10a.5}.

		\subsubsection{Sub-case~2: \tops{$|B|\geq 5$}}

\paragraph{Elimination step.} We begin with eliminating 5 variables. Let
\[
\sigma = \sqrt{1-\sum_{i=1}^{5} a_i^2} \qquad \mbox{and} \qquad X' = \sum_{i=6}^{n} \frac{a_i}{\sigma} x_i.
\]
By Lemma~\ref{lem:nm}, applied with $m=5$, it is sufficient to prove
\begin{equation}\label{eq:nm5}
\sum_{i=0}^{31} \pr[X' > T_k] \leq 8,
\end{equation}
where $\{T_k\}_{k=0}^{31}$ range over all options
\begin{equation*}
T_0, \ldots, T_{31} = \frac{1\pm a_1 \pm a_2 \pm a_3 \pm a_4 \pm a_5}{\sigma}.
\end{equation*}
			We order the $T_k$'s according to the bit-representation of $k$, that is:
			\[
				T_{k} = \frac{1 - \sum_{i=1}^{5} (-1)^{\lfloor k/2^{5-i} \rfloor} a_i}{\sigma},
			\]
			so that, for example
			\[
				T_{22} = \frac{1+a_1-a_2+a_3+a_4-a_5}{\sigma}.
			\]
			
\paragraph{Auxiliary inequalities.} To proceed, we use the following auxiliary inequalities in the $T_k$'s.
\begin{equation}\label{eq:everything}
\begin{gathered}
T_{12} \cdot \max(T_{10}, T_{17}) \isgeq 1, \\
T_{18}, T_{20}, T_{24} \isgeq 1, \\
T_{11}, T_{13}, T_{14} \isgeq \sqrt{1 + (3/3)}, \\
T_{19}, T_{21}, T_{22}, T_{25}, T_{26}, T_{28}  \isgeq \sqrt{1 + (9/3)}, \\
T_{15}, T_{23}, T_{27}, T_{29}, T_{30}, T_{31}  \isgeq \sqrt{1 + (15/3)}, \\
\text{with:}\qquad T_{k} = \frac{1 - \sum_{i=1}^{5} (-1)^{\lfloor k/2^i \rfloor} a_i}{\sigma}, \qquad \sigma = \sqrt{1-\sum_{i=1}^{5} a_i^2}, \\
1-a_2-a_4 \leq a_5 \leq a_4 \leq a_3 \leq a_2 \leq a_1 \leq 0.387.
\end{gathered}
\end{equation}
The proof of~\eqref{eq:everything} is presented in Appendix~\ref{app:everything}.

\medskip \noindent Note that in light of Appendix~\ref{ssec:31mid2}, we may assume $a_2 + a_j + a_k \geq 1$ for any distinct $j,k > 1$ in $B$. In particular, as $|B| \geq 5$, we may assume that $a_2+a_4+a_5 \geq 1$, and thus, the assumptions of~\eqref{eq:everything} hold true in our region.

\paragraph{Reduction step.}
			Note that if $a+b \geq 0$, then $\pr[X' > a] + \pr[X' > b] \leq 1$. Indeed, denoting $X'' = X' + \frac{a-b}{2}x_{n+1}$, where $x_{n+1}$ is a Rademacher random variable independent of $x_1,\ldots,x_n$, we have
			\begin{equation}\label{Eq:Aux10a.6}
				\pr[X' > a] + \pr[X' > b] = 2\pr \left[X'' > \frac{a+b}{2}\right] \leq 2\pr[X'' > 0] \leq 1.
			\end{equation}
			Since in our range, $\forall i,j \cc a_i + a_j < 1$, we have
			\begin{equation*}
			\begin{gathered}
				T_0 + T_7 > 0, \qquad T_1 + T_6 > 0, \qquad T_2 + T_9 > 0, \qquad T_4 + T_{10} > 0, \\
				T_4 + T_{17} > 0, \qquad T_8 + T_3 > 0, \qquad T_{16} + T_5 > 0.			
			\end{gathered}
			\end{equation*}
			Thus, by~\eqref{Eq:Aux10a.6},
			\[
				\forall (i,j) \in P \cc \pr[X' > T_i] + \pr[X' > T_j] \leq 1, \qquad \pr[X' > T_4] + \pr[X' > \min(T_{10}, T_{17})] \leq 1,
			\]
			where $P = \li\{ (0,7), (1,6), (2,9), (3,8), (5, 16)  \ri\}$.
			Hence, in order to prove~\eqref{eq:nm5}, it is sufficient to prove
			\begin{equation}
			\label{eq:nm51}
				\pr[X' > \max(T_{10}, T_{17})] + \sum_{i \neq 0,1,2,3,4,5,6,7,8,9,10,16,17} \pr[X' > T_i] \leq 2.
			\end{equation}
			
\paragraph{Semi-inductive step.}			
			We further claim that
			\begin{equation}\label{Eq:Aux10a.7}
			\pr[X' > \max(T_{10}, T_{17})] + \pr[X' > T_{12}] \leq 1/2.
			\end{equation}
			To see this, rewrite this inequality as
			\begin{equation}
			\label{eq:t1t}
				\pr[X' \in \hleft 0, T_{12} \ri]] \geq \pr[X' > \max(T_{10}, T_{17})].
			\end{equation}
			We deduce~\eqref{eq:t1t} from the assertion of Theorem~\ref{thm:main} for a Rademacher sums on $n-4$ variables, whose correctness we may assume by induction.
			
			\medskip \noindent
			Note that by Lemma~\ref{lem:nm}, applied with $m=1$, for any Rademacher sum $Z'$ with $\mathrm{Var}[Z']=1$ and any $0<t\leq 1$, the assertion \begin{equation}\label{Eq:Aux10a.8}
			\pr[Z' \in \hleft 0, t \ri]] \geq \pr[Z' > 1/t]
			\end{equation}
			follows from Tomaszewski's assertion $\Pr[|Z|\leq 1] \geq 1/2$ for the Rademacher sum $Z=b_1 x_1 + \sigma' Z'$, where $b_1=(1-t^2)/(1+t^2)$ and $\sigma=\sqrt{1-b_1^2}$ (see~\eqref{eq:nm1} and~\eqref{Eq:Prelim-Eliminate-one})). Furthermore, the range $0<t \leq 1$ in~\eqref{Eq:Aux10a.8} can be extended to all $t>0$, since the inequality $\pr[Z' \in \hleft 0, t \ri]] \geq \pr[Z' > 1/t]$ is equivalent to $\pr[Z' \in \hleft 0, 1/t \ri]] \geq \pr[Z' > t]$.
			
			\medskip \noindent Applying this to the Rademacher sum $X'$, with $t=T_{12}$, we deduce
			\begin{equation}\label{Eq:Aux10a.9}
			\pr[X' \in \hleft 0, T_{12} \ri]] \geq \pr[X' > 1/T_{12}]
			\end{equation}
			from Tomaszewski's assertion for Rademacher sums on $n-5+1=n-4$ variables, which holds by the induction hypothesis.
			
			\medskip \noindent As $\max\{T_{10}, T_{17}\} \geq 1/T_{12}$ by~\eqref{eq:everything}, the assertion~\eqref{eq:t1t} follows from~\eqref{Eq:Aux10a.9}.
			
			\paragraph{Chebyshev-type inequality step.} By combining~\eqref{eq:nm51} with~\eqref{Eq:Aux10a.7} and replacing $\Pr[X'>t']$ with $1/2-\Pr[X' \in \langle 0,t']]$ for $t'=T_{18},T_{20},T_{24}$, we are left with proving
			\begin{equation}
				\sum_{i \in \{11,13,14,15, 19, 21,22,23, 25,26,27,28,29,30,31 \}} \pr[X' > T_i] \leq \sum_{j \in \{18, 20, 24\}} \pr[X' \in \hleft 0, T_j \ri] ].
			\end{equation}
			Since $T_{18}, T_{20}, T_{24} \geq 1$ by~\eqref{eq:everything}, it is sufficient to prove
			\begin{equation}\label{eq:nm52}
			\Pr[X' \in \langle 0,1]] \geq \frac{1}{3} \sum_{i \in \{11,13,14,15, 19, 21,22,23, 25,26,27,28,29,30,31 \}} \pr[X' > T_i].
			\end{equation}
			The Chebyshev-type inequality~\eqref{eq:cheby_our}, applied to the Rademacher sum $X'$ with $c_0,c_1=0,1$ and the sequence $d_i=\sqrt{1+(i/3)}$, $i=1,2,3,\ldots$, yields
			\begin{equation}
			\label{eq:nm5cheby}
				\Pr[X' \in \langle 0,1]] \geq \frac{1}{3} \sum_{i=1}^{\infty} \pr[X' > \sqrt{1+(i/3)}].
			\end{equation}
			The assertion~\eqref{eq:nm52} follows  from~\eqref{eq:nm5cheby} instantly, via the lower bounds on the $T_i$'s proved in~\eqref{eq:everything}.
			
			\medskip \noindent This completes the proof of Theorem~\ref{thm:main}.

\section{Proofs of inequalities}
\label{app:inequalities}

In this appendix we prove a series of inequalities that are used at various stages of the proof of Theorem~\ref{thm:main}. 

\paragraph{Polynomial inequalities.} We usually choose to prove inequalities through \textit{Positivstellensatz}, i.e., representation as a combination of terms that are transparently positive. For example, in order to prove that $2a^2 - ab + 1/16 \geq 0$ holds for any $a \geq b \geq 0$, we just write
	\[
	2a^2 - ab + 1/16 = a(a-b+1/2) + (a-1/4)^2 \geq 0.
	\]

\subsection{Proof of Inequality~\tops{\eqref{eq:055fin}}}\label{app:055fin}
	Recall we have to prove the following, in the range $0< t \leq \sqrt{1-\sqrt{1/2}}$:
	\[
		1 + 2(1-t^2) + 4\sum_{k=2}^{\lceil 1/t \rceil -1} \li( 1-(kt)^2\ri) \leq (1/t^2) - 1.
	\]
	The proof splits into three simple cases.
	
	\paragraph{Case~1: $1/2 \leq t \leq \sqrt{1-\sqrt{1/2}}$.}
		In this case we have to prove $1+2(1-t^2) \leq 1/(t^2) - 1$, which is equivalent to
		\[
		2t^4 - 4t^2+1 \geq 0.
		\]
		This is a simple quadratic inequality in $t^2$, which holds in particular when $t^2 \leq 1-\sqrt{1/2}$, as required.
	
	\paragraph{Case~2: $1/3 \leq t < 1/2$.}
		In this case the inequality states $1+2(1-t^2)+4(1-4t^2)\leq 1/t^2-1$, which is equivalent to
		\[
		18t^4 - 8t^2 + 1 \geq 0.
		\]
		Applying the inequality $2ab \leq a^2+b^2$, we see that $8t^2 \leq \sqrt{72t^4} \leq 1+18t^4$, as required.
	\paragraph{Case~3: $0<t < 1/3$.}
		Notice that
		\[
		\sum_{k=2}^{\lceil 1/t \rceil-1} \li( 1-(kt)^2\ri) \leq \int_{1}^{1/t} (1-(xt)^2) dx=\frac{1}{t}-1-\frac{1}{3}\li(\frac{1}{t}-t^2\ri).
		\]
		Thus, it is sufficient to prove $1 + 2(1-t^2) + 4(1/t-1-(1/t-t^2)/3) \leq 1/t^2 - 1$. Rearranging, this is equivalent to $2t^4-8t+3 \geq 0$. Since $t< 1/3$, one trivially has $3-8t\geq 0$, as required.

		\subsection{Proof of Inequality~\tops{\eqref{eq:05saux}}}\label{app:05saux}
		
		Recall we have to prove $\frac{2a_2}{\sigma} \leq t$ and $(1 - 3t/2) + \frac{2a_2}{\sigma} \leq t$, with
		\[
		a_1 \in [0.5, 0.55], \qquad \sigma = \sqrt{1-a_1^2}, \qquad t=\sqrt{\frac{1-a_1}{1+a_1}}, \qquad a_2 \leq \frac{a_1-3+\sqrt{25+10a_1-63a_1^2}}{8}.
		\]
		Note the first inequality follows from the second one, as $t \leq \sqrt{1/3}<2/3$ in our parameters range. Rewriting the second inequality to depend only on $a_1$ (by writing $\sigma,t$ in terms of $a_1$), we are required to prove:
		\[
		4\sqrt{1-a_1^2} + \sqrt{25+10a_1-63a_1^2} \leq 13 - 11a_1.
		\]
		Notice $\sqrt{1-a_1^2} \leq 1-a_1^2/2$, and so after rearranging, it is sufficient to show
		\[
		25+10a_1-63a_1^2 \leq (2a_1^2 - 11a_1 + 9)^2,
		\]
		or equivalently, $a_1^4 - 11a_1^3 + 55a_1^2 - 52a_1  + 14\geq 0$. We present it as a sum of squares:
		\[
		a_1^4 - 11a_1^3 + 55a_1^2 - 52a_1  + 14 = \li( a_1^2 - \frac{11}{2}a_1 + 3\ri)^2 + \frac{75}{4}\li( a_1-\frac{38}{75}\ri) ^ 2 + \frac{14}{75}>0.
		\]
		
		\subsection{Proof of Inequality~\tops{\eqref{eq:05rep2}}}\label{app:05fin2}
		
		Recall we have to prove
		\[
		1 + (1-t^2) + (1-(3t/2)^2) \leq 1/t^2-1,
		\]
		in the range $t^2 \in (0, 1/3]$. This inequality is equivalent to $13t^4-16t^2+4 \geq 0$. This latter inequality is quadratic in $t^2$, and holds whenever $t^2 \leq (8-\sqrt{12})/13$. In particular, it holds for $t^2 \leq 1/3$, as asserted.

\subsection{Proof of Inequality~\tops{\eqref{eq:05fin}}}\label{app:05fin}
	Recall we are required to prove
	\begin{equation*}
	\begin{gathered}
		L_2 \isgeq \frac{2}{3}, \qquad R_1 \isgeq \max(\sqrt{3 - L_2^2},\sqrt{2}), \qquad R_2 \isgeq \max(\sqrt{5 - 2L_2^2},\sqrt{3}) \\
		\text{with:}\qquad L_2 = \frac{1-a_1+a_2}{\sigma},\quad R_1 = \frac{1+a_1-a_2}{\sigma},\quad R_2 = \frac{1+a_1+a_2}{\sigma}, \quad
		\sigma = \sqrt{1-a_1^2-a_2^2},
	\end{gathered}
	\end{equation*}
	where $a_1 \in [0.5, 0.55]$, $a_1+a_2 < 1$, and $a_2 \geq \li( a_1-3 + \sqrt{25+10a_1-63a_1^2} \ri)/8$.

	\subsubsection{Proving \tops{$R_1 \geq \max(\sqrt{3 - L_2^2},\sqrt{2})$}} In the inequality $R_1 \geq \sqrt{3 - L_2^2}$, both sides are positive, and squaring shows  equivalence to the inequality
	\begin{equation}\label{Eq:AuxB4.1}
	5a_2^2 - 4a_1 a_2 + (5a_1^2-1) \geq 0.
	\end{equation}
	Considering the left hand side as a quadratic function of $a_2$, we find it has no zeros, as
	\[
	\Delta = (-4a_1)^2 - 4\cdot 5\cdot (5a_1^2-1) = 20 - 84a_1^2<0,
	\]
	where the ultimate inequality holds since $a_1 \geq 1/2$. Thus,~\eqref{Eq:AuxB4.1} holds for any value of $a_2$.
	
	\medskip \noindent In a similar way, in the inequality $R_1 \isgeq \sqrt{2}$, squaring shows equivalence to the inequality
	\begin{equation}\label{Eq:AuxB4.2}
	3a_2^2-2(1+a_1) a_2 + 3a_1^2+2a_1-1 \isgeq 0.
	\end{equation}
	Considering the left hand side as a quadratic function of $a_2$, we have
	\[
	\Delta = 4(1+a_1)^2 - 4\cdot 3\cdot (3a_1^2+2a_1-1) = 16 - 16a_1- 32a_1^2 \leq 0,
	\]
	where the ultimate inequality holds since $a_1 \geq 1/2$. Thus,~\eqref{Eq:AuxB4.2} holds for any value of $a_2$.
	
	\subsubsection{Proving \tops{$R_2 \geq \max (\sqrt{5 - 2L_2^2}, \sqrt{3})$}} The inequality $R_2 \geq \sqrt{3}$ holds since the assumption $a_1 \geq 1/2$ implies
	\[
	R_2 = \frac{1+a_1+a_2}{\sqrt{1-a_1^2-a_2^2}} \geq \frac{1+1/2}{\sqrt{1-(1/2)^2}}= \sqrt{3}.
	\]
	Regarding the inequality $R_2 \isgeq \sqrt{5 - 2L_2^2}$, squaring shows equivalence to $8a_2^2 + (6-2a_1) a_2 + (8a_1^2-2a_1-2) \geq 0$. Solving this quadratic inequality in $a_2$ shows that it holds whenever
	\[
		a_2 \geq \frac{(2a_1 - 6) + \sqrt{(6-2a_1)^2 - 32(8a_1^2-2a_1-2)}}{16},
	\]
	or equivalently, $a_2 \geq \li( a_1-3 + \sqrt{25+10a_1-63a_1^2} \ri)/8$, which is precisely the assumption we made on $a_2$.
	
	\subsubsection{Proving \tops{$L_{2} \geq 2/3$}} By squaring and rearranging, one sees the inequality is equivalent to
	\[
		13a_2^2 + (18-18a_1) a_2 + (13 a_1^2-18 a_1 + 5) \geq 0.
	\]
	We already observed (in the proof of $R_2 \geq \sqrt{5 - 2L_2^2}$) that $8a_2^2 + (6-2a_1) a_2 + (8a_1^2-2a_1-2) \geq 0$ holds in our range of parameters. Thus, it is sufficient to prove
	\[
		13a_2^2 + (18-18a_1) a_2 + (13 a_1^2-18 a_1 + 5) - \frac{3}{2}(8a_2^2 + (6-2a_1) a_2 + (8a_1^2-2a_1-2)) \geq 0,
	\]
	or equivalently, $a_2^2 + (9-15a_1)a_2 + (a_1^2 - 15a_1 + 8) \geq 0$. This indeed holds in our range, since $a_2^2 \geq 0$ holds trivially, $9-15a_1 \geq 0$ holds as $a_1 \leq 0.55$, and $a_1^2 - 15a_1 + 8 \geq 0$ holds whenever $a_1 \leq 0.553$, whereas we assume $a_1 \leq 0.55$. This completes the proof.

\subsection{Proof of Inequality~\tops{\eqref{eq:31a2s}}}\label{app:31a2s}
	Recall we have to prove
	\begin{equation}\label{eq:31a2srecall}
		\pr[Z > t] + \pr[Z > 1/t] \isleq 0.324,
	\end{equation}
	for	$a_1 \in (0.31, 0.5)$, $t=\sqrt{\frac{1-a_1}{1+a_1}}$, and a standard Gaussian $Z\sim N(0,1)$. We claim that the function $f:t \mapsto \pr[Z > t] + \pr[Z > 1/t]$ is decreasing in $(0,1)$, and thus, it is sufficient to verify~\eqref{eq:31a2srecall} only at the maximal $t=\sqrt{1/3}$, achieved at $a_1 = 1/2$. For this value of $t$, we indeed have $\pr[Z > \sqrt{1/3}] + \pr[Z > \sqrt{3}] \leq 0.324$, as required.
	
	\medskip \noindent To see that $f$ is decreasing in $(0, 1)$, note that
	\[
		f'(t) = \frac{1}{\sqrt{2\pi}}\li( \frac{e^{-1/(2t^2)}}{t^2} - e^{-t^2/2} \ri).
	\]
	Letting $s=1/t^2$ so that $s \geq 1$, we see that the assertion $f'(t) \leq 0$ is equivalent to $\log(s) \leq (s - \frac{1}{s})/2$. This indeed holds for all $s \geq 1$, as at $s=1$ the two sides are equal, and the derivative of the l.h.s.~(namely, $1/s$) is never larger than the derivative of the r.h.s.~(namely, $(1+1/s^2)/2$), by the arithmetic vs. geometric means inequality. This completes the proof.

\subsection{Proof of Inequality~\tops{\eqref{eq:31a3s}}}\label{app:31a3s}
	Recall we have to prove
	\begin{equation}\label{eq:31a3srecall}
		\begin{gathered}
		\sum_{k=1}^{4} \pr[Z > T_k] \isleq 0.664, \\
		\text{with:}\quad Z\sim N(0,1), \quad 0.19 \leq a_2 \leq a_1 \leq 1/2, \quad \sigma = \sqrt{1-a_1^2-a_2^2}, \quad T_k = \frac{1\pm a_1 \pm a_2}{\sigma}.
	\end{gathered}
	\end{equation}
	\indent Let
	$
		f(a_1, a_2) = \sum_{k=1}^{4} \pr[Z > T_k],
	$
	so that we wish to prove $f(a_1, a_2) \leq 0.664$ in the region $\mathcal{D} = [0.19,1/2]^2$. Experimental evidence strongly suggests that in this region, $f$ is maximized at $(a_1, a_2) = (1/2, 1/2)$, where $f(1/2, 1/2) = 0.6596 \pm 10^{-4} < 0.664$. However, we could not find a short analytic proof of this, and so we (rigorously) prove~\eqref{eq:31a3srecall} numerically.
	
	Using a bound on the gradient of $f$, we verify $f(a_1, a_2) \leq 0.664$ for all $(a_1, a_2) \in \mathcal{D}$, by sampling a sufficiently fine net. Specifically, by Lagrange's mean value theorem, if for some function $g:\mathcal{D} \rightarrow \mathbb{R}$ that is continuous in $\mathcal{D}$ and differentiable in the interior $\mathcal{D}^{\circ}$, the gradient $\nabla g$ satisfies $\norm{\nabla g (\bar{a})}_2 \leq C$ for any $\bar{a}=(a_1,a_2) \in \mathcal{D}^{\circ}$, then for any $\bar{a},\bar{a}' \in \mathcal{D}$ we have
	\begin{equation}\label{eq:deriv-bound}
	g(\bar{a}) \leq g(\bar{a}') + C\norm{\bar{a}-\bar{a}'}_2.
	\end{equation}
	
	\paragraph{Bounding the gradient of $f$.} Recall $f = \sum_{k=1}^{4} \pr[Z > T_k]$ for a Gaussian $Z \sim N(0,1)$.
	Hence, for $i=1,2$ we have
	\[
		\frac{\partial f}{\partial a_i} = \frac{-1}{\sqrt{2\pi}} \sum_{k=1}^{4} \frac{\partial T_k}{\partial a_i} \exp(-T_k^2/2).
	\]
	For $T_k = (1 + a_1 y_1 + a_2 y_2) / \sigma$ with $y_1,y_2 \in \{-1, 1\}$, we have
	\[
		\frac{\partial T_k}{\partial a_i} = \frac{a_i T_k}{\sigma^2} + \frac{y_i}{\sigma}.
	\]
	Hence, using $0 \leq a_1, a_2 \leq 1/2$, $\sigma \geq 1/\sqrt{2}$ and $\forall k: T_k \geq 0$, we obtain
	\begin{equation*}
	\begin{aligned}
		\li| \frac{\partial f}{\partial a_i} \ri|
		& \leq
		\frac{4}{\sqrt{2\pi}} \max_k \li(\exp(-T_k^2/2) \li(\frac{\max(a_1, a_2)}{\sigma^2}T_k + \frac{1}{\sigma} \ri) \ri)
		\\
		& \leq \frac{4}{\sqrt{2\pi}} \max_k (\exp(-T_k^2/2) (T_k + \sqrt{2} )).
	\end{aligned}
	\end{equation*}
	A simple analysis shows that the function $T_k \mapsto \exp(-T_k^2/2) (T_k + \sqrt{2} )$ has a global maximum at $T_k = (\sqrt{6}-\sqrt{2})/2$, where it attains a value $< 1.69$. Thus, for $i=1,2$ we have $\li| \frac{\partial f}{\partial a_i} \ri| < 2.7$, and therefore,
	\[
	\norm{\nabla f}_2 \leq \sqrt{2}\cdot 2.7 < 4.
	\]
	\paragraph{Sampling $f$.} Consider the set of points $A^2$, where $A = \set{0.185 + 0.0015 \ell}{\ell \in \integers, 0\leq \ell \leq 210}$. Any point $\bar{a} \in [0.19, 0.5]^2$ has some $\bar{a}' \in A^2$ with $\norm{\bar{a}-\bar{a}'}_2 \leq 0.0015 / \sqrt{2}$. Hence, applying~\eqref{eq:deriv-bound} (with the bound $\norm{\nabla f}_2 < 4$), we get
	\[
		f(\bar{a}) < f(\bar{a}') + 4\cdot \frac{0.0015}{\sqrt{2}} < f(\bar{a}') + 0.0043.
	\]
	Therefore, to deduce~\eqref{eq:31a3srecall} it is sufficient to verify
	\begin{equation}\label{eq:net-bound}
		\forall \bar{a}' \in A^2 \cc f(\bar{a}') \leq 0.6597.
	\end{equation}
	A computer program which verifies the $\binom{211}{2}$ inequalities included in~\eqref{eq:net-bound} is provided in \url{https://github.com/IamPoosha/tomaszewski-problem/blob/master/formal_verification.py}.

\subsection{Proof of \tops{Inequality~\eqref{eq:31check-cheby}}}\label{app:31check-cheby}
	Recall we have to prove
	\begin{equation*}
	\begin{gathered}
		L_3 \isgeq 0, \qquad L_4 \isgeq \sqrt{1/2},\qquad \forall i \in \{1,2,3,4\}\cc R_i \isgeq \sqrt{(1+i)/2},
		\\
		\text{with:}
		\qquad a_1 \in [0.31, 0.5],
		\quad a_2 \in [0.19, a_1],
		\quad a_3 \in [0.15, a_2],
		\quad a_1+a_2+a_3 \leq 1,
		\\
		\sigma = \sqrt{1-a_1^2-a_2^2-a_3^2},
		\\
		L_3 = \frac{1-a_1+a_2-a_3}{\sigma},
		\quad L_4 =	\frac{1-|a_1-a_2-a_3|}{\sigma},
		\quad R_1 = \frac{1+|a_1-a_2-a_3|}{\sigma},
		\\
		R_2 = \frac{1+a_1-a_2+a_3}{\sigma}
		\quad R_3 = \frac{1+a_1+a_2-a_3}{\sigma},
		\quad R_4 = \frac{1+a_1+a_2+a_3}{\sigma},
	\end{gathered}
	\end{equation*}
	\subsubsection*{Proving $L_3 \geq 0$}
		Follows from $a_1 + a_2 + a_3 \leq 1$ and $0 < a_2$.

	\subsubsection*{Proving $L_4 \geq \sqrt{1/2}$}
		We wish to prove $1-|a_1-a_2-a_3| \geq \sigma/\sqrt{2}$. Notice we have \[
		|a_1-a_2-a_3| \leq 1/3,
		\]
		as $a_1-a_2-a_3 \leq 0.5-0.19-0.15 \leq 0.16$ and
		\[-(a_1-a_2-a_3) = a_3-(a_1-a_2) \leq a_3 \leq (a_1+a_2+a_3)/3 \leq 1/3.
		\]
		Notice also $\sigma^2 \leq 1-0.31^2-0.19^2-0.15^2 \leq 0.8453$. Hence, we conclude with
		\[
			1-|a_1-a_2-a_3| \geq 2/3 \geq \sqrt{0.8453/2} \geq \sigma/\sqrt{2}.
		\]

	\subsubsection*{Proving $R_1 \geq 1$}
		We verify the stronger inequality $1-a_1+a_2+a_3 \geq \sigma$. Notice that as we assume $a_2 \geq 0.19$ and $a_3 \geq 0.15$, we have
		\[
		1-a_1+a_2+a_3 \geq 1.34-a_1 \qquad \mbox{and} \qquad \sigma^2 \leq 0.9414-a_1^2.
		\]
		Thus, it is sufficient to check $(1.34-a_1)^2 \geq 0.9414-a_1^2$. This indeed holds, as
		\[
			(1.34-a_1)^2 - (0.9414-a_1^2) = (1-2a_1)(0.84-a_1) + 0.0142 > 0.
		\]

	\subsubsection*{Proving $R_2 \geq \sqrt{3/2}$}
		Recall $\sigma\leq \sqrt{1-0.31^2-0.19^2-0.15^2}$, hence $R_2 \geq (1+a_3)/\sigma \geq 1.15/\sqrt{0.8453} > \sqrt{3/2}$.

	\subsubsection*{Proving $R_3 \geq \sqrt{2}$}
		Recall $\sigma\leq \sqrt{1-0.31^2-0.19^2-0.15^2}$, hence $R_3 \geq (1+a_1)/\sigma \geq 1.31/\sqrt{0.8453} > \sqrt{2}$.

	\subsubsection*{Proving $R_4 \geq \sqrt{5/2}$}
		Simply, $R_4 \geq 1+0.31+0.19+0.15 > \sqrt{5/2}$.
	
\subsection{Proof of Inequality~\tops{\eqref{eq:39gglue}}}\label{app:39gglue}
	Recall we have to prove
	\begin{equation}
		\begin{gathered}
			\max\li\{ 1-L_3^2, 3/2-L_4^2, 1/2 \ri\} \isleq \min \li\{ \frac{R_1^2-1}{1}, \frac{R_2^2-1}{2}, \frac{R_3^2-1}{3}, \frac{R_4^2-1}{4} \ri\},\\
			\text{with:}
			\qquad a_1 \in [0.387, 0.5],
			\quad 0.15 \leq a_3 \leq a_2 \leq a_1,
			\quad a_1+a_2+a_3 \geq 1,
			\\
			\sigma = \sqrt{1-a_1^2-a_2^2-a_3^2},
			\\
			L_3 = \frac{1-a_1+a_2-a_3}{\sigma},
			\quad L_4 =	\frac{1+a_1-a_2-a_3}{\sigma},
			\quad R_1 = \frac{1-a_1+a_2+a_3}{\sigma},
			\\
			R_2 = \frac{1+a_1-a_2+a_3}{\sigma}
			\quad R_3 = \frac{1+a_1+a_2-a_3}{\sigma},
			\quad R_4 = \frac{1+a_1+a_2+a_3}{\sigma},
		\end{gathered}
	\end{equation}

In principle, a proof that $\max A \leq \min B$ for two sets $A,B$ consists of $|A|\cdot|B|$ comparisons. In our case, $|A| = 3$ and $|B|=4$. We
effectively reduce $|B|$ to 3 by showing that  $(R_3^2-1)/3 \leq (R_4^2-1)/4$, and then we verify the $3 \cdot 3=9$ remaining inequalities.

	\subsubsection*{Proving \tops{$(R_3^2-1)/3 \leq (R_4^2-1)/4$}}
	Recall $0.15 \leq a_3 \leq a_2 \leq a_1 \leq 1/2$. Hence,
	\begin{equation*}
	\begin{aligned}
	(3(R_4^2-1) - 4(R_3^2-1))\sigma^2 = \ & \sigma^2 + (1-a_1-a_2)(0.9+a_1+a_2) + \\
										  & +(a_3-0.15)(13.85+14a_1+14a_2-a_3)+0.1775 > 0.
	\end{aligned}
	\end{equation*}

	\subsubsection*{Proving \tops{$1-L_3^2   \leq R_1^2-1$}}
	This inequality reads as $2 \leq L_3^2 + R_1^2$, or equivalently, $\sigma^2 \leq (1-a_1+a_2)^2 + a_3^2$. As $a_1+a_2+a_3\geq 1$, it is sufficient to prove the inequality
	\[
	1-a_1^2-a_2^2 \leq (1-a_1+a_2)^2 + 2(1-a_1-a_2)^2.
	\]
	Since $0.25 \leq a_2 \leq a_1 \leq 1/2$, we have
	\[
		(1-a_1+a_2)^2 + 2(1-a_1-a_2)^2-(1-a_1^2-a_2^2) = (1-2a_1)(2-a_2-2a_1)+a_2(4a_2-1) \geq 0.
	\]

	\subsubsection*{Proving \tops{$1-L_3^2   \leq (R_2^2-1)/2$}}
	This inequality reads as $(1+a_1-a_2+a_3)^2 + 2(1-a_1+a_2-a_3)^2 - 3\sigma^2 \geq 0$.
	By the Cauchy-Schwarz inequality, we have
	\[
	a_1^2+a_2^2+a_3^2 \geq (a_1+a_2+a_3)^2/3 \geq 1/3,
	\]
	and hence, $\sigma^2 \leq 2/3$. Thus, it suffices to show $(1+a_1-a_2+a_3)^2 + (1-a_1+a_2-a_3)^2 - 2 \geq 0$, which clearly holds as $(1+A)^2+(1-A)^2 \geq 2$ for any $A$.

	\subsubsection*{Proving \tops{$1-L_3^2   \leq (R_3^2-1)/3$}}
	This inequality reads as $4\sigma^2 \leq 3(1-a_1+a_2-a_3)^2 + (1+a_1+a_2-a_3)^2$. Since $\sigma^2 \leq 2/3$ and $a_3 \leq a_2$, it suffices to prove $3(1-a_1)^2 + (1+a_1)^2 \geq 8/3$. This indeed holds, as
	\[
		3(1-a_1)^2 + (1+a_1)^2 = 4(a_1-1/2)^2 + 3 > 8/3.
	\]

	\subsubsection*{Proving \tops{$3/2-L_4^2 \leq (R_1^2-1)/2$}, which implies \tops{$3/2-L_4^2 \leq \min(R_1^2-1, (R_2^2-1)/2)$}}
	This inequality reads as $2\sigma^2 \leq (1-A)^2 + (1+A)^2/2$, with $A=-a_1+a_2+a_3$. Since $\sigma^2 \leq 2/3$, it suffices to show
	\[
	(1-A)^2 + (1+A)^2/2 \geq 4/3.
	\]
	This indeed holds, as $(1-A)^2 + (1+A)^2/2 = 4/3 + (3A-1)^2/6$.

	\subsubsection*{Proving \tops{$3/2-L_4^2 \leq (R_3^2-1)/3$}}
	This inequality reads as $11\sigma^2/2 \leq 3(1+a_1-a_2-a_3)^2+(1+a_1+a_2-a_3)^2$. Set $\alpha = 141/365$, then
	\begin{equation*}
	\begin{gathered}
		6(1+a_1-a_2-a_3)^2+2(1+a_1+a_2-a_3)^2 - 11\sigma^2 \\
		= \\
		(a_1-\alpha) (16+19\alpha+19a_1-8a_2-16a_3) + 4(a_2-a_3)(1+\alpha-a_2+a_3) + \\
		+ 23(a_2-6(1+\alpha)/23)^2 + 23(a_3-6(1+\alpha)/23)^2 > 0,
	\end{gathered}
	\end{equation*}
	where the last inequality follows from the assumptions $a_3 \leq a_2 \leq a_1 \leq 1/2$ and $a_1 \geq 0.387 > \alpha$.
	
	\subsubsection*{Proving \tops{$1/2 \leq R_1^2-1$}}
	This inequality reads as $\frac{3}{2} \cdot \sigma^2 \leq (1-a_1+a_2+a_3)^2$. Recall $\sigma^2 \leq 2/3$, so we conclude with
	\[
		(1-a_1+a_2+a_3)^2 = (1+(a_1+a_2+a_3-1)+(1-2a_1))^2 \geq 1 \geq \frac{3}{2} \cdot \sigma^2.
	\]

	\subsubsection*{Proving \tops{$1/2 \leq (R_2^2-1)/2$}}
	This inequality reads as $(1+a_1-a_2+a_3)^2 - 2\sigma^2 \geq 0$. This indeed holds, since
	\[
		(1+a_1-a_2+a_3)^2 - 2\sigma^2 = (a_1-a_2)(2-a_1-3a_2+2a_3)+((2a_1)^2-(1-a_3)^2)+4a_3^2 \geq 0,
	\]
	where the last inequality follows from the assumptions $0 \leq a_3 \leq a_2\leq a_1 \leq 1/2$ and $2a_1 \geq a_1+a_2 \geq 1-a_3$.
	
	\subsubsection*{Proving \tops{$1/2 \leq (R_3^2-1)/3$}}
	This inequality reads as $\frac{5}{2} \cdot \sigma^2 \leq (1+a_1+a_2-a_3)^2$. Recall $\sigma^2 \leq 2/3$. Using the inequality $a_1+a_2-a_3\geq a_1 \geq 0.387$, we get
	\[
		(1+a_1+a_2-a_3)^2 \geq 1.387^2 > \frac{5}{3} \geq \frac{5}{2} \cdot \sigma^2.
	\]

\subsection{Proof of Inequality~\tops{\eqref{eq:exotic-final}}}\label{app:exotic-final}
We prove a slight strengthening of the inequality:
	\begin{equation*}
		\begin{gathered}
			\max\li\{ 1/2, \frac{3-2L_4^2}{3}\ri\} \isleq \frac{10}{27\sigma^2} \isleq \min\li\{R_1^2-1, \frac{R_2^2-1}{2}, \frac{R_3^2-1}{3}, \frac{R_4^2-1}{4}\ri\},
			\\
			\text{with:}
			\qquad a_1 \in [0.31, 0.387],
			\quad a_3 \leq a_2 \leq a_1,
			\quad a_1+a_2+a_3 \geq 1,
			\\
			\sigma = \sqrt{1-a_1^2-a_2^2-a_3^2},
			\quad L_4 =	\frac{1+a_1-a_2-a_3}{\sigma},
			\quad R_1 = \frac{1-a_1+a_2+a_3}{\sigma},
			\\
			R_2 = \frac{1+a_1-a_2+a_3}{\sigma}
			\quad R_3 = \frac{1+a_1+a_2-a_3}{\sigma},
			\quad R_4 = \frac{1+a_1+a_2+a_3}{\sigma},
	\end{gathered}
	\end{equation*}
	Proving this involves checking 6 inequalities, and so we proceed.
	\subsubsection*{Proving $1/2 \leq 10/(27\sigma^2)$}
		This inequality is equivalent to $a_1^2+a_2^2+a_3^2 \geq 7 / 27$. This indeed holds, as we have $a_1^2+a_2^2+a_3^2 \geq (a_1+a_2+a_3)^2/3 \geq 1/3$ by the Cauchy-Schwarz inequality.

	\subsubsection*{Proving $(3-2L_4^2)/3 \leq 10/(27\sigma^2)$}
		This inequality is equivalent to $27\sigma^2-18(1+a_1-a_2-a_3)^2 \leq 10$. This indeed holds, as
		\begin{equation*}
		\begin{aligned}
			10 - 27\sigma^2 + 18(1+a_1-a_2-a_3)^2 =\ & 9(a_1-a_3)(4-4a_2-a_1-5a_3) + \\
													 & + 9(a_1-a_2)(4-5a_1-5a_2)+(3a_1-1)(33a_1-1) \geq 0,
		\end{aligned}
		\end{equation*}
		where the ultimate inequality uses the assumptions $a_3 \leq a_2 \leq a_1 \leq 0.4$ and $a_1 \geq 1/3$.

	\subsubsection*{Proving $10/(27\sigma^2) \leq (R_1^2-1)/2$, which implies $10/(27\sigma^2) \leq \min\{R_1^2-1,(R_2^2-1)/2\}$}
		This inequality reads as $27(1-a_1+a_2+a_3)^2 - 27\sigma^2 \geq 20$, and is satisfied, since
		\[
			27(1-a_1+a_2+a_3)^2 - 27\sigma^2 \geq 27\cdot (1+0.2)^2 - 27\cdot 2/3 > 20,
		\]
		where the first inequality uses $a_2+a_3 - a_1 = (a_1+a_2+a_3)-2a_1 \geq 1-2a_1 > 0.2$ and $a_1^2+a_2^2+a_3^2 \geq 1/3$.

	\subsubsection*{Proving $10/(27\sigma^2) \leq (R_3^2-1)/3$}
		This inequality is equivalent to $27(1+a_1+a_2-a_3)^2 - 27\sigma^2 \geq 30$. This indeed holds, as
		\[
			27(1+a_1+a_2-a_3)^2 - 27\sigma^2 \geq 27(1+1/3)^2 - 27\cdot 2/3 = 30,
		\]
		where the inequality uses $a_1 \geq 1/3$, $a_2 \geq a_3$ and $a_1^2+a_2^2+a_3^2 \geq 1/3$.

	\subsubsection*{Proving $10/(27\sigma^2) \leq (R_4^2-1)/4$}
		This inequality is equivalent to $27(1+a_1+a_2+a_3)^2 - 27\sigma^2 \geq 40$. This indeed holds, as
		\[
			27(1+a_1+a_2+a_3)^2 - 27\sigma^2 \geq 27\cdot 2^2 - 27\cdot 2/3 = 90,
		\]
		where the inequality uses $a_1+a_2+a_3\geq 1$ and $a_1^2+a_2^2+a_3^2 \geq 1/3$.

\subsection{Proof of Inequality~\tops{\eqref{eq:L34Ri}}}\label{app:L34Ri}
		Recall we are required to prove:
		\begin{equation*}
			\begin{gathered}
				L_4 \isgeq L_3 \isgeq \sqrt{1/2}, \qquad \forall i\in \{1,2,3,4\}\cc R_i \isgeq \sqrt{1+(i/2)}, \\
				\text{with:}\qquad 1-2a_1 \leq a_k \leq a_j \leq a_1 \leq 0.387,\quad \sigma = \sqrt{1-a_1^2-a_j^2-a_k^2}, \\
				L_3 = \frac{1-a_1+a_j-a_k}{\sigma}, \quad L_4 = \frac{1+a_1-a_j-a_k}{\sigma}, \quad R_1 = \frac{1-a_1+a_j+a_k}{\sigma}, \\
				R_2 = \frac{1+a_1-a_j+a_k}{\sigma}, \quad R_3 = \frac{1+a_1+a_j-a_k}{\sigma}, \quad R_4 = \frac{1+a_1+a_j+a_k}{\sigma}.
		\end{gathered}
		\end{equation*}

		\subsubsection*{Proving $L_4 \geq L_3 \geq \sqrt{1/2}$}
		
		The inequality $L_4 \geq L_3$ holds trivially, being equivalent to $a_j \leq a_1$. The inequality
		$L_3 \geq \sqrt{1/2}$ reads
		\[
		2(1-a_1+a_j-a_k)^2\geq 1-a_1^2-a_j^2-a_k^2.
		\]
		As $a_j \geq a_k \geq 1-2a_1$, it suffices to check $2(1-a_1)^2 \geq 1-a_1^2-2(1-2a_1)^2$, or equivalently, \[
		11a_1^2-12a_1+3 \geq 0.
		\]
		This indeed holds, as the two roots of the quadratic $11a_1^2-12a_1+3$ are $> 0.387$ (just barely).

		\subsubsection*{Proving $R_1 \geq \sqrt{3/2}$}

		This inequality is equivalent to $2(1-a_1+a_j+a_k)^2 - 3\sigma^2 \geq 0$, which holds as
		\begin{equation*}
		\begin{aligned}
			2(1-a_1+a_j+a_k)^2 - 3\sigma^2 =\  & (a_j-(1-2a_1))(9-14a_1+5a_j+4a_k) + \\
				                                 & + (a_k-(1-2a_1))(13-22a_1+5a_k) + 7(11a_1^2-12a_1+3) > 0.
		\end{aligned}
		\end{equation*}
		
		\subsubsection*{Proving $R_2 \geq \sqrt{4/2}$}
		
		This inequality is equivalent to $(1+a_1-a_j+a_k)^2 - 2\sigma^2 \geq 0$, which holds as
		\begin{equation*}
		\begin{aligned}
			(1+a_1-a_j+a_k)^2 - 2\sigma^2 =\  & (a_k - (1-2a_1))(5-4a_1-2a_j + 3a_k) + \\
			 								& + (a_1-a_j)(4-5a_1-3a_j) + (4a_1-2)^2 > 0.
		\end{aligned}
		\end{equation*}
		
		\subsubsection*{Proving $R_3 \geq \sqrt{5/2}$}
		
		This inequality is equivalent to $2(1+a_1+a_j-a_k)^2 - 5\sigma^2 \geq 0$, which holds since
		\[
			2(1+a_1+a_j-a_k)^2 - 5\sigma^2 = (a_j-a_k)(4+4a_1-3a_j-7a_k) + 10(a_j^2-(1-2a_1)^2) + (47a_1^2-36a_1+7),
		\]
		where the inequality $47a_1^2-36a_1+7 > 0$ holds for all $a_1$ as its discriminant is $36^2-4\cdot 7 \cdot 47 < 0$.
		
		\subsubsection*{Proving $R_4 \geq \sqrt{6/2}$}
		
		This inequality is equivalent to $(1+a_1+a_j+a_k)^2-3\sigma^2 \geq 0$, which holds since $a_1+a_j+a_k \geq a_1 + 2(1-2a_1) \geq 2-3a_1 > 0.8$, and thus,
		\[
			(1+a_1+a_j+a_k)^2-3\sigma^2 \geq 1.8^2 - 3 > 0.
		\]

\subsection{Proof of Inequality~\tops{\eqref{eq:everything}}}\label{app:everything}
	Recall we have to prove
	\begin{equation*}
	\begin{gathered}
			T_{12} \cdot \max(T_{10}, T_{17}) \isgeq 1, \\
			T_{18}, T_{20}, T_{24} \isgeq 1, \\
			T_{11}, T_{13}, T_{14} \isgeq \sqrt{1 + (3/3)}, \\
			T_{19}, T_{21}, T_{22}, T_{25}, T_{26}, T_{28}  \isgeq \sqrt{1 + (9/3)}, \\
			T_{15}, T_{23}, T_{27}, T_{29}, T_{30}, T_{31}  \isgeq \sqrt{1 + (15/3)}, \\
			\text{with:}\qquad T_{k} = \frac{1 - \sum_{i=1}^{5} (-1)^{\lfloor k/2^i \rfloor} a_i}{\sigma}, \qquad \sigma = \sqrt{1-\sum_{i=1}^{5} a_i^2}, \\
			1-a_2-a_4 \leq a_5 \leq a_4 \leq a_3 \leq a_2 \leq a_1 \leq 0.387.
	\end{gathered}
	\end{equation*}
	
	\paragraph{Simplification.} As $0<a_5 \leq a_4 \leq a_3 \leq a_2 \leq a_1$, it is easy to check that
	\begin{equation*}
		\begin{gathered}
			T_{18} \leq T_{20} \leq T_{24}, \\
			T_{11} \leq T_{13} \leq T_{14}, \\
			T_{19} \leq T_{21} \leq \min(T_{22}, T_{25}) \leq T_{26} \leq T_{28}, \\
			T_{15} \leq T_{23} \leq T_{27} \leq T_{29} \leq T_{30} \leq T_{31}.
		\end{gathered}
	\end{equation*}
	Hence, it is sufficient to verify the five inequalities
	$T_{12} \cdot \max(T_{10}, T_{17}) \geq 1$,
	$T_{18} \geq 1$, $T_{11} \geq \sqrt{2}$, $T_{19} \geq \sqrt{4}$, $T_{15} \geq \sqrt{6}$.
	
	\subsubsection*{Proving $T_{12} \cdot \max(T_{10}, T_{17}) \geq 1$}
	
	Note that the assumption $a_2+a_4+a_5 \geq 1$ implies $a_2 \geq a_4 \geq (1-a_2)/2$, and thus, $a_2 \geq 1/3$. Furthermore, using the assumptions $a_3 \geq a_4 \geq a_5$, $a_2+a_4+a_5 \geq 1$ and $a_2 \leq 0.4$, we have
	\[
	a_3+a_4+a_5 \geq \frac{3}{2} \cdot (a_4 +a_5) \geq \frac{3}{2} \cdot (1-a_2) \geq \frac{3}{2} \cdot 0.6 \geq a_2.
	\]
	Thus, we have
	\begin{equation}\label{Eq:Aux-Every1}
	a_2 \geq 1/3 \qquad \mbox{and} \qquad a_3+a_4+a_5-a_2 \geq 0.
	\end{equation}
	Now we can prove the required inequality. Since $\max(T_{10}, T_{17}) \geq (T_{10}+T_{17})/2$, it suffices to prove $T_{12} (T_{10}+T_{17})/2 \geq 1$, which reads $(1-a_1+a_2+a_3-a_4-a_5)(1-a_3) \geq \sigma^2$. This indeed holds, as
	\begin{equation*}
	\begin{gathered}
		(1-a_1+a_2+a_3-a_4-a_5)(1-a_3) -  \sigma^2 = \\
		(a_2+a_4+a_5-1)(-a_2+a_3+a_4+a_5) + 2(a_2 a_3-a_4 a_5) + \\
		+ (a_2-a_3)(3 a_2-1) + (a_1-a_2)(a_1+a_2+a_3-1) \geq 0,
	\end{gathered}
	\end{equation*}
	where all terms in the left hand side of the ultimate inequality are nonnegative by the assumptions $a_2+a_4+a_5 \geq 1$, $a_1 \geq a_2 \geq a_3 \geq a_4 \geq a_5$, and the inequality~\eqref{Eq:Aux-Every1}.
	
	\subsubsection*{Proving $T_{18} \geq 1$}
	
	This inequality holds since
	\[
		(1+a_1-a_2-a_3+a_4-a_5)^2 - \sigma^2 \geq (1-a_3)^2 - (1-2a_3^2-1/3) = 3(a_3-1/3)^2 \geq 0.
	\]
	Here, the first inequality uses the assumptions $a_1 \geq a_2$, $a_4 \geq a_5$, $a_1^2 \geq a_3^2$, and the inequality $a_2^2+a_4^2+a_5^2 \geq 1/3$, which follows from the assumption $a_2+a_4+a_5 \geq 1$ via the Cauchy-Schwarz inequality.
	
	\subsubsection*{Proving $T_{11} \geq \sqrt{2}$}
	
	This inequality holds since
	\[
		(1-a_1+a_2-a_3+a_4+a_5)^2 - 2\sigma^2 \geq (2-2\cdot 0.387)^2 - 2\cdot (1-1/3) > 0,
	\]
	where the first inequality uses $a_2+a_4+a_5 \geq 1$, $a_1+a_3 \leq 2 \cdot 0.387$, and $a_2^2+a_4^2+a_5^2 \geq 1/3$.
	
	\subsubsection*{Proving $T_{19} \geq \sqrt{4}$}
	
	This inequality holds since
	\begin{equation*}
	\begin{aligned}
		(1+a_1-a_2-a_3+a_4+a_5)^2 - 4\sigma^2 & \geq (2-a_2-a_3)^2 - 4(2/3-a_2^2-a_3^2) =  \\
											  & = 3(a_2+a_3-2/3)^2 + 2(a_2-a_3)^2 \geq 0,
	\end{aligned}
	\end{equation*}
	where the first inequality uses $a_1+a_4+a_5 \geq 1$ and $a_1^2+a_4^2+a_5^2 \geq 1/3$.
	
	\subsubsection*{Proving $T_{15} \geq \sqrt{6}$}
	
	Note that the assumption $a_2+a_4+a_5 \geq 1$ implies $a_3 \geq (a_4+a_5)/2 \geq (1-a_2)/2$. As $a_2 \leq a_1<0.4$, this in turn implies
	\[
	a_1-a_3 \leq a_1 - (1-a_2)/2 \leq 0.4-(1-0.4)/2 = 0.1.
	\]
	Therefore, we have $T_{15} \geq \sqrt{6}$, as
	\[
		\sigma^2 (T_{15}^2-6) = (1-a_1+a_2+a_3+a_4+a_5)^2 - 6\sigma^2 \geq (2-a_1+a_3)^2 - 6(2/3-a_1^2) \geq 1.9^2 - 6\cdot (5/9) > 0,
	\]
	where the first inequality uses $a_2+a_4+a_5 \geq 1$, $a_2^2+a_4^2+a_5^2 \geq 1/3$, and $a_3^2 \geq 0$, and the second inequality uses $a_1^2 \geq 1/9$ (which holds since $3a_1 \geq a_1+a_2+a_3\geq 1$) and $a_1-a_3 \leq 0.1$.
	
\section{Proof of Lemma~\ref{lem:31sfin}}\label{app:31sfin}
	Let us recall the assertion we have to prove. Let $X=\sum_{i=1}^n a_i x_i$ be a Rademacher sum with variance~$1$, where the weights $\{a_i\}$ satisfy $a_3 \leq a_2 \leq a_1 \leq 1/2$, $a_1+a_2+a_3 \leq 1$, $a_1 \geq 0.31$, $a_2\geq 0.19$, $a_3 \geq 0.15$. Let $X' = \sum_{i=4}^{n} a_i' x_i$, where $\forall i \cc a_i' = a_i / \sigma$ and $\sigma = \sqrt{1-a_1^2-a_2^2-a_3^2}$. Denote
	\[
		L_1, L_2, L_3, L_4 = \frac{1-a_1-a_2-a_3}{\sigma},
		\frac{1-a_1-a_2+a_3}{\sigma}, \frac{1-a_1+a_2-a_3}{\sigma}, \frac{1-|a_1-a_2-a_3|}{\sigma}
	\]
	and
	\[
		R_1, R_2, R_3, R_4 = \frac{1+|a_1-a_2-a_3|}{\sigma}, \frac{1+a_1-a_2+a_3}{\sigma},
		\frac{1+a_1+a_2-a_3}{\sigma}, \frac{1+a_1+a_2+a_3}{\sigma}.
	\]
	Furthermore, let
	\begin{equation*}
	\begin{gathered}
		(c_1, d_1, e_1) = (1-2L_3^2, L_3, \sqrt{1/2})\\
		(c_2, d_2, e_2) = (2-2L_4^2, L_4, \sqrt{2/2})\\
		(c_3, d_3, e_3) = (3-2R_1^2, R_1, \sqrt{3/2})\\
		(c_4, d_4, e_4) = (4-2R_2^2, R_2, \sqrt{4/2})\\
		(c_5, d_5, e_5) = (5-2R_3^2, R_3, \sqrt{5/2})\\
		(c_6, d_6, e_6) = (6-2R_4^2, R_4, \sqrt{6/2}).
	\end{gathered}
	\end{equation*}
	We have to prove three statements:
	\begin{enumerate}
		\item $\forall i\cc \hseg{d_i}{e_i} \prec_{X'} \hseg{0}{L_2}$, and thus by Theorem~\ref{thm:seg_compare}, $\hpr{X'}{d_i}{e_i} \leq \hpr{X'}{0}{L_2}$.
		\item $\sum_i \max(c_i, 0) \leq 3/2$.
		\item $\sum_{i\in B(X')} \max(c_i, 0) \leq 1$ where $B(X') = \li\{\given{i \in [6]}{\hseg{d_i}{e_i} \not \prec_{X'} \hseg{-L_1}{L_1}}\ri\}$.
	\end{enumerate}
	Unfortunately, the proof we present below is grueling and not enlightning.

\subsection{Proving \tops{$\hseg{d_i}{e_i} \prec_{X'} \hseg{0}{L_2}$}}
		First, note that if for some $i$, we have $e_i < d_i$, then clearly, $\hpr{X'}{d_i}{e_i} \leq \hpr{X'}{0}{L_2}$. Thus, we henceforth assume $d_i \leq e_i$.
		
		\medskip In order to prove $\hseg{d_i}{e_i} \prec_{X'} \hseg{0}{L_2}$, we have to verify:
		\begin{enumerate}
			\item[(a)] $0 \leq \min(L_2,d_i)$;
			
			\item[(b)] $2M \leq d_i-0$;
			
			\item[(c)] $e_i-d_i+\min(2M,e_i-L_2) \leq L_2-0$,
		\end{enumerate}
		where $M=\max_i \{ a_i'\} \leq a_3/\sigma$.

	\medskip \noindent The assertion~(a) holds trivially, as the assumption $a_1+a_2+a_3 \leq 1$ implies $L_i,R_i \geq 0$ for all $i$.
	
	\medskip \noindent The assertion~(b) holds for all $i$, since
		\[
			2M \leq \frac{2a_3}{\sigma} \leq \frac{2a_3}{\sigma} + \frac{1-a_1-a_2-a_3}{\sigma} = L_2 \leq d_i,
		\]
		where the second inequality uses again the assumption $a_1+a_2+a_3 \leq 1$.
	
	\medskip \noindent In the following paragraphs, we verify the assertion~(c) for each $i \in [6]$.
		
	\subsubsection*{Proving~(c) for~\tops{$i=1$}: Checking \tops{$\sqrt{1/2}-L_3 + (\sqrt{1/2}-L_2) \leq L_2$}}
		The inequality is equivalent to $\sqrt{2}\sigma \leq 3-3a_1-a_2+a_3$. Both sides are positive, so after squaring we should prove
		\[
		Q \defeq (3-3a_1-a_2+a_3)^2-2(1-a_1^2-a_2^2-a_3^2) \geq 0.
		\]
		This indeed follows from the assumptions $a_1+a_2 \leq 1$, $a_1\leq 1/2$, and $0 \leq a_3 \leq a_2$, as
		\[
			Q = (1-a_1-a_2)(3-3a_1-3a_2+6a_3) + 4(1-a_1)(1-2a_1) + a_3(4a_2+3a_3) \geq 0.
		\]
		
	\subsubsection*{Proving~(c) for~\tops{$i=2$}: Checking \tops{$1-L_4 + \min(2M,1-L_2) \leq L_2$}}
		We split into two sub-cases according to whether $a_3 \isleq 1/5$.
		\paragraph{Sub-Case~1: $a_3 \geq 1/5$.} We show that in this case, $1-L_4 + (1-L_2) \leq L_2$. This inequality reads $2 \leq 2L_2 + L_4$, or equivalently,
		\[
		2\sigma \leq 3-2a_1-2a_2+2a_3-|a_1-a_2-a_3|.
		\]
		Notice $-a_3\leq a_1-a_2-a_3 \leq 0.5-2\cdot 0.2\leq a_3$, and thus, $|a_1-a_2-a_3| \leq a_3$.
		Hence, it is sufficient to prove $2\sigma \leq 3-2a_1-2a_2+a_3$, or equivalently (since both sides are positive),
		\[
			Q \defeq (3-2a_1-2a_2+a_3)^2 - 4(1-a_1^2-a_2^2-a_3^2) \geq 0.
		\]
		This assertion indeed follows from the assumptions $a_1+a_2+a_3 \leq 1$, $a_3 \geq 1/5$, as
		\[
			Q = 6(1-a_1-a_2-a_3)^2 + 16a_3(1-a_1-a_2-a_3) + 2(a_1-a_2)^2 + (3a_3+1)(5a_3-1) \geq 0.
		\]

		\paragraph{Sub-Case~2: $a_3 \leq 1/5$.} We show that in this case,
		\[
		1-L_4 +2M \leq 1-L_4 + 2a_3/\sigma \leq L_2.
		\]
		Unfolding this, we have to prove $\sigma + a_1+a_2+a_3+|a_1-a_2-a_3| \leq 2$.
		
		\medskip \noindent If $a_1 \geq a_2+a_3$, this inequality reads $\sigma + 2a_1 \leq 2$, which holds since $a_1 \leq 1/2$ and $\sigma \leq 1$.
		
		\medskip \noindent If $a_1 \leq a_2+a_3$, we should prove $\sigma \leq 2 - 2a_2-2a_3$, or equivalently (as both sides are positive),
		\[
			Q \defeq 4(1-a_2-a_3)^2 - (1-a_1^2-a_2^2-a_3^2) \geq 0.
		\]
		This indeed follows from the assumptions $a_1+a_2+a_3\leq 1$, $0\leq a_3 \leq a_2 \leq a_1$, and $a_3 \leq 1/5$, as
		\[
			Q = (a_1-a_2)(a_1+a_2) + (1-5a_3)(1-a_2-a_3) + (2-3a_2)(1-2a_2-a_3) \geq 0.
		\]
		
	\subsubsection*{Proving~(c) for~\tops{$i=3$}: Checking \tops{$\sqrt{3/2}-R_1 + 2M \leq L_2$}}
	We prove the stronger assertion
	\[
	\sqrt{3/2} - \frac{1}{\sigma} + \frac{2a_3}{\sigma} \leq L_2,
	\]
	which is equivalent to the inequality $\sqrt{3/2}\sigma \leq 2-a_1-a_2-a_3$. Denote $t \defeq a_1+a_2+a_3 \leq 1$. By the Cauchy-Schwarz inequality, we have $a_1^2 + a_2^2 + a_3^2 \geq t^2/3$. Hence, the above inequality follows from
	\[
	\sqrt{\frac{3}{2}\cdot(1-t^2/3)} \leq 2-t.
	\]
	By squaring, the latter inequality is equivalent to $3t^2-8t+5\geq 0$, which indeed holds for any $t \leq 1$, as $3t^2-8t+5 = (1-t)(5-3t)\geq 0$.
	
	\subsubsection*{Proving~(c) for~\tops{$i=4$}: Checking \tops{$\sqrt{2}-R_2 + 2M \leq L_2$}}
	We prove the stronger assertion $\sqrt{2}-R_2 + 2a_3/\sigma \leq L_2$, or equivalently, $\sigma \leq \sqrt{2} (1-a_2)$. By squaring, we should prove
	\[
	Q\defeq 2(1-a_2)^2-(1-a_1^2-a_2^2-a_3^2) \geq 0.
	\]
	This indeed holds, as $Q = (1-2a_2)^2 + (a_1^2-a_2^2)+a_3^2 \geq 0$.

	\subsubsection*{Proving~(c) for~\tops{$i=5$}: Checking \tops{$\sqrt{5/2}-R_3 + 2M \leq L_2$}}
	We prove the stronger assertion $\sqrt{5/2}-R_3 + 2a_3/\sigma \leq L_2$, or equivalently,
	\[
	\sqrt{5/2} \cdot \sigma \leq 2 (1-a_3).
	\]
	As $a_3 \leq a_2 \leq a_1$, this inequality follows from
	\[
	\sqrt{5/2} \cdot \sqrt{1-3a_3^2} \leq 2(1-a_3).
	\]
	The latter inequality is equivalent (via squaring) to the quadratic inequality  $23a_3^2-16a_3+3 \geq 0$, that holds for any $a_3 \in \mathbb{R}$, as its discriminant is $\Delta = -20 < 0$.
	
	\subsubsection*{Proving~(c) for~\tops{$i=6$}: Checking \tops{$\sqrt{3}-R_4 + 2M \leq L_2$}}
	This inequality follows from the stronger inequality $\sqrt{3}-R_4 + 2a_3/\sigma \leq L_2$, that is equivalent to $\sqrt{3}\sigma \leq 2$. The latter inequality is clear, as $\sigma <1 <2/\sqrt{3}$.
		
\subsection{Proving \tops{$\sum_i \max(c_i, 0) \leq 3/2$}}
	
		\paragraph{Reminder of the assumptions and notation.}
		Recall the region we consider is $0.15 \leq a_3 \leq a_2 \leq a_1 \leq 1/2$, $a_1+a_2+a_3\leq 1$, $a_2 \geq 0.19$, $a_1 \geq 0.31$. In particular, in our region we have \begin{equation}\label{Eq:Aux-C2}
		\sigma^2 \leq 1-0.31^2 - 0.19^2-0.15^2 = 0.8453.
		\end{equation}
		Throughout the proof below, we write $c_i' = \sigma^2 c_i$.
	
\paragraph{Proof strategy.}	We prove the following three statements, which together clearly imply the assertion.
	\begin{enumerate}
		\item[(a)] $\max(c_1,0)+\max(c_3,0) \leq 0.58$;
		
		\item[(b)] $\max(c_2,0)+\max(c_4,0) +\max(c_5,0) \leq 0.92$;
		
		\item[(c)] $c_6 \leq 0$.
	\end{enumerate}
	
\noindent The assertion~(c) is immediate. Indeed, as $a_1 \geq 0.31$, $a_2 \geq 0.19$, $a_3 \geq 0.15$, we have
\[
c_6' = 6(1-a_1^2-a_2^2-a_3^2) - 2(1+a_1+a_2+a_3)^2 < 6(1-0.31^2) - 2(1+0.31+0.19+0.15)^2 < 0,
\]
and hence, $c_6 = c_6'/\sigma^2 \leq 0$. It is thus left to prove~(a) and~(b).
	
\medskip \noindent In the proof we use the following obvious claim.
	\begin{claim}\label{Cl:AuxC.1}
	Let $T$ be a finite set, and let $\{c_i\}_{i \in T}$ be a set of values, and $B > 0$. In order to prove $\sum_{i\in T} \max(c_i, 0) \leq B$, it is sufficient to verify $2^{|T|}-1$ inequalities:
	\[
	\forall S \seq T,\ S \neq \emptyset \cc \sum_{i \in S} c_i \leq B.
	\]
	\end{claim}
	
	\subsubsection{Proving \tops{$\max(c_1, 0) + \max(c_3, 0) \leq 0.58$}}\label{ssec:1358}
		By Claim~\ref{Cl:AuxC.1}, it is sufficient to verify $c_1 \leq 0.58$, $c_3 \leq 0.58$ and $c_1 + c_3 \leq 0.58$.
		
	\subsubsection*{Checking $c_1 \leq 0.58$}
		This inequality is equivalent to $0.58\sigma^2 - c_1' \geq 0$. We have
		\[
			0.58\sigma^2 - c_1' = 2 (1-a_1+a_2-a_3)^2 - 0.42\sigma^2 \geq 2(1-a_1)^2 - 0.42 \geq 0.08>0,
		\]
		where we used the assumptions $a_3 \leq a_2$, $a_1 \leq 1/2$, and $\sigma \leq 1$.
		
	\subsubsection*{Checking $c_3 \leq 0.58$}
		This inequality is equivalent to $0.58\sigma^2 - c_3' \geq 0$. That is, we should prove
		\[
			0.58\sigma^2 - c_3' = 2(1 + |a_1-a_2-a_3|)^2 -2.42\sigma^2 \geq 0.
		\]
		If $a_1 \geq 0.34$, then
		\[
			2-2.42\sigma^2 \geq 2-2.42(1-0.34^2-0.19^2-0.15^2) > 0,
		\]
		and the assertion follows.
		
		\medskip \noindent If $a_1 \leq 0.34$ then
		$|a_1-a_2-a_3| \geq 0.34-a_1$, and thus, it is sufficient to prove
		\[
			2(1.34-a_1)^2 - 2.42(1-a_1^2-0.19^2-0.15^2) > 0.
		\]
		This quadratic inequality indeed holds for all $a_1 \leq 0.34$.

	\subsubsection*{Checking \tops{$c_1+c_3 \leq 0.58$}}
		This inequality is equivalent to $0.58\sigma^2 - c_3' - c_1' \geq 0$. Since $a_2 \geq 0.19$, $a_3 \geq 0.15$, we have
		\begin{align*}
			0.58\sigma^2 - c_3' - c_1'
			& = 2(|a_1-a_2-a_3|+1)^2 + 2(1-a_1+a_2-a_3)^2 - 3.42 \sigma^2 \\
			& \geq 2 + 2(1-a_1)^2 - 3.42(1-a_1^2-0.19^2-0.15^2)\\
			& = 5.42(a_1-100/271)^2 + 2872913/67750000 > 0.
		\end{align*}

	\subsubsection{Proving \tops{$\max(c_2, 0) + \max(c_4, 0) + \max(c_5, 0) \leq 0.92$}}\label{ssec:c24592}
		By Claim~\ref{Cl:AuxC.1}, it is sufficient to verify the seven inequalities $c_2 \leq 0.92$, $c_4 \leq 0.92$, $c_5 \leq 0.92$, $c_2 + c_4 \leq 0.92$, $c_2 + c_5 \leq 0.92$, $c_4 + c_5 \leq 0.92$, $c_2 + c_4 + c_5 \leq 0.92$.
		\subsubsection*{Checking $c_2 \leq 0.92$}
			This inequality is equivalent to $0.92\sigma^2 - c_2' \geq 0$, which is in turn equivalent to
			\[
			2(1 - |a_1-a_2-a_3|)^2 - 1.08\sigma^2 \geq 0.
			\]
			In our region, $a_1-a_2-a_3 \leq 0.16 \leq a_3+0.01$ and $-a_1+a_2+a_3 \leq a_3$, and thus, $|a_1-a_2-a_3|\leq a_3+0.01$. As $\sigma^2 \leq 1-0.31^2-0.19^2-a_3^2$, we have
			\[
				2(1 - |a_1-a_2-a_3|)^2 - 1.08\sigma^2 \geq 2(0.99-a_3)^2 - 1.08 (0.87-a_3^2) > 0,
			\]
			 where the ultimate inequality holds since $a_3 \leq (a_1+a_2+a_3) / 3 \leq 1/3$.
		
		\subsubsection*{Checking $c_4 \leq 0.92$}
			This inequality is equivalent to $0.92\sigma^2 - c_4' \geq 0$, which is in turn equivalent to
			\[
			2(1+a_1-a_2+a_3)^2-3.08\sigma^2 \geq 0.
			\]
			As $\sigma^2 \leq 0.8453$ by~\eqref{Eq:Aux-C2} and $1+a_1-a_2+a_3 \geq 1+a_3 \geq 1.15$, we have
			\[
			2(1+a_1-a_2+a_3)^2-3.08\sigma^2 \geq 2\cdot 1.15^2-3.08\cdot 0.8453 > 0.
			\]
		\subsubsection*{Checking $c_5 \leq 0.92$}
			This inequality is equivalent to $0.92\sigma^2 - c_5' \geq 0$, which is in turn equivalent to
			\begin{equation}\label{Eq:Aux-C3}
			2(1+a_1+a_2-a_3)^2-4.08\sigma^2 \geq 0.
			\end{equation}
			Write $d = a_2 - a_3$ and recall $d \geq 0$. We claim that
			\[
			a_3 \geq |0.19-d|.
			\]
			Indeed, on the one hand we have $a_3=a_2-d \geq 0.19-d$, since $a_2 \geq 0.19$. On the other hand, as $a_1+a_2+a_3 \leq 1$, we have $a_2 \leq (1-a_3)/2$, and thus,
			\[
			a_3-(d-0.19)=2a_3-a_2+0.19 \geq 2a_3-\frac{1-a_3}{2}+0.19=2.5a_3-0.31 > 0,
			\]
			where the ultimate inequality holds since $a_3 \geq 0.15$. As $a_1 \geq 0.31$, in order to prove~\eqref{Eq:Aux-C3} it is sufficient to show
			\[
			2(1.31+d)^2-4.08(1-0.31^2-0.19^2-(0.19-d)^2) \geq 0.
			\]
			This inequality, which reads $6.08d^2 + 3.6896d + 0.038864 \geq 0$, indeed holds for all $d \geq 0$.

		\subsubsection*{Checking $c_2 + c_4 \leq 0.92$}
			This inequality is equivalent to $0.92\sigma^2 - c_2' - c_4' \geq 0$, which is in turn equivalent to the inequality
			\[
				2(1+a_1-a_2+a_3)^2 + 2(1-|a_1-a_2-a_3|)^2-5.08\sigma^2 \geq 0.
			\]
			As in the proof of $c_2 \leq 0.92$ above, we may use the bound $|a_1-a_2-a_3| \leq 0.01 + a_3$, and so, it is sufficient to prove
			\[
			2(1+a_1-a_2+a_3)^2 + 2(0.99-a_3)^2-5.08\sigma^2 \geq 0.
			\]
			Differentiating the left hand side with respect to $a_3$, we get $4(0.01+a_1-a_2+4.54a_3)$ which is clearly positive. Hence, it suffices to verify the inequality for the minimal possible $a_3$ in our region, that is, for $a_3=0.15$.
			
			\medskip \noindent Write $d=a_1-a_2$. We clearly have $0 \leq 0.31-d \leq a_2$. Hence, it is sufficient to prove
			\[
				2(1+d+0.15)^2 + 2(0.99-0.15)^2 - 5.08(1-0.31^2-(0.31-d)^2-0.15^2) \geq 0,
			\]
			or equivalently, $7.08d^2 + 1.4504d + 0.066876 \geq 0$, which indeed holds for all $d\geq 0$.

		\subsubsection*{Checking $c_2 + c_5 \leq 0.92$}
			This inequality is equivalent to $0.92\sigma^2 - c_2' - c_5' \geq 0$, which is in turn equivalent to
			\begin{equation}\label{Eq:Aux-C4}
				2(1+a_1+a_2-a_3)^2 + 2(1-|a_1-a_2-a_3|)^2 - 6.08\sigma^2 \geq 0.
			\end{equation}
			The proof splits into two sub-cases according to whether $a_1 \isgeq a_2+a_3$.
			
			\paragraph{Sub-case~1: $a_1 \geq a_2+a_3$.} In this case we have $1+a_1+a_2-a_3 \geq 1+2a_2 \geq 1.38$ and $|a_1-a_2-a_3| \leq 0.5-0.19-0.15=0.16$. Recalling $\sigma \leq 0.8453$, we deduce
			\begin{align*}
				0.92\sigma^2 - c_2' - c_5' &= 2(1+a_1+a_2-a_3)^2 + 2(1-|a_1-a_2-a_3|)^2 - 6.08\sigma^2\geq \\
				&\geq 2\cdot 1.38^2 + 2(1-0.16)^2 - 6.08 \cdot 0.8453 > 0.
			\end{align*}
			
			\paragraph{Sub-case~2: $a_1 \leq a_2+a_3$.} In this case, the inequality~\eqref{Eq:Aux-C4} reads
			\[
			2(1+a_1+a_2-a_3)^2 + 2(1+a_1-a_2-a_3)^2 - 6.08\sigma^2 \geq 0,
			\]
			or equivalently, $4(1+a_1-a_3)^2+4a_2^2 - 6.08\sigma^2 \geq 0$.
			Write $d=a_1-a_3$, so that $a_2 \geq a_3 \geq |0.31-d|$. We have
			\begin{equation*}
			\begin{aligned}
				0.92\sigma^2 - c_2' - c_5'
				&= 4(1+a_1-a_3)^2+4a_2^2 - 6.08\sigma^2 \\
				&\geq 4(1+d)^2+4(0.31-d)^2-6.08(1-0.31^2-2(0.31-d)^2) \\
				&= 20.16d^2 - 2.0192d + 0.057264 > 0,
			\end{aligned}
			\end{equation*}
			where the ultimate inequality holds since the quadratic polynomial in $d$ has $\Delta < -0.54 < 0$.

		\subsubsection*{Checking $c_4 + c_5 \leq 0.92$}
			This inequality is equivalent to $0.92\sigma^2 - c_4' - c_5' \geq 0$. We have
			\[
			0.92\sigma^2 - c_4' - c_5' = 2(1+a_1-a_2+a_3)^2+2(1+a_1+a_2-a_3)^2-8.08\sigma^2=4(1+a_1)^2 + 4(a_2-a_3)^2 - 8.08 \sigma^2.
			\]
			Using the inequalities $1+a_1 \geq 1.31$ and $\sigma^2 \leq 0.8453$, we get
			\[
				0.92\sigma^2 - c_4' - c_5' = 4(1+a_1)^2 + 4(a_2-a_3)^2 - 8.08 \sigma^2 \geq 4\cdot 1.31^2 - 8.08\cdot 0.8453 > 0.
			\]
			
		\subsubsection*{Checking $c_2 + c_4 + c_5 \leq 0.92$}
			This inequality is equivalent to $0.92\sigma^2 - c_2' - c_4' - c_5' \geq 0$, which is in turn equivalent to
			\[
				4(1+a_1)^2 + 4(a_2-a_3)^2 + 2(1-|a_1-a_2-a_3|)^2 - 10.08 \sigma^2 \geq 0.
			\]
			The proof splits to two sub-cases, according to whether $a_1 \isgeq a_2+a_3$.
			
			\paragraph{Sub-case~1: $a_1 \geq a_2+a_3$.} In this case, $|a_1 - a_2-a_3| \leq 0.5-0.19-0.15 \leq 0.16$, and hence, it is sufficient to verify
			\[
			4(1+a_1)^2+2\cdot0.84^2-10.08\sigma^2\geq 0.
			\]
			Since $a_1 \geq a_2+a_3 \geq 0.34$ and $\sigma^2 \leq 0.8453$, we have
			\[
			4(1+a_1)^2+2\cdot0.84^2-10.08\sigma^2 \geq 4 \cdot 1.34^2 +2 \cdot 0.84^2-10.08 \cdot 0.8453 >0.
			\]
			\paragraph{Sub-case~2: $a_1 \leq a_2+a_3$.} In this case, we should prove
			\[
				Q \defeq 4(1+a_1)^2 + 4(a_2-a_3)^2 + 2(1+a_1-a_2-a_3)^2 - 10.08 (1-a_1^2-a_2^2-a_3^2) \geq 0.
			\]
			This indeed holds, as
			\begin{equation*}
			\begin{aligned}
				Q =\ &(a_1-0.3)(16.08 a_1-4a_2-4a_3+16.824) + 7.04(a_2+a_3-65/176)^2 + \\
					 & + 9.04 (a_2-a_3)^2 + 767/110000 > 0.
			\end{aligned}
			\end{equation*}

\paragraph{Summarizing.} Combining the above bounds, we have
\begin{align*}
\sum_{i=1}^6 \max(c_i,0) &= \left(\max(c_1,0)+\max(c_3,0)\right)+\left(\max(c_2,0)+\max(c_4,0)+\max(c_5,0)\right) + \max(c_6,0) \\
&\leq 0.58+0.92+0=1.5,
\end{align*}
as asserted.

\subsection{Proving \tops{$\sum_{i\in B(X')} \max(c_i, 0) \leq 1$}}
	
	The proof consists of two steps. First, we show that either $B(X') \subseteq \{1,3,5,6\}$ or $B(X') \subseteq \{2,4,5,6\}$. Then, we verify the assertion in each of these cases separately.
	
	\subsubsection{Proving that \tops{$B(X') \subseteq \{1,3,5,6\}$} or \tops{$B(X') \subseteq \{2,4,5,6\}$}}
	To prove this, it is sufficient to show that for any $(i,j) \in \{1,3\}\times \{2,4\}$ we have
	\begin{equation}\label{Eq:Aux-C5}
	\hseg{d_i}{e_i} \prec_{X'} \hseg{-L_1}{L_1} \qquad \mbox{or} \qquad \hseg{d_j}{e_j} \prec_{X'} \hseg{-L_1}{L_1}.
	\end{equation}
	By Lemma~\ref{lem:seg-compare1}, in order to prove
	$\hseg{d_k}{e_\ell} \prec_{X'} \hseg{-L_1}{L_1}$
	for some $k,\ell$, it is sufficient to show
	\[
	e_\ell \leq d_k \qquad \mbox{or} \qquad e_\ell-d_k+2\frac{a_3}{\sigma} \leq 2L_1.
	\]
	(Note that the other assumption of Lemma~\ref{lem:seg-compare1}, namely, $L_1 \leq d_k$, holds in our case for all $k$, by the definition of the $d_k$'s.) The latter condition can be rewritten as
	\begin{equation}\label{Eq:Aux-C5.5}
	e_\ell-d_k \leq 2L', \qquad \mbox{with} \qquad L'=\frac{1-a_1-a_2-2a_3}{\sigma}.
	\end{equation}
	Hence, in order to verify~\eqref{Eq:Aux-C5} for some $(i,j)$, we have to show
	\begin{equation}\label{Eq:Aux-C6}
	\min(e_i - d_i, e_j - d_j) \leq \max(0, 2L').
	\end{equation}
	In the following paragraphs we show this for all $(i,j) \in \{1,3\}\times \{2,4\}$.
	
	\subsubsection*{Case~1: $i=1$ and $j=2$}
		We prove a stronger inequality: $\min(e_1-d_1, e_2-d_2) \leq 0$, or equivalently,
		\[
			\max\li( 2(1-a_1+a_2-a_3)^2, (1-|a_1-a_2-a_3|)^2 \ri) \geq \sigma^2.
		\]
		If $a_1 \geq a_2 + a_3$, it is sufficient to show $(1-a_1+a_2+a_3)^2 - \sigma^2 \geq 0$. This indeed holds, as
		\begin{equation*}
		\begin{aligned}
			(1-a_1+a_2+a_3)^2 - \sigma^2 =\  & (a_3-0.15)(2.3+2a_3+2a_2-2a_1) + (a_2-0.19)(2.68+2a_2-2a_1) \\
										   & + (1-2a_1)(0.84-a_1) + 0.0142 > 0.
		\end{aligned}
		\end{equation*}
		
		\noindent If $a_1 \leq a_2 + a_3$, it is sufficient to prove
		\[
		Q \defeq (1-a_1+a_2-a_3)^2 + \frac{1}{2}(1+a_1-a_2-a_3)^2 - \sigma^2 \geq 0,
		\]
		as a maximum between two quantities is no smaller than their average. This indeed holds, as
		\[
			Q = 2(1-a_3)(a_2-a_3) + 2(a_1-a_2)^2 + (1-a_1-a_2-a_3)^2/2 \geq 0.
		\]
		
		\subsubsection*{Case~2: $i=1$ and $j=4$}
			We prove a stronger inequality: $e_1-d_1+e_4-d_4 \leq 0$, or equivalently,
			$
				\frac{3}{\sqrt{2}} - \frac{2}{\sigma} \leq 0.
			$
			Since $\sigma \leq 0.8453$, we have
			\[
			\frac{3}{\sqrt{2}} - \frac{2}{\sigma} \leq \frac{3}{\sqrt{2}}-\frac{2}{0.8453} < 0.
			\]
		
		\subsubsection*{Case~3: $i=3$ and $j=2$}
			We prove a stronger inequality: $e_3-d_3 + e_2-d_2 \leq L'/2$, or equivalently,
			\[
			\sqrt{3/2}+1-2/\sigma \leq (1-a_1-a_2-2a_3)/(2\sigma).
			\]
			This latter is further equivalent to $(\sqrt{6}+2)\sigma \leq 5-a_1-a_2-2a_3$. By squaring and using the assumption $a_2 \geq a_3$, it is sufficient to prove
			\[
				f(a_1, a_2, a_3) \defeq (5-a_1-1.5(a_2+a_3))^2 - (10+4\sqrt{6}) \sigma^2 \geq 0.
			\]
			It is easy to verify that $\frac{\partial f}{\partial a_1} = (8\sqrt{6}+22)a_1 + 3a_2 + 3a_3 -10$, which is positive in our region since $a_1 \geq 0.31$ and $a_2,a_3 > 0$. Hence, $f(0.31, a_2, a_3) \leq f(a_1, a_2, a_3)$. Furthermore, by the definition of $\sigma$ and the convexity of the function $x \mapsto x^2$, we have
			\[
			f(a_1, (a_2+a_3)/2, (a_2+a_3)/2) \leq f(a_1, a_2, a_3).
			\]
			Hence, writing $t=(a_2+a_3)/2$, we conclude with
			\[
				f(a_1, a_2, a_3) \geq f(0.31, t, t) = (29+8\sqrt{6})\li(t-\frac{14.07}{8\sqrt{6}+29}\ri)^2+\frac{180.4126 - 68.61 \cdot \sqrt{6}}{457} > 0.
			\]

		\subsubsection*{Case~4: $i=3$ and $j=4$}
		We prove a stronger inequality: $e_3-d_3 + e_4-d_4 \leq L'$, or equivalently, $3 - 2a_2 - a_3 + |a_1-a_2-a_3| \geq ((2+\sqrt{3})/\sqrt{2}) \cdot \sigma$. By squaring, this is equivalent to
		\[
			(3 - 2a_2 - a_3 + |a_1-a_2-a_3|)^2 - (\sqrt{12} + 7/2)\sigma^2 \geq 0.
		\]
		If $a_1 \geq a_2 + a_3$, we should prove
		\[
		f(a_1, a_2, a_3) \defeq (3+a_1-3a_2-2a_3)^2 - (\sqrt{12} + 7/2)\sigma^2 \geq 0.
		\]
		It is clear that $f(a_1, a_2, a_3) \geq f(a_2+a_3, a_2, a_3)$, and thus, we may assume $a_1 = a_2+a_3$, which is covered in the other case.
		
		\medskip \noindent If $a_1 \leq a_2 + a_3$, the inequality reads
		\[
		f(a_1, a_2, a_3) \defeq (3-a_1-a_2)^2-(\sqrt{12} + 7/2)\sigma^2 \geq 0.
		\]
		It is clear that $f(a_1, a_2, 0.15) \leq f(a_1, a_2, a_3)$. Furthermore, by the convexity of the function $x \mapsto x^2$, we have  $f((a_1+a_2)/2, (a_1+a_2)/2, a_3) \leq f(a_1, a_2, a_3)$.
		Hence, writing $t = (a_1+a_2)/2$, we conclude with
		\[
			f(a_1, a_2, a_3) \geq f(t,t,0.15) = (11+4\sqrt{3})\li(t-\frac{6}{4\sqrt{3} + 11}\ri)^2 + \frac{10.28 \cdot \sqrt{3}+89.99}{584} > 0.
		\]
	
	\subsubsection{Proving \tops{$\sum_{i\in B(X')} \max(c_i, 0) \leq 1$} in the case \tops{$B(X') \subseteq \{2,4,5,6\}$}}
	By the inequality $\sum_{i \in \{2,4,5\}} \max(c_i,0) \leq 0.92$ proved in Subsection~\ref{ssec:c24592} and the easy inequality $c_6 \leq 0$ proved above, we see that $B(X') \subseteq \{2,4,5,6\}$ implies
	\[
	\sum_{i\in B(X')} \max(c_i, 0) \leq \max(c_2, 0)+\max(c_4, 0)+\max(c_5, 0)+\max(c_6, 0) \leq 0.92 < 1,
	\]
	as asserted.
	
	\subsubsection{Proving \tops{$\sum_{i\in B(X')} \max(c_i, 0) \leq 1$} in the case \tops{$B(X') \subseteq \{1,3,5,6\}$}, \tops{$a_3 \leq 0.2$}}
	We show that in this case, we actually have $5 \notin B(X')$, and hence $B(X') \seq \{1,3,6\}$. Thus, the inequality $\max(c_1,0)+\max(c_3,0) \leq 0.58$ proved in Subsection~\ref{ssec:1358} and the inequality $c_6 \leq 0$ imply $\sum_{i\in B(X')} \max(c_i, 0) \leq 0.58 < 1$, as required.
	
	\paragraph{Showing $5 \not \in B(X')$.}
	
	To show this, we have to verify $\hseg{d_5}{e_5} \prec_{X'} \hseg{-L_1}{L_1}$.
	By~\eqref{Eq:Aux-C5.5}, it is sufficient to show $e_5-d_5 \leq \max(0, 2L')$. We show the stronger inequality $e_5-d_5 \leq 1.5L'$. The inequality reads
	\[
		\sqrt{5/2} - \frac{1+a_1+a_2-a_3}{\sigma} \leq \frac{3(1-a_1-a_2-2a_3)}{2\sigma},
	\]
	or equivalently, $(5-a_1-a_2-8a_3)^2 - 10\sigma^2 \geq 0$. This indeed holds, as the left hand side is
	\[
		(0.2-a_3)(65.2-16a_1-16a_2-74a_3) + 6(a_1+a_2-17/30)^2+5(a_1-a_2)^2+1/30 > 0.
	\]

\subsubsection{Proving \tops{$\sum_{i\in B(X')} \max(c_i, 0) \leq 1$} in the case \tops{$B(X') \subseteq \{1,3,5\}$}, \tops{$a_3 \geq 0.2$}}
	As $c_6 \leq 0$, it follows from Claim~\ref{Cl:AuxC.1} that in order to prove the assertion, it is sufficient to verify the seven inequalities $c_1\leq 1$, $c_3\leq 1$, $c_5\leq 1$, $c_1+c_3 \leq 1$, $c_1+c_5 \leq 1$, $c_3+c_5 \leq 1$, $c_1+c_3+c_5 \leq 1$.
	
	\subsubsection*{Proving $c_1\leq 1$, $c_3\leq 1$, $c_5\leq 1$, $c_1+c_3 \leq 1$} This was already done in Subsections~\ref{ssec:1358} and ~\ref{ssec:c24592}.
	
	\subsubsection*{Proving $c_1+c_5 \leq 1$} This inequality is equivalent to $\sigma^2-c_1'-c_5' \geq 0$, which is in turn equivalent to
	\[
		2(1+a_1+a_2-a_3)^2 + 2(1-a_1+a_2-a_3)^2 - 5\sigma^2 \geq 0.
	\]
	By rearranging, we have to show $4(1+a_2-a_3)^2 + 4a_1^2 - 5\sigma^2 \geq 0$. The left hand side is at least $4+4\cdot 0.31^2 - 5 \cdot 0.8453 > 0.15 > 0$, as required.

	\subsubsection*{Proving $c_3+c_5 \leq 1$} This inequality is equivalent to $\sigma^2-c_3'-c_5' \geq 0$, which is in turn equivalent to
	\[
		2(1+a_1+a_2-a_3)^2+2(1+|a_1-a_2-a_3|)^2-7\sigma^2 \geq 0.
	\]
	Thus, it is sufficient to show $2(1+a_1)^2+2(1-a_1+a_2+a_3)^2-7\sigma^2 \geq 0$. Since $a_2 \geq a_3 \geq 0.2$, it is enough to verify
	\[
	2(1+a_1)^2+2(1.4-a_1)^2-7(0.92-a_1^2)\geq 0.
	\]
	This indeed holds for all $a_1 \geq 0.31$, as the left hand side is equal to  $(a_1-0.31)(11a_1+1.81)+0.0411 > 0$.
	
	\subsubsection*{Proving $c_1+c_3+c_5 \leq 1$}
	This inequality is equivalent to $\sigma^2-c_1'-c_3'-c_5' \geq 0$, which is in turn equivalent to
	\[
		2(1+a_1+a_2-a_3)^2 + 2(1-a_1 + a_2 - a_3)^2 + 2(1+|a_1-a_2-a_3|)^2 -8\sigma^2 \geq 0.
	\]
	As $a_2 \geq a_3 \geq 0.2$, it is sufficient to verify
	\[
	(1+a_1)^2+(1-a_1)^2+(1.4-a_1)^2-4(0.92-a_1^2) \geq 0.
	\]
	This holds for all $a_1$, as the left hand side is equal to $7a_1^2-2.8a_1+0.28 = 7(a_1-0.2)^2 \geq 0$.

\end{appendices}


\begin{thebibliography}{99}
\bibitem{BNR02} A. Ben-Tal, A. Nemirovski, and C. Roos, Robust solutions of uncertain quadratic and conic-quadratic
problems, {\it SIAM J. Optimization}, \textbf{13(2)} (2002), pp.~535--560.

\bibitem{BD15} V. K. Bentkus and D. Dzindzalieta, A tight Gaussian bound for weighted sums of Rademacher random variables,
    {\it Bernoulli}, \textbf{21(2)} (2015), pp.~1231--1237.

\bibitem{BG97} V. Bentkus and F. G\"{o}tze, Uniform rates of convergence in the {CLT} for quadratic forms in multidimensional spaces,
{\it Probab. Theory Relat. Fields}, \textbf{109} (1997), pp.~367--416.

\bibitem{BGZ97}  V. Bentkus, F. G\"{o}tze, and W. R. van Zwet, An Edgeworth expansion for symmetric statistics, {\it Ann. Statistics}, \textbf{25(2)} (1997), pp.~851--896.

\bibitem{Berry41} A. C. Berry, The accuracy of the Gaussian approximation to the sum of independent variates, {\it Trans. Amer. Math. Soc.}, \textbf{49(1)} (1941), pp.~122--136.

\bibitem{BG02} M. Blonzelis and F. G\"{o}tze, An Edgeworth expansion for symmetric finite population statistics, {\it Ann. Probab.}, \textbf{30(3)} (2002), pp.~1238--1265.

\bibitem{Bobkov16} S. G. Bobkov, Proximity of probability distributions in terms of Fourier-Stieltjes transforms, {\it Russian Math. Surveys}, \textbf{71(6)} (2016), pp.~1021--1079.

\bibitem{BGH01} S. G. Bobkov, F. G\"{o}tze, and C. Houdr\'{e}, On Gaussian and Bernoulli covariance representations,
    {\it Bernoulli}, \textbf{7} (2001), pp.~439--451.

\bibitem{BHZ20} R. B. Boppana, H. Hendriks, and M. C. A. van Zuijlen, Tomaszewski's problem on
randomly signed sums, revisited, {\it Electron. J. Combin.}, \textbf{28(2)} (2021), P2.35.

\bibitem{BH17} R. B. Boppana, R. Holzman, Tomaszewski's problem on randomly signed sums: Breaking the 3/8 barrier,
    {\it Electron. J. Combin.}, \textbf{24(3)} (2017), P3.40.
	
\bibitem{DDS16} A. De, I. Diakonikolas, and R. A. Servedio, A robust Khintchine inequality, and algorithms
    for computing optimal constants in Fourier analysis and high-dimensional geometry,
    {\it SIAM J. Disc. Math.}, \textbf{30(2)} (2016), pp.~1058--1094.

\bibitem{DL08} L. Devroye and G. Lugosi, Local tail bounds for functions of independent random variables, {\it Ann.
Probab.}, \textbf{36} (2008), pp.~143--159.

\bibitem{DHT20} V. Dvo{\v{r}}{\'{a}}k, P. van Hintum, and M. Tiba, Improved bound for Tomaszewski's problem, 
{\it  SIAM J. Discret. Math.}, \textbf{34(4)} (2020), pp.~2239--2249.

\bibitem{DK21} V. Dvo{\v{r}}{\'{a}}k and O. Klein, Probability mass of Rademacher sums beyond one standard deviation, preprint, 2021. Available at: https://arxiv.org/abs/2104.10005.

\bibitem{Dzindzalieta14a} D. Dzindzalieta, A note on random signs, {\it Lith. Math. J.}, \textbf{54(4)} (2014), pp.~403--408.

\bibitem{Dzindzalieta14} D. Dzindzalieta, Tight Bernoulli tail probability bounds, PhD thesis, Vilnius
    University, 2014. Available at http://talpykla.elaba.lt/elaba-fedora/objects/elaba:2121206/datastreams/MAIN/content.

\bibitem{DG20} D. Dzindzalieta and F G\"{o}tze,  Half-spaces with influential variable, \textit{Theory Probab. Appl.}, \textbf{65(1)} (2020), pp.~114--120.

\bibitem{Eaton70} M. L. Eaton, A note on symmetric Bernoulli random variables, {\it Ann. Math. Statist.}, \textbf{41} (1970), pp.~1223--1226.

\bibitem{Efron69} B. Efron, Student's t-test under symmetry conditions, {\it J. Amer. Statist. Assoc.},
    \textbf{64} (1969), pp.~1278--1302.

\bibitem{Esseen42} C.-G. Esseen, On the Liapunoff limit of error in the theory of probability, {\it Arkiv Mat. Astronom. Fys.}, \textbf{A28} (1942), pp.~1--19.

\bibitem{Esseen56} C.-G. Esseen, A moment inequality with an application to the central limit theorem, {\it Skand.
	Aktuarietidskr.}, \textbf{39} (1956), pp.~160--170.

\bibitem{F+14} Y. Filmus, H. Hatami, S. Heilman, E. Mossel, R. O'Donnell, S. Sachdeva, A. Wan,
and K. Wimmer, Real analysis in computer science: A collection of open problems, manuscript, 2014.
Available at https://simons.berkeley.edu/sites/default/files/openprobsmerged.pdf.

\bibitem{GZ14} F. G\"{o}tze and A. Yu. Zaitsev, Explicit rates of approximation in the {CLT} for quadratic forms, {\it Ann. Probab.}, \textbf{42(1)} (2014), pp.~354--397.

\bibitem{Guy86} R. K. Guy, Any answers anent these analytical enigmas?, {\it Amer. Math. Monthly},
    \textbf{93(4)} (1986), pp.~279--281.

\bibitem{HW04} P. Hall and Q. Wang, Exact convergence rate and leading term in {C}entral {L}imit {T}heorem for {S}tudent's $t$ statistic, {\it Ann. Probab.}, \textbf{32(2)} (2004), pp.~1419--1437.

\bibitem{HZ17a} H. Hendriks and M. C. A. van Zuijlen, Linear combinations of Rademacher random variables, 2017. Available at: https://arxiv.org/pdf/1703.07251.pdf


\bibitem{Hiriart09} J.-B. Hiriart-Urruty, A new series of conjectures and open questions in
    optimization and matrix analysis, {\it ESAIM: Control, Optimisation and Calculus of Variations},
    \textbf{15(2)} (2009), pp.~454--470.

\bibitem{HK94} P. Hitczenko and S. Kwapie\'{n}, On the Rademacher series, in: Proceedings of 9th Conference on
    Probability in Banach spaces, Birkhauser, 1994, pp.~31--36.

\bibitem{HK92} R. Holzman and D. J. Kleitman, On the product of sign vectors and unit vectors,
{\it Combinatorica}, \textbf{12(3)} (1992), pp.~303--316.

\bibitem{Kah94} J. P. Kahane, Some Random Series of Functions, Cambridge University Press, 1993.
	
\bibitem{KK17} N. Keller and O. Klein, Biased halfspaces, noise sensitivity, and relative Chernoff inequalities,
    {\it Discrete Analysis}, \textbf{2019:13} (2019), pp.~1--50.
	
	
\bibitem{KR20} H. K\"{o}nig and M. Rudelson, On the volume of non-central sections of a cube, {\it Adv. Math.}, \textbf{360} (2020), pp.~106929.

\bibitem{KS10} V. Yu. Korolev and I. G. Shevtsova, On the upper bound for the absolute constant in the Berry-Esseen inequality, {\it Theory Probab. Appl.}, \textbf{54(4)} (2010), pp.~638--658.

\bibitem{MS18} L. Mattner and J. Schulz, On normal approximations to symmetric hypergeometric laws, {\it Trans. Amer. Math. Soc.} \textbf{370(1)} (2018), pp.~727--748.

\bibitem{MS90} S. Montgomery-Smith, The distribution of Rademacher sums, {\it Proc. Amer. Math. Soc.},
    \textbf{109} (1990), pp.~517--522.
	
\bibitem{Oles96} K. Oleszkiewicz, On the Stein property of Rademacher sequences, {\it Probab. Math. Statist.},
    \textbf{16} (1996), pp.~127--130.

\bibitem{Pinelis94} I. Pinelis, Extremal probabilistic problems and Hotelling's $T^2$ test under
a symmetry condition, {\it Ann. Statist.}, \textbf{22} (1994), pp.~357--368.

\bibitem{Pinelis12} I. Pinelis, An asymptotically Gaussian bound on the Rademacher tails, {\it Electron. J. Probab.},
    \textbf{17} (2012), pp.~1--22.

\bibitem{pra72}  H. Prawitz, Limits for a distribution, if the characteristic function is given in a finite domain,	
    {\it Skand. Aktuarietidskr.}, \textbf{1972(2)} (1972), pp.~138--154.

\bibitem{Shevtsova10} I. G. Shevtsova. Refinement of estimates for the rate of convergence in Lyapunov’s theorem.
{\it Dokl. Akad. Nauk}, \textbf{435(1)} (2010), pp.~26--28.

\bibitem{Shnurnikov12} I. Shnurnikov, On a sum of centered random variables with nonreducing variances,
    manuscrupt, 2012. Available at arXiv:1202.2990v2.

\bibitem{So09} A.M.-C. So, Improved approximation bound for quadratic optimization problems with orthogonality
    constraints, proceedings of SODA 2009 conference, pp.~1201--1209.

\bibitem{Tan12} L.-Y. Tan, Analysis of Boolean functions -- lecture notes from a series of lectures by Ryan
O'Donnell, 2012. Available at https://arxiv.org/abs/1205.0314.

\bibitem{Toufar18} T. Toufar, Tomaszewski's conjecture, M.Sc. Thesis, Charles University, 2018.

\bibitem{vanZuijlen11} M. C. A. van Zuijlen, On a conjecture concerning the sum of independent Rademacher random variables, 2011. Available at arXiv:1112.4988v1.

\bibitem{Veraar08} M. Veraar, A note on optimal probability lower bounds for centered random variables,
    {\it Colloq. Math.}, \textbf{113} (2008), pp.~231--240.

\bibitem{Heymann12} F. von Heymann, Ideas for an old analytic enigma about the sphere that fail in intriguing ways,
    manuscript, 2012. Available at http://www.mi.uni-koeln.de/opt/wp-content/uploads/2017/02/Cube\_sphere.pdf.

	
\end{thebibliography}
\end{document}